\documentclass[11pt,a4paper]{article}
\usepackage[utf8]{inputenc}
\usepackage{epsfig}
\usepackage{amsmath}
\usepackage{amsthm}
\usepackage{amssymb}
\usepackage{bm}
\usepackage{amscd}
\usepackage{wrapfig}
\usepackage{subcaption}
\usepackage{float}
\usepackage{color,soul}
\usepackage{listings}
\usepackage{caption}
\usepackage{mathtools}
\usepackage{stmaryrd}
\usepackage[normalem]{ulem}
\usepackage{algorithm}
\usepackage{algorithmic}
\usepackage{tikz}

\usepackage{ifthen}

\theoremstyle{plain}
\newtheorem{theorem}{Theorem}
\newtheorem{lemma}[theorem]{Lemma}
\newtheorem{corollary}[theorem]{Corollary}
\newtheorem{proposition}[theorem]{Proposition}

\newtheorem{definition}{Definition}
\newtheorem{example}{Example}

\newtheorem{remark}{Remark}

\newcommand{\Reel}{\mathbb{R}}
\newcommand{\R}{\mathbb{R}}
\newcommand{\N}{\mathbb{N}}
\newcommand{\E}{\mathbb{E}}
\newcommand{\Natural}{\mathbb{N}}

\newcommand{\dd}{{\mathrm{d}}}

\DeclareMathOperator*{\argmin}{arg\,min}

\definecolor{applegreen}{rgb}{0.55, 0.71, 0.0}
\definecolor{green-yellow}{rgb}{0.68, 1.0, 0.18}
\definecolor{darkpastelgreen}{rgb}{0.01, 0.75, 0.24}
\definecolor{chartreuse(web)}{rgb}{0.5, 1.0, 0.0}

\newcommand{\Esp}[1]{\mathbb{E}\left( #1 \right)}
\newcommand{\Prob}[2][]{ \ifthenelse{\equal{#1}{}}
				    {\mathbb{P} \left( #2 \right)}
				    {\mathbb{P} \left( #2 \mid #1 \right)}
			  }
\newcommand{\Ind}[1]{\mathbf{1}_{#1}}

\newcommand{\cX}{\mathcal{X}}
\newcommand{\cY}{\mathcal{Y}}
\newcommand{\cP}{\mathcal{P}}
\newcommand{\cF}{\mathcal{F}}

\newcommand{\cK}{\mathcal{K}}
\newcommand{\cS}{\mathcal{S}}

\newcommand{\bN}{\mathbb{N}}
\newcommand{\bP}{\mathbb{P}}
\newcommand{\bE}{\mathbb{E}}

\numberwithin{equation}{subsection}
\numberwithin{figure}{subsection}
\numberwithin{algorithm}{section}
\numberwithin{theorem}{section}
\numberwithin{definition}{section}
\numberwithin{property}{section}
\numberwithin{remark}{section}
\numberwithin{example}{section}

\title{
	Approximation of Optimal Transport problems with marginal moments constraints
}

\date{\today}

\author{
	Aur\'elien Alfonsi \and Rafa\"el Coyaud \and Virginie Ehrlacher \and Damiano Lombardi 
}

\begin{document}

\maketitle

\noindent {\bf Abstract: } Optimal Transport (OT) problems arise in a wide range of applications, from physics to economics. Getting numerical approximate solution of these problems is a challenging issue of practical importance. 
In this work, we investigate the relaxation of the OT problem when the marginal constraints are replaced by some moment constraints. Using Tchakaloff's theorem, we show that the Moment Constrained Optimal 
Transport problem (MCOT) is achieved by a finite discrete measure. Interestingly, for multimarginal OT problems, the number of points weighted by this measure scales linearly with the number of marginal laws, which is encouraging to
bypass the curse of dimension. This approximation method is also relevant for Martingale OT problems. We show the convergence of the MCOT problem toward the corresponding OT problem. 
In some fundamental cases, we obtain rates of convergence in $O(1/n)$ or $O(1/n^2)$ where $n$ is the number of moments, which illustrates the role of the moment functions. Last, we present algorithms exploiting the fact that the MCOT is 
reached by a finite discrete measure and provide numerical examples of approximations. 

\section{Introduction}

The aim of this paper is to investigate a new direction to approximate optimal transport problems~\cite{villani2008optimal,santambrogio2015optimal}. Such problems arise in a variety of application fields ranging from economy~\cite{galichon2017survey,carlier2012optimal} to quantum chemistry~\cite{cotar2015infinite} or machine learning~\cite{peyre2019computational} for instance. 
The simplest prototypical example of optimal transport problem is the two-marginal Kantorovich problem, which reads as follows: for some $d\in \mathbb{N}^*$, let $\mu$ and $\nu$ be two probability measures on $\R^d$ and consider the optimization problem 
\begin{equation}\label{eq:two}
\inf \int_{\R^d\times \R^d} c(x,y) \,\dd \pi(x,y)
\end{equation}
where $c$ is a non-negative lower semi-continuous cost function defined on $\R^d\times \R^d$ and where the infimum runs on the set of probability measures $\pi$ on $\R^d\times \R^d$ with marginal laws $\mu$ and $\nu$. 

The most straightforward approach for the resolution of problems of the form (\ref{eq:two}) consists in introducing discretizations of the state spaces, which are fixed a priori. 
More precisely, $N$ points $x^1,\cdots, x^N \in \R^d$ are chosen a priori and fixed, marginal laws $\mu$ and $\nu$ are approximated by discrete measures of the form $\mu \approx \sum_{i=1}^N \mu_i \delta_{x^i}$ and $\nu \approx \sum_{i=1}^N \nu_i \delta_{x^i}$ 
with some non-negative coefficients $\mu_i$ and $\nu_i$ for $1\leq i \leq N$. An optimal measure $\pi$ minimizing (\ref{eq:two}) is then approximated by a discrete measure $\pi \approx \sum_{1\leq i,j \leq N} \pi_{ij} \delta_{x^i, x^j}$
where the non-negative coefficients $(\pi_{ij})_{1\leq i,j\leq N}\in \R_+^{N^2}$ are solution to the optimization problem
\begin{equation}\label{eq:disclin}
\inf \sum_{1\leq i,j\leq N} \pi_{ij} c(x^i,x^j)
\end{equation}
and satisfy the following discrete marginal constraints: 
$$
\forall 1\leq i,j \leq N,\quad \sum_{j=1}^N \pi_{ij} = \mu_i \quad \mbox{ and } \sum_{i=1}^N \pi_{ij} = \nu_j, 
$$
which boils down to a classical linear programming problem, which becomes computationally prohibitive when $N$ is large.

\medskip

Several numerical methods have already been proposed in the literature for the resolution of optimal transport problems at a lower computational cost. Most of them rely on an a priori discretization of the state spaces
as presented above. One of the most successful approach consists in minimizing a regularized cost involving the Kullback-Leibler divergence (or relative entropy)
via iterative Bregman projections: the so-called Sinkhorn algorithm~\cite{benamou2015iterative, nenna2016numerical,sharify2011solution}. Let us also mention other approaches such as the auction algorithm~\cite{bertsekas1989auction},
numerical methods based on Laguerre cells~\cite{gallouet2016lagrangian}, multiscale algorithms~\cite{merigot2011multiscale,schmitzer2016sparse} and augmented Lagrangian methods using the Benamou-Brenier dynamic 
formulation~\cite{benamou2000computational,benamou2015augmented}.  

\medskip

In this work, we are also interested in studying multi-marginal and martingale-constrained optimal transport problems. 

\medskip

Multimarginal optimal transport problems arise in a wide variety of contexts~\cite{villani2008optimal,santambrogio2015optimal}, 
like for instance the computation of Wasserstein barycenters~\cite{agueh2011barycenters} or the approximation of the correlation energy for strongly correlated systems 
in quantum chemistry~\cite{seidl2007strictly,colombo2015multimarginal,cotar2015infinite}. Such problems read as follows: let $M\in \N^*$ and $\mu_1,\cdots, \mu_M$ be $M$ probability measures on $\R^d$ and consider the optimization problem 
\begin{equation}\label{eq:multi}
\inf \int_{(\R^d)^M} c(x_1,\cdots,x_M) \,\dd \pi(x_1,\cdots,x_M)
\end{equation}
where $c$ is a lower semi-continuous cost function defined on $(\R^d)^M$ and where the infimum runs on the set of probability measures $\pi$ on $(\R^d)^M$ with marginal laws given by $\mu_1,\cdots,\mu_M$. Approximations of such multi-marginal problems 
on discrete state spaces can be introduced in a similar way to (\ref{eq:disclin}), leading to a linear programming problem of size $N^M$. For large values of $M$, such discretized problems become intractable numerically. 
The most successful method up to now for solving such problems, which avoids this curse of dimensionality, is based on an entropic regularization together with the Sinkhorn algorithm~\cite{benamou2015iterative, nenna2016numerical}.

\medskip

Constrained martingale transport arise in problems related to finance~\cite{beiglbock2017monotone}. Few numerical methods have been proposed so far for the resolution of such problems. In~\cite{ACJ1,ACJ2}, algorithms using sampling techniques preserving the convex order is proposed, which enables then to use linear programming solvers. Algorithms making use of an entropy regularization and the Sinkhorn algorithm have been studied in~\cite{de2018entropic,GuOb}.

\medskip

In this paper, we consider an alternative direction to approximate optimal transport problems, with a view to the design of numerical schemes. In this approach, the state space is \itshape not \normalfont discretized, but the 
approximation consists in relaxing the marginal laws constraints (or the martingale constraint) of the original problem into a finite number of moment constraints against some well-chosen test functions. More precisely, in the case of Problem~(\ref{eq:two}), 
let us introduce some real-valued bounded functions $\phi_1,\cdots, \phi_N$ defined on $\R^d$, which are called hereafter \itshape test functions\normalfont, and consider the approximate optimal transport problem, 
called hereafter the Moment Constrained Optimal Transport (MCOT) problem:
$$
\inf \int_{\R^d\times \R^d} c(x,y) \,d\pi(x,y)
$$
where the infimum runs over the set of probability measures $\pi$ on $\R^d\times \R^d$ satisfying for all $1\leq i,j \leq N$, 
$$
\int_{\R^d\times \R^d} \phi_i(x)\,\dd \pi(x,y) = \int_{\R^d} \phi_i(x)\,\dd \mu(x) \quad \mbox{ and }\int_{\R^d\times \R^d} \phi_j(y)\,\dd \pi(x,y) = \int_{\R^d} \phi_j(y)\,\dd \nu(y). 
$$
The aim of this paper is to study the properties of this alternative approximation of optimal transport problems, and its generalization for multi-marginal and martingale-constrained optimal transport problems. 
A remarkable feature of this approximation is that it circumvents the curse of dimensionality with respect to the number of marginal laws in the case of multimarginal optimal transport problems. 
Note that in the martingale constrained case, our method enables to consider the original formulation of the financial problem that has moment constraints (see for instance Example~2.6 of~\cite{henry2017model}), 
which is not the case of the previous methods. 

\medskip

Our first contribution in this paper is to characterize some minimizers of the MCOT problem. 
Thanks to the Tchakaloff theorem, we prove that, under suitable assumptions, there exists at least one minimizer which is a discrete measure charging a finite number of points. 
Interestingly, in the multi-marginal case, the number of charged points 
scales at most linearly in the number of marginals. In the particular case of problems issued from quantum chemistry applications, due to the inherent symmetries of the system, 
the number of charged points is independent from the number of marginals, and only scales linearly with the number of imposed moments. 
This formulation of the multimarginal optimal transport problem thus does not suffer from the curse of dimensionality. 
The result obtained in the quantum chemistry case is close in spirit to the one of~\cite{friesecke2018breaking} where the authors studied a multimarginal Kantorovich problem with Coulomb cost on finite state spaces. 

These considerations motivate us to consider a new family of algorithms for the resolution of multi-marginal and martingale constrained optimal transport problems, in which an optimal measure is approximated by a discrete measure charging 
a relatively low number of points. The points and weights of this discrete measure are then optimized in order to satisfy a finite number of moment constraints and to minimize the cost of the original optimal transport problem. 

\medskip

Of course, another interesting issue consists in determining how the choice of the particular test functions influences the quality of the approximation with respect to the exact optimal transport problem. 
In this paper, 
we prove interesting one-dimensional results in this direction. More precisely, for piecewise affine test functions defined on a compact interval, and for the two-marginal optimal transport problems involved in the computation of the $W_2$ or the $W_1$ distance 
between two sufficiently regular measures, the convergence of the approximate optimal cost with respect to the optimal cost scales like $\mathcal{O}\left( \frac{1}{N^2}\right)$ where $N$ is the number of test functions. 
These results indicate that the choice of appropriate test functions has an influence on the rate of convergence of the approximate problem to the exact problem. Besides, there is very few results, up to our knowledge, 
concerning the speed of convergence 
of approximations of optimal transport problems. The study of these rates of convergence for more general set of test functions and 
of optimal transport problems is an interesting issue which is left for future research.

\medskip

The article is organized as follows. Some preliminaries, including the Tchakaloff theorem, are recalled in Section~\ref{sec:prel}. In Section~\ref{sec:MCOT}, we introduce the approximate MCOT problem and prove under suitable assumptions 
that one of its minimizers reads as a discrete measure whose number of charged points is estimated depending on the number of moment constraints and on the nature of the optimal transport problem considered. Under additional assumptions, we prove that the MCOT problem converges to the exact optimal transport problem as the number of test functions grows in Section~\ref{sect:cvgMCOTtoOT}. Rates of convergence of the approximate problem to the exact problem depending on the choice 
of test functions are proved in Section~\ref{Sec:rateCV}. Finally, algorithms which exploits the particular structure of the MCOT problem are proposed in Section~\ref{sect:algoMkngUseOfPropTchack} and tested on some numerical examples.

\section{Preliminaries}\label{sec:prel}

\subsection{Presentation of the problem and notation}

We begin this section by recalling the classical form of the 2-marginal optimal transport (OT) problem, which will be the prototypical problem considered in this paper, and introduce the notation used in the sequel. 

\medskip

Let $d_x,d_y\in \bN^*$. We assume that $\cX \subset \Reel^{d_x}$ (resp. $\cY \subset \Reel^{d_y}$) is a $G_\delta$-set, i.e. a countable intersection of open sets. This property ensures by Alexandroff's lemma (see e.g.~\cite{Aliprantis}, p.~88) that $\cX$ (resp. $\cY$) is a Polish space for a metric which is equivalent to the original one on $\R^{d_x}$ (resp. $\R^{d_y}$). In particular, $\cX$ can either be a closed or an open set of  $\R^{d_x}$.

\medskip

Let $\mu \in \cP(\cX)$ and $\nu \in \cP(\cY)$ and let us define 
$$
\Pi(\mu,\nu):= \left\{ \pi \in \cP(\cX \times \cY): \; \int_\cX \dd \pi(x,y) = \dd\nu(y), \; \int_\cY \dd\pi(x,y) = \dd\mu(x) \right\},
$$
the set of probability couplings between $\mu$ and $\nu$. We consider a non-negative cost function
$c: \cX\times \cY \to \Reel_+ \cup\{+\infty\}$, which we assume to be lower semi-continuous (l.s.c.).
The Kantorovich optimal transport (OT) problem with the two marginal laws $\mu$ and $\nu$ associated to the cost function $c$ is the following minimization problem: 
\begin{equation}\label{eqn:OTDef}
I=\inf \left\{ \int_{\cX \times \cY} c(x,y) \dd \pi(x, y) : \pi \in \Pi(\mu,\nu) \right\}.
\end{equation}
Under our assumptions, it is known (see e.g. Theorem 1.7 in~\cite{santambrogio2015optimal}) that there exists $\pi^*\in \Pi(\mu,\nu)$ such that $I=\int_{\cX \times \cY} c(x,y) \dd \pi^*(x, y) $. 
Problem (\ref{eqn:OTDef}) will be referred hereafter as the \itshape exact \normalfont OT problem, with respect to the \itshape approximate \normalfont problem which we define hereafter.

\medskip

The aim of this paper is to study a relaxation of Problem~(\ref{eqn:OTDef}) with a view to the design of numerical schemes to approximate the exact OT problem. 
More precisely, the \itshape approximate \normalfont problem considered in this paper consists in relaxing the marginal constraints
into a \itshape finite number of \normalfont moments constraints. Let $ M, N\in\mathbb{N}^*$ and  $(\phi_m)_{1\leq m \leq M} \subset L^1(\cX,\mu;\R)$ (respectively $(\psi_n)_{1\leq n \leq N} \subset L^1(\cY,\nu;\R)$) measurable real-valued functions that are integrable with respect to $\mu$ (resp. $\nu$).  The functions $(\phi_m)_{1\leq m \leq M}$ and $(\psi_n)_{1\leq n \leq N}$ will be called \itshape test functions \normalfont hereafter. We define for such families of functions
\begin{align}
  &\Pi(\mu,\nu;(\phi_m)_{1\leq m \leq M},(\psi_n)_{1\leq n \leq N}) := \Bigg\{ \pi \in \cP(\cX \times \cY) :  \label{coupl_mom} \\
  & \forall 1\le m \le M, 1\le n\le N,\notag \int_{\cX\times \cY} |\phi_m(x)|+|\psi_n(y)| \dd \pi(x,y) <\infty, \\
  &\int_{\cX\times \cY} \phi_m(x) \dd \pi(x,y) = \int_{\cX} \phi_m(x) \dd \mu(x) , \;  \int_{\cX\times \cY} \psi_n(y) \dd \pi(x,y) = \int_{\cX} \psi_n(y) \dd \mu(x)  \Bigg\}, \notag
\end{align}
which is the set of probability measures on $\cX \times \cY$ that have the same moments as $\mu$ and $\nu$ for the test functions. We are then interested in the moment constrained optimal transport (MCOT) problem, which we defined as the following minimization problem :
\begin{equation}\label{eqn:MCOTDef}
  I^{M,N}=\inf \left\{ \int_{\cX \times \cY} c(x,y) \dd \pi(x, y) : \pi \in \Pi(\mu,\nu;(\phi_m)_{1\leq m \leq M},(\psi_n)_{1\leq n \leq N}) \right\}.
\end{equation}
Since $\Pi(\mu,\nu) \subset \Pi(\mu,\nu;(\phi_m)_{1\leq m \leq M},(\psi_n)_{1\leq n \leq N})$, we clearly have $ I^{M,N} \le I$. In this paper, we are interested in the following question.
\begin{itemize}
\item Is the MCOT problem attained by some probability measure $\pi^* \in \Pi(\mu,\nu;(\phi_m)_{1\leq m \leq M},(\psi_n)_{1\leq n \leq N})$?
\item Under which assumptions does it hold: $I^{M,N}\to_{M,N\to +\infty} I$? Can the speed of convergence be estimated? 
\end{itemize}

For simplicity, we will assume that $M=N$ in the whole paper and we will denote for $1\le m,n\le N$:
\begin{equation}\label{eq:defmumnun}
  \mu_m := \int_\cX \phi_m \dd \mu \quad \text{and} \quad
  \nu_n := \int_\cY \psi_n \dd \nu.
\end{equation}
For all $x\in \cX$ (respectively for all $y\in \cY$), we define $\phi(x) := (\phi_1(x), ..., \phi_N(x)) \in \Reel^N$ (respectively $\psi(y):= (\psi_1(y), ..., \psi_N(y)) \in \Reel^N$) and 
$\Phi(x):=(1, \phi(x))\in \Reel^{N+1}$ (respectively $\Psi(y):= (1,\psi(y)) \in \Reel^{N+1}$).

\subsection{Tchakaloff's theorem}

In this section, we present a corollary of the Tchakaloff theorem which is the backbone of our results concerning the existence of a minimizer to the MCOT problem. 
A general version of the Tchakaloff theorem has been proved by Bayer and Teichmann~\cite{bayer2006proof} and Bisgaard~\cite{berschneider2012theorem}. The next proposition is rather an immediate consequence of Tchakaloff's theorem, see Corollary~2 in~\cite{bayer2006proof}. We recall first that 

\begin{proposition}\label{cor:Tchakaloff}
Let $\pi$ be a measure on $\R^d$ concentrated on a Borel set $A \in \cF$, i.e. $\pi(\R^d \setminus A) = 0$. 
Let $N_0\in \bN^*$ and $\Lambda : \R^d \to \Reel^{N_0}$ a Borel measurable map.
Assume that the first moments of $\Lambda \# \pi$ exist, i.e. 
\[
\int_{\Reel^{N_0}} \Vert u \Vert \dd \Lambda \# \pi(u) =\int_{\R^d} \Vert \Lambda(z) \Vert \dd \pi(z) < \infty,
\]
where $\|\cdot\|$ denotes the Euclidean norm of $\Reel^{N_0}$. Then, there exist an integer $1 \leq K \leq N_0$, points $z_1, ..., z_K \in A$ and weights $p_1, ..., p_K > 0$ such that
\[
\forall 1\leq i \leq N_0, \quad \int_{\R^d} \Lambda_i(z) \dd \pi(z) = \sum_{k = 1}^K p_k \Lambda_i(z_k),
\]
where $\Lambda_i$ denotes the $i$-th component of $\Lambda$.
\end{proposition}
We recall here that $\Lambda \# \pi$ is the push-forward of $\pi$ through $\Lambda$, and is defined as  $\Lambda \# \pi (A)=\pi(\Lambda^{-1}(A))$ for any Borel set $A\subset \R^{N_0} $. Let us note here that even if $\pi$ is a probability measure, we may have $\sum_{k = 1}^K p_k \not = 1$. In the sequel, we will apply this proposition to functions $\Lambda$ such that the constant~$1$ is a coordinate of $\Lambda$, which will ensure $\sum_{k = 1}^K p_k  = 1$.

Last, let us remark that the number of points $K$ needed may be significantly different from $N_0$. Lemma~\ref{lem:cvgP1_W2} gives, for any $N\in \N^*$, an example with $N_0=2N+1$ and $K=N+1$.

\subsection{An admissibility property}\label{sect:genProp}

A natural requirement for the MCOT Problem~\eqref{eqn:MCOTDef} to have a sense is to assume that it has finite value. This is precisely our definition of admissibility.

\begin{definition}[Admissibility]\label{propt:admissibility}
Let $\mu \in \cP(\cX), \ \nu \in \cP(\cY)$ and a l.s.c. cost function 
$c : \cX \times \cY \to \Reel_+ \cup \{ \infty \}$. Then, a set of test functions $\left( (\phi_m)_{1\leq m \leq N}, (\psi_n)_{1\leq n \leq N}\right) \in L^1(\cX,\mu;\R)^N \times L^1(\cY,\nu;\R)^N $ is said to be admissible for $(\mu,\nu,c)$ if
\begin{equation} \label{eq:finitecost}
  \exists \gamma \in \Pi(\mu,\nu;(\phi_m)_{1\leq m \leq M},(\psi_n)_{1\leq n \leq N}), \int_{\cX \times \cY} c(x,y) \dd \gamma(x, y) < \infty.
\end{equation}
\end{definition}

Thanks to Tchakaloff's theorem, the admissibility can be checked on finite probability measure as stated in the next Lemma.

\begin{lemma}\label{lem:admissibleProprtyCharact}
  Let $\mu \in \cP(\cX)$, $\nu \in \cP(\cY)$ and $c:\cX\times \cY \to \Reel_+\cup\{+\infty\}$ a l.s.c. function. A set $\left( (\phi_m)_{1\leq m \leq N}, (\psi_n)_{1\leq n \leq N}\right) \in L^1(\cX,\mu;\R)^N \times L^1(\cY,\nu;\R)^N$ is admissible for $(\mu,\nu,c)$ if, and only if, there exist weights $w_1,\dots,w_{2N+1}\ge 0$ and points $(x_1,y_1),\dots,(x_{2N+1},y_{2N+1}) \in \cX\times \cY$ such that
  $$\sum_{k=1}^{2N+1} w_k \delta_{(x_k,y_k)} \in \Pi(\mu,\nu;(\phi_m)_{1\leq m \leq N},(\psi_n)_{1\leq n \leq N}) \text{ and } \sum_{k=1}^{2N+1} w_k c(x_k,y_k) <\infty  .$$
In particular, if $c$ is finite valued (i.e.  $c:\cX\times \cY \to \Reel_+$), any set of test functions $\left( (\phi_m)_{1\leq m \leq N}, (\psi_n)_{1\leq n \leq N}\right) \in L^1(\cX,\mu;\R)^N \times L^1(\cY,\nu;\R)^N$ is admissible for $(\mu,\nu,c)$ in the sense of Definition~\ref{propt:admissibility}. 
\end{lemma}

\begin{proof}
  Let $\Lambda:\cX\times \cY \rightarrow \R^{2N+1}$ be defined by $\Lambda_m(x,y)=\phi_m(x)$ and $\Lambda_{m+N}(x,y)=\psi_m(y)$ for $m\in \{1,\dots,N\}$, $\Lambda_{2N+1}(x,y)=1$. Let $A=\{(x,y)\in \cX\times \cY: c(x,y)=+\infty\}$.
  Since the set of test function is admissible, there exists $\gamma\in \Pi(\mu,\nu;(\phi_m)_{1\leq m \leq N},(\psi_n)_{1\leq n \leq N})$ such that $\int_{\cX\times \cY} c(x,y) \dd \gamma(x,y)<\infty$. In particular, $\gamma(A)=0$. We can thus apply Proposition~\ref{cor:Tchakaloff}, which gives the implication. The reciprocal result is obvious.
  
Last, when $c$ is finite valued ($A=\emptyset$), any $\gamma \in \Pi(\mu,\nu;(\phi_m)_{1\leq m \leq N},(\psi_n)_{1\leq n \leq N})$ satisfies $\gamma(A)=0$ and the claim follows by using again Proposition~\ref{cor:Tchakaloff}. 
\end{proof}

\section{Existence of discrete minimizers for MCOT problems}\label{sec:MCOT}

\subsection{General case}

When Definition \ref{propt:admissibility} is satisfied, in order to have the existence of a minimizer for the MCOT problem, we make two further assumptions.
\begin{itemize}
\item We assume that the test function are continuous.
\item We add to the MCOT problem \eqref{eqn:MCOTDef}
  a moment inequality constraint.
\end{itemize}
 The additional moment constraint will ensure the tightness of
a minimizing sequence satisfying the moment equality and inequality constraints, while the continuity of the test functions will ensure that any limit  satisfies the moment constraints. Our main result is stated in Theorem~\ref{prop:approxProblmDiscreteMeasGeneral} thereafter.

\begin{theorem}\label{prop:approxProblmDiscreteMeasGeneral}
Let $\mu \in \cP(\cX)$, $\nu \in \cP(\cY)$ and $c: \cX \times \cY\to \Reel_+ \cup \{ +\infty\}$ a l.s.c. function. Let $\Sigma_\mu\subset \cX ,\Sigma_\nu\subset \cY$ be Borel sets such that $\mu(\Sigma_\mu)=\nu(\Sigma_\nu)=1$. Let $N\in \mathbb{N}^*$ and 
let $\left( (\phi_m)_{1\leq m \leq N}, (\psi_n)_{1\leq n \leq N}\right) \in L^1(\cX,\mu;\R)^N \times L^1(\cY,\nu;\R)^N$ be an admissible set of test functions for $(\mu,\nu,c)$ in the 
sense of Definition~\ref{propt:admissibility}. We assume that
\begin{enumerate}
\item For all $n\in \{1,\dots,N\}$, the functions $\phi_n$ and $\psi_n$ are continuous.
\item There exist $\theta_\mu:\Reel_+ \to \Reel_+$ and $\theta_\nu: \Reel_+ \to \Reel_+$ two non-negative non-decreasing continuous functions such that $\displaystyle \theta_\mu(r)\mathop{\longrightarrow}_{r\to +\infty} +\infty$ 
  and $\displaystyle \theta_\nu(r) \mathop{\longrightarrow}_{r\to +\infty} +\infty$, and such that
there exist $C >0$ and $0 < s < 1$ such that for all $1\leq n\leq N$, and all $(x,y)\in \cX\times \cY$,
\begin{equation}\label{eq:ass1}
|\phi_n(x)| \leq C(1 + \theta_\mu(|x|))^s \quad 
	\mbox{ and }\quad |\psi_n(y)| \leq C(1 + \theta_\nu(|y|))^s.
\end{equation}
  
\end{enumerate}For all $A>0$, let us introduce
\begin{equation}\label{eqn:approxProblemGeneral}
I^N_{A} = \inf_{
			\substack{ 	\pi \in \Pi(\mu,\nu;(\phi_m)_{1\leq m \leq N},(\psi_n)_{1\leq n \leq N})
			 \\
				\int_{\cX \times \cY} (\theta_\mu(|x|) + \theta_\nu(|y|))\dd \pi(x,y) \leq A
			}
		}
			\int_{\cX \times \cY} c(x, y) \dd \pi( x, y). 
\end{equation}
Then, there exists $A_0>0$ such that for all $A \geq A_0$, $I^N_{ A}$ is finite and is a minimum. Moreover, for all $A \ge A_0$, there exists a minimizer 
$\pi^N_A$ for Problem~\eqref{eqn:approxProblemGeneral} such that $\pi^{N}_{A} = \sum_{k = 1}^{K} p_k \delta_{x_k, y_k}$, 
for some $0 < K \leq 2 N + 2$, with $p_k\geq 0$, $x_k \in \Sigma_\mu$ and $y_k \in \Sigma_\nu$ for all $1\leq k \leq K$.
\end{theorem}

\begin{remark}\label{rk:Ifinite} When $I$ defined by~\eqref{eqn:OTDef} is finite and $$A'_0=\int_{\cX } \theta_\mu(|x|) \dd \mu(x) + \int_{\cY} \theta_\nu(|y|) \dd \nu(y)<\infty, $$ we have for all $A\ge A'_0$, $I^N_{A}\le I <\infty$.
\end{remark}

\begin{remark}\label{Rk:prop31}
  When the functions $\phi_m$ and $\psi_n$ are bounded continuous (which holds automatically when $\cX$ and $\cY$ are compact), Assumption~\eqref{eq:ass1} is obviously satisfied. Besides, when $\cX$ and $\cY$ are compact sets, we can then take $C=\max_{1\le n\le N}(\max(\|\phi_n\|_\infty,\|\psi_n\|_\infty))$ and $\theta_\mu=\theta_\nu=0$, and therefore we get for all $A>0$, $I^N_A=I^N$ with
  $$I^N_{A} = \inf_{			\substack{ 	\pi \in \Pi(\mu,\nu;(\phi_m)_{1\leq m \leq N},(\psi_n)_{1\leq n \leq N})			
			}
		}
  \int_{\cX \times \cY} c(x, y) \dd \pi( x, y).   $$
\end{remark}

\begin{proof}[Proof of Theorem~\ref{prop:approxProblmDiscreteMeasGeneral}]
Let us introduce the function
  \begin{equation}\label{eqn:proofPropGenMeasPhi}
      \Lambda : \left\{
      \begin{array}{ccc} \cX \times \cY & \to & \Reel^{2N +2} \\
        (x,y) & \mapsto & \begin{pmatrix}
          \phi(x) \\
          \psi(y) \\
          1 \\
          c(x,y)
        \end{pmatrix}\\
    \end{array} 
    \right.
  \end{equation}
and let us denote by $\Lambda_i$ the $i^{th}$ component of $\Lambda$ for all $1 \leq i \leq 2N + 2$. By assumption there exists $\gamma \in \Pi(\mu,\nu;(\phi_m)_{1\leq m \leq N},(\psi_n)_{1\leq n \leq N})$ such that $\int_{\cX \times \cY} c(x,y) \dd \gamma(x, y) < \infty$. We apply Proposition~\ref{cor:Tchakaloff} with $N_0 = 2 N + 2$ and get that there exist $K\in\{1,\dots,2N+2\}$, $x_1, ..., x_K \in \cX$, $y_1, ..., y_K \in \cY$ and weights $w_1, ..., w_K \in \Reel^*_+$ 
such that
  \begin{equation}\label{eqn:PhiMomentConserved}
    \int_{\cX \times \cY} \Lambda(x,y) \dd \gamma(x,y) = \sum_{k=1}^K w_k \Lambda(x_k, y_k).
  \end{equation}
Denoting by $\widetilde{\gamma}:= \sum_{k=1}^K w_k \delta_{x_k, y_k}$, we have that
	\[
		\int_{\cX \times \cY} \left( \theta_\mu(|x|) + \theta_\nu(|y|) \right) \dd \tilde{\gamma}(x, y) < \infty.
	\]
We thus get that, for all $A\geq A_0:= \int_{\cX \times \cY} (\theta_\mu(|x|) + \theta_\nu(|y|) ) \dd \tilde{\gamma}(x,y)$, $I^N_{A}$ is finite, since we have $\tilde{\gamma}\in  \Pi(\mu,\nu;(\phi_m)_{1\leq m \leq N},(\psi_n)_{1\leq n \leq N})$.

\medskip
Let us now assume that $A\geq A_0$ and 
let us prove that this infimum is a minimum. Let $(\pi_l)_{l \in \Natural}$ be a minimizing sequence for the minimization problem
	\eqref{eqn:approxProblemGeneral}. We first prove the tightness of this sequence.
	For $M_1, M_2 > 0$, let us introduce the compact sets
	\begin{align*}
		\cK_1 = \left\{ x \in \cX, \, \text{s.t.} \, | x | \leq M_1 \right\}, \ \cK_2 = \left\{ y \in \cY, \, \text{s.t.} \, | y | \leq M_2 \right\}.
	\end{align*}	
	Then, we have
	\begin{equation*}
 		\begin{split}
 		\pi_l((\cK_1 \times \cK_2)^c) &= \int_{\cX \times \cY} \mathbf{1}_{(x,y)\not \in \cK_1 \times \cK_2}\dd \pi_l(x,y) 
\le  \int_{\cX \times \cY} \mathbf{1}_{x \not \in \cK_1 }+\mathbf{1}_{y \not \in \cK_2}\dd \pi_l(x,y)		\\
			&\le  \int_{\cX \times \cY} \frac{\theta_\mu(|x|)}{\theta_\mu(M_1)} + \frac{\theta_\nu(|y|)}{\theta_\nu(M_2)}\dd \pi_l(x,y) \le \frac{A}{\min(\theta_\mu(M_1),\theta_\nu(M_2))}, 		
 		\end{split}
 	\end{equation*}
	which implies the tightness of the sequence $(\pi_l)_{l \in \Natural}$. We can thus extract a subsequence that weakly converges. For notational simplicity, we still denote $(\pi_l)_{l \in \Natural}$ this subsequence, and there exists 
    $\pi_\infty \in \cP(\cX \times \cY)$ such that $\pi_l \xrightharpoonup[l \to \infty]{} \pi_\infty$.

    By Skorokhod's representation theorem (see e.g. Theorem 4.30~\cite{Kallenberg}), there exists a probability space
	$(\Omega, \mathcal{F}, \bP)$ and  random variables $(X_l, Y_l)_{l \in \N \cup \{\infty\}}$ on this probability space such that
    $(X_l, Y_l)$  is distributed according to $\pi_l$ and  $(X_l, Y_l) \to (X_\infty, Y_\infty)$, $\bP$-a.s. From $\sup_{l \in \N } \E[\theta_\mu(|X_l|)+\theta_\nu(|Y_l|) ]\le A$ and~\eqref{eq:ass1}, we deduce that the families $(\phi_m(X_l))_{l\in \N}$ and  $(\psi_n(X_l))_{l\in \N}$ uniformly integrable. Therefore, we get from the continuity of $\phi_m$ and $\psi_n$ that $\E[\phi_m(X_\infty)]=\lim_{l\rightarrow  \infty}\E[\phi_m(X_l)]=\mu_m$ and $\E[\psi_n(X_\infty)]=\lim_{l\rightarrow  \infty}\E[\psi_n(X_l)]=\nu_n$  (see e.g. Lemma 4.11~\cite{Kallenberg}). Fatou's lemma also gives
    $\E[\theta_\mu(|X_\infty|)+\theta_\nu(|Y_\infty|) ]\le A$, which shows that $\pi_\infty$ satisfies the constraints of Problem~\eqref{eqn:approxProblemGeneral} and thus
    $$I^N_A\le \int_{\cX \times \cY} c(x,y) \dd \pi_\infty(x, y).$$

    We now show the converse inequality. Since $c$ is l.s.c, $\liminf_{l \to \infty} c(X_l, Y_l) \geq c(X_\infty, Y_\infty)$, $\bP\mathrm{-a.s.}$ and	Fatou's lemma yields that	
	\begin{equation}\label{eqn:proofPropGenMeasFatouCsq}
		\liminf_{l \to \infty} \bE\left[c(X_l, Y_l)\right]  \geq \bE\left[c(X_\infty, Y_\infty)\right].
	\end{equation}
As $(\pi_l)_{l\in \bN}$ is a minimizing sequence associated to Problem~\eqref{eqn:approxProblemGeneral}, it holds that
$\bE\left[c(X_l, Y_l)\right] \xrightarrow[l \to \infty]{} I^N_{A}$.
 Then, by using \eqref{eqn:proofPropGenMeasFatouCsq}, one gets
\begin{equation*}
I^N_{A} \geq \bE\left[c(\tilde{X}_\infty, \tilde{Y}_\infty)\right] = \int_{\cX \times \cY} c(x,y) \dd \pi_\infty(x, y).
\end{equation*}
Thus, $\pi_\infty$ is a minimizer of Problem~\eqref{eqn:approxProblemGeneral}.

Last, we apply Proposition~\ref{cor:Tchakaloff} to the measure $\pi_\infty$ and the application
$\Lambda$ defined in~\eqref{eqn:proofPropGenMeasPhi} and get the existence of $\pi^N_A$.
\end{proof}

Example \ref{exmpl:necessityOfInequalityMomentConstraint} below shows that the MCOT problem may not be a minimum if 
we remove the constraint $\int_{\cX \times \cY} (\theta_\mu(|x|) + \theta_\nu(|y|))\dd \pi(x,y) \leq A$.

\begin{example}\label{exmpl:necessityOfInequalityMomentConstraint}
	Let
	\begin{equation*}
		c: \left\{\begin{array}{ccl}
			\Reel \times \Reel & \to & \Reel_+ \\
			(x,y) & \mapsto & (x-y)^2 + \varphi(|x|) + \varphi(|y|),
		\end{array}\right.
	\end{equation*}
	where for $r \in \Reel_+$, $\varphi(r)=\mathbf{1}_{ 0 \leq r \leq 1} r + \mathbf{1}_{ 1< r} e^{1-r}$. 
	Let us consider the problem
	\begin{equation*}
		I = \inf_{\substack{
			\pi \in \cP(\Reel \times \Reel) \\
			\int_\Reel x \dd \pi(x, y) = 1 \\
			\int_\Reel y \dd \pi(x, y) = 1
		}}
		\left\{
			\int_{\Reel \times \Reel} c(x,y) \dd \pi(x,y)
		\right\}.
	\end{equation*}
	The sequence defined for $l \in \Natural^*$ by $\pi_l =  \left( 1 - \frac{1}{l} \right) \delta_{(0,0)}+ \frac{1}{l} \delta_{(l,l)}$
	is a minimizing sequence since $\int_{\Reel \times \Reel} x \dd \pi_l(x, y)=\int_{\Reel \times \Reel} y \dd \pi_l(x, y)= 1$, $c\ge 0$ 	and 
	\begin{equation*}
		\int_{\Reel \times \Reel} c(x,y) \dd \pi_l(x, y) = \frac{2}{l} e^{1-l}
		\xrightarrow[l \to \infty]{} 0.
	\end{equation*}

	Hence, $I = 0$. Now, since $\varphi(r)>0$ for $r>0$, the only probability measure $\pi\in \cP(\R \times \R)$ such that $\int c \dd \pi =0$ is $\delta_{(0,0)}$. Since this probability measure does not satisfy the constraints ($\int_{\Reel \times \Reel} x \dd \delta_{(0,0)}(x, y) =\int_{\Reel \times \Reel} y \dd \delta_{(0,0)}(x, y) = 0$),  this shows that $I$ is not a minimum.
\end{example}

Let us also note here that the test functions $(\phi_m)_{1\leq m \leq N}$ and $(\psi_n)_{1\leq n \leq N}$ are needed to be continuous
to guarantee the existence of a minimum in Theorem~\ref{prop:approxProblmDiscreteMeasGeneral}
as Example \ref{exmpl:necessityOfContinuity} shows.

\begin{example}\label{exmpl:necessityOfContinuity}
 Let $\cX = \cY =[0,1]$,  $\,\dd\nu(x) = \left(\frac{1}{2} \Ind{(0, \frac{1}{2})}(x) + \frac{3}{2} \Ind{(\frac{1}{2}, 1)}(x)\right)\,\dd x$, $\,\dd \mu(x) = \,\dd x$
   and $c(x,y)=(y-x)^2$.
Let $N=4$, $\phi_1 = \Ind{[0,\frac{1}{4}]}$, $\phi_m = \Ind{(\frac{m-1}{4}, \frac{m}{4}]}$ for $2 \le m \le 4$ and $\psi_m=\phi_m$ for $1\le m \le 4$, so that 
$$
   \mu_1 = \mu_2 = \mu_3 = \mu_4 = \frac{1}{4}, \;  \nu_1 = \nu_2 = \frac{1}{8} \mbox{ and } \nu_3 = \nu_4 = \frac{3}{8}.
$$
For $l \in \Natural$, $l >4$, let
  \begin{equation}
    \gamma_l = \frac{1}{8} \delta_{\frac{1}{8}, \frac{1}{8}}
      + \frac{1}{8} \delta_{\frac{1}{4} - \frac{1}{l}, \frac{1}{4} + \frac{1}{l}}
      + \frac{1}{4} \delta_{\frac{1}{2} - \frac{1}{l}, \frac{1}{2} + \frac{1}{l}}
      + \frac{1}{8} \delta_{\frac{5}{8}, \frac{5}{8}}
      + \frac{1}{8} \delta_{\frac{3}{4} - \frac{1}{l}, \frac{3}{4}+ \frac{1}{l}}
      + \frac{1}{4} \delta_{\frac{7}{8}, \frac{7}{8}}.
  \end{equation}
  For all $l > 4$, $\gamma_l$ satisfies the constraints of the MCOT problem,
  and
  \begin{equation*}
    \int_0^1 \int_0^1 |x-y|^2 \dd \gamma_l(x,y) =
      \left( \frac{1}{8} + \frac{1}{4} + \frac{1}{8} \right) \frac{4}{l^2}
      = \frac{2}{ l^2} \xrightarrow[l \to + \infty]{} 0.
  \end{equation*}
  Thus, the infimum value of the associated MCOT problem is $0$. Now, let $\pi \in \cP(\cX\times \cY)$ be such that $\int c \dd \pi=0$. We have $\pi(\{(x,y)\in \cX\times \cY : y=x \})=1$ and thus
  $$\forall m, \ \int_{\cX\times \cY} \phi_m(x)\, \dd {\pi}(x,y)=\int_{\cX\times \cY} \phi_m(y)\, \dd {\pi}(x,y).$$
Therefore, we cannot have the left hand side equal to $\mu_m$ and the right hand side equal to $\nu_m$, which shows that there does not exist any minimizer to the MCOT problem.
\end{example}

\subsection{Compactly supported test functions}\label{sect:refinOfGenProp}

An alternative statement of Theorem~\ref{prop:approxProblmDiscreteMeasGeneral} that avoids to impose the constraint $\int_{\cX \times \cY} (\theta_\mu(|x|) + \theta_\nu(|y|))\dd \pi(x,y) \leq A$ can be obtained under stronger assumptions on
the test functions and the cost.  In all Subsection~\ref{sect:refinOfGenProp}, we consider  the case $$\cX=\R^{d_x} \text{ and } \cY=\R^{d_y},$$
for some $d_x,d_y \in \N^*$, and assume that the cost $c$ is continuous and satisfies:
\begin{align}\label{cond_infty}&\forall x\in \cX, \ c(x,y)\underset{|y|\to +\infty} \longrightarrow +\infty, \ \forall y\in \cY, c(x,y) \underset{|x|\to +\infty} \longrightarrow +\infty, \\
 & \exists (x_n) \in \cX^\N , (y_n) \in \cY^\N, |x_n|\rightarrow +\infty, |y_n|\rightarrow +\infty \text{ and } c(x_n,y_n)=0.\label{cond_infty2}
\end{align}
This condition is satisfied for example when $d_x=d_y$ and $c(x,y)=H(|x-y|)$, with $H$ continuous satisfying $H(0)=0$ and $H(r)\underset{r \to + \infty}\to +\infty$. 
We assume also that the test functions 
$\phi_m, \psi_n$, $1 \leq m,n \leq N$ are continuous with compact support, and define their compact support as follows
\begin{align*}
\cS_\cX & :=\overline{\left\{ x\in\cX, \quad \exists 1\leq m \leq N, \; \phi_m(x)\neq 0\right\}},\\
\cS_\cY & :=\overline{\left\{ y\in\cY, \quad \exists 1\leq n \leq N, \; \psi_n(y)\neq 0\right\}}.\\
\end{align*}
Let $M = \max_{x, y \in \cS_\cX \times \cS_\cY}c(x,y)$ and let us define
\begin{align}
  \widetilde{\cS}_\cX &= \left\{ x \in \cX :  \exists y \in \cS_\cY, \,  c(x,y) \leq M + 1 \right\} \\
  \widetilde{\cS}_\cY &= \left\{ y \in \cY :  \exists x \in \cS_\cX, \, c(x,y) \leq M + 1 \right\}
\end{align}
together with
\begin{equation*}
  \cK = \left( \cS_\cX \times \tilde{\cS}_\cY \right)
    \cup \left(\tilde{\cS}_\cX \times \cS_\cY \right).
\end{equation*}
It can be easily seen that $\tilde{\cS}_\cX$ (resp. $\tilde{\cS}_\cY$) is a compact set that contains $\cS_\cX$ (resp. ${\cS}_\cY$), and thus the set $\cK$ is compact.
Then, from~\eqref{cond_infty2}, we take an arbitrary point $(\bar{x},\bar{y}) \notin \cK $ such that $c(\bar{x},\bar{y})=0$, and we define
\begin{equation}\label{eq:defK0}
\bar{\cK} = \cK \cup \{ (\bar{x}, \bar{y}) \}.
\end{equation}

\begin{lemma}\label{lem:refinCvxCostCmctMeasExistence}
Let $K\in \bN^*$, and for all $1\leq k \leq K$, $x_k \in \cX$, $y_k\in \cY$, $p_k\geq 0$ such that $\sum_{k=1}^K p_k = 1$. 
  If $\gamma = \sum_{k=1}^K p_k \delta_{x_k, y_k} \in  \Pi(\mu,\nu;(\phi_m)_{1\leq m \leq N},(\psi_n)_{1\leq n \leq N}) $  then there exist $K$ points $(\tilde{x}_k, \tilde{y}_k) \in \bar{\cK}$ for $1\leq k \leq K$
  such that the discrete probability measure
  $\tilde{\gamma} = \sum_{k=1}^K p_k \delta_{\tilde{x}_k, \tilde{y}_k}\in  \Pi(\mu,\nu;(\phi_m)_{1\leq m \leq N},(\psi_n)_{1\leq n \leq N}) $
  and 
  \begin{equation*}
    \sum_{k=1}^K p_k c(\tilde{x}_k, \tilde{y}_k) =
    \int_{\cX \times \cY} c(x,y) \dd \tilde{\gamma}(x, y)
    \leq \int_{\cX \times \cY} c(x,y) \dd \gamma(x, y)
    = \sum_{k=1}^K p_k c(x_k, y_k).
  \end{equation*}
\end{lemma}

\begin{proof}
  Consider a measure $\gamma = \sum_{k=1}^K p_k \delta_{x_k, y_k}$ satisfying the assumptions of Lemma~\ref{lem:refinCvxCostCmctMeasExistence}.
  We construct $\tilde{\gamma} = \sum_{k=1}^K p_k \delta_{\tilde{x}_k, \tilde{y}_k}$
  using the following procedure.
  \begin{itemize}
  \item[\textbf{Case 1:}] If $(x_k, y_k) \in \cK$, then we define $(\tilde{x}_k, \tilde{y}_k) = (x_k, y_k)$.

  \item[\textbf{Case 2:}] If $x_k \notin {\cS}_\cX$ and $y_k \notin {\cS}_\cY$,
  then we define $(\tilde{x}_k, \tilde{y}_k) = (\bar{x}, \bar{y})$.

  \item[\textbf{Case 3:}] Let us suppose $x_k \in \cS_\cX$ and $y_k \notin \tilde{\cS}_\cY$ (the case $y_k\in \cS_\cY$ and $x_k \not \in  \tilde{\cS}_\cX$ is treated symmetrically).
  By definition of $\tilde{\cS}_\cY$, it holds that
  \begin{equation*}
    \forall x \in \cS_\cX, \quad c(x,y_k) > M + 1.
  \end{equation*}
  In particular, we have $c(x_k,y_k) > M + 1$.
  Let $y^* \in \cS_\cY$. Then,
  \begin{equation*}
    c(x_k,y^*) \le \max_{x,y \in \cS_\cX \times \cS_\cY} c(x,y) = M
  \end{equation*}
 Let $y_\lambda := (1- \lambda)y^* + \lambda y_k$ for $\lambda\in[0,1]$.
  As $c$ is continuous, there exists $\lambda^*$ such that
  $c(x_k,y_{\lambda^*}) = \frac{2M + 1}{2}$.
  Then, $y_{\lambda^*} \notin \cS_\cY$ because $\frac{2M + 1}{2} > M$,
  and $y_{\lambda^*} \in \widetilde{\cS}_\cY$. Then, we define $(\tilde{x}_k, \tilde{y}_k) = (x_k, y_{\lambda^*} )$.
  \end{itemize}
  This construction preserves the points in the supports $\cS_\cX$ and $\cS_{\cY}$, and the points outside the supports are replaced by other points outside the supports. Thus, we have
  \begin{align*}
    \forall 1  \le m \le N, \quad \sum_{k=1}^K p_k \phi_m(\tilde{x}_k) &= \sum_{k=1}^N p_k \phi_m(x_k) \\
    \forall 1  \le n \le N, \quad \sum_{k=1}^K p_k \psi_n(\tilde{y}_k) &= \sum_{k=1}^N p_k \psi_n(y_k),
  \end{align*}
  and the moment constraints are satisfied by~$\tilde{\gamma}$. 
 In addition, it is clear that the cost does not change in Case~1 and is lowered in
  Cases 2 and 3.
\end{proof}

\begin{proposition}\label{prop:refinCvxCostCmpctSuppTestFn}
Let us assume that $\cX = \Reel^{d_x} $, $ \cY = \Reel^{d_y}$ and   $c: \Reel^{d_x}\times  \Reel^{d_y} \rightarrow \R_+$ is continuous and satisfies~\eqref{cond_infty},~\eqref{cond_infty2}.

Let us assume that for all $1\leq m,n\leq N$, $\phi_m$ and $\psi_n$ are compactly supported real-valued continuous functions defined on $\Reel^d$. 
 Then, there exists at least one minimizer to the minimization problem
  \begin{equation}\label{eqn:approxProblemRefinementCompactSupport}
  	I^N = \inf_{\substack{
  \pi \in	\Pi(\mu,\nu;(\phi_m)_{1\leq m \leq N},(\psi_n)_{1\leq n \leq N})
  	}}
  		\int_{\cX \times \cY} c(x,y) \dd \pi(x, y).
  \end{equation}
  Moreover, there exists $K\in \bN$ such that $K \leq 2N + 2$, and for all $1\leq k \leq K$, $(x_k,y_k)\in \bar{\cK}$, $p_k\geq 0$ such that $\sum_{k=1}^K p_k = 1$ such that 
  $\widetilde{\pi}:= \sum_{k=1}^K p_k \delta_{x_k, y_k}$ is a minimum.
\end{proposition}

\begin{proof}
	Let us consider a minimizing sequence $(\pi_l)_{l \in \Natural}$ for 
	Problem~\eqref{eqn:approxProblemRefinementCompactSupport}.
	For all $l \in \Natural$, we will denote by $\gamma_l$ a
  finite discrete measure which has the same cost and same moments than
	$\pi_l$, with at most $2N + 2$ points, which exists thanks to Proposition \ref{cor:Tchakaloff}, and the fact that the test functions are compactly supported.
  Then, using Lemma \ref{lem:refinCvxCostCmctMeasExistence},
  for all $l \in \Natural$, one can define a measure $\tilde{\gamma}_l$
  which satisfies the moment constraints,
  has a support contained in the set
  $\bar{\cK}$ defined in~\eqref{eq:defK0}, and such that,
  \begin{equation*}
    \int_{\cX \times \cY} c(x,y) \dd \tilde{\gamma}_l(x, y)
    \leq \int_{\cX \times \cY} c(x,y) \dd \gamma_l( x, y).
  \end{equation*}
  Thus, $(\tilde{\gamma}_l)_{l \in \Natural}$ is a minimizing sequence. Besides, $(\tilde{\gamma}_l)_{l \in \Natural}$  is tight since the support of $\widetilde{\gamma}_l$ is included in the compact set $\bar{\cK}$ for all $l\in \bN$.
	Then, following the same lines as in the proof of Theorem~\ref{prop:approxProblmDiscreteMeasGeneral},
  one can extract a weakly converging subsequence,
	the cost of the limit of which is equal to $I^N$.
	The fact that there exists a finite discrete measure charging at most
	$K \leq 2N+2$ points can be deduced from Proposition \ref{cor:Tchakaloff}, following the same lines as in the proof of Theorem~\ref{prop:approxProblmDiscreteMeasGeneral}.
\end{proof}

\subsection{Multimarginal and martingale OT problem}

In this section, two important extensions of the previous problem are introduced,
the multimarginal problem and the martingale problem.
Alike Problem~\eqref{eqn:approxProblemGeneral},
several formulations and refinements can be established. We only keep here the
more general ones for conciseness.

\subsubsection{Multimarginal problem}\label{sect:multimargCase}

The propositions introduced until now for two marginal laws can be extended
to an arbitrary (finite) number of marginal laws. The proof can be straightforwardly
adapted from the one of Theorem~\ref{prop:approxProblmDiscreteMeasGeneral}.
 For all $1 \le i \le M$,  we consider $\cX_i = \Reel^{d_i}$ with $d_i \in \N^*$ or more generally a $G_\delta$-set $\cX_i\subset \R^{d_i}$. We consider $M$ probability measures $\mu_1 \in \cP(\cX_1), ..., \mu_M \in \cP(\cX_M)$ and an l.s.c. cost function $c: \cX_1 \times ... \times \cX_M \to \Reel_+ \cup \{ \infty \}$. We consider the following multimarginal optimal transport problem
\begin{equation}\label{eq:multimarg}
I=	\inf_{
		\pi \in \Pi(\mu_1, ..., \mu_M) \\
	}
	\left\{
		\int_{\cX_1 \times ... \times \cX_M}
		c(x_1, ..., x_M)
		\dd \pi(x_1, ..., x_M)
	\right\},
\end{equation}
where
$\Pi(\mu_1, ..., \mu_M) =
\{ \pi \in \cP(\cX_1 \times ... \times \cX_M) \, \text{s.t.}
\forall 1 \le i \le M, \, \int_{\cX_i} \dd \pi = \dd \mu_i
\}$.

In order to build the moments constrained optimal transport problem,
we introduce, for each $i$, $N_i\in \N^*$ test functions $(\phi^i_{n})_{1\le n\le N_i} \in L^1(\cX_i,\mu_i;\R)^{N_i}$.
We say that this set of test functions is admissible for $(\mu_1,\dots,\mu_M,c)$ if there exists $\gamma \in  \cP(\cX_1 \times ... \times \cX_M)$ such that
$$\forall i\in \{1,\dots,M\}, \forall n\in \{1,\dots,N_i\}, \ \int_{\cX_1 \times ... \times \cX_M}  \phi^i_n(x_i) \dd\gamma(x_1,\dots,x_M)=\int_{\cX_i}  \phi^i_n(x) \dd\mu_i(x)$$
and $\int_{\cX_1 \times ... \times \cX_M} c(x_1,\dots,x_M)\dd \gamma(x_1,\dots,x_M) <\infty$. 
We can now state the analogous of  Theorem~\ref{prop:approxProblmDiscreteMeasGeneral} for the multimarginal case.

\begin{proposition}\label{prop:multimarginal}
  For $i\in \{1,\dots,M\}$, let $\mu_i \in \cP(\cX_i)$ and $\Sigma_{\mu_i}\subset \cX_i$ a Borel set such that $\mu_i(\Sigma_{\mu_i})=1$. We assume that $c: \cX_1 \times ... \times \cX_M \to \Reel^+ \cup \{ \infty \}$ is a l.s.c. cost function, and that the set of test functions $\phi^i_n\in  L^1(\cX_i,\mu_i;\R)$ for $i\in \{1,\dots,M\}$ and $n\in\{1,\dots,N_i\}$ is admissible for $(\mu_1,\dots,\mu_M,c)$. We make the following assumptions.
  \begin{enumerate}
  \item For all $i$ and $n$, the function  $\phi^i_n$ is continuous.
  \item For all $i$, there exists $\theta_i:\R_+\rightarrow \R_+$ a non-decreasing continuous function such that $\theta_i(r)\underset{r\to +\infty}\to +\infty$  and such that
there exist $C >0$ and $0 < s < 1$ such that for all $1\leq n\leq N_i$, we have
\begin{equation}
\forall x\in \cX_i , \ |\phi^i_n(x)| \leq C(1 + \theta_i(|x|))^s  .
\end{equation}
    \end{enumerate}
We note $\mathbf{N}=(N_1,\dots,N_M)$, $\cX=\cX_1\times\dots\times\cX_M$ and consider the following problem
	\begin{equation}\label{eq:MCOTmulti}
		I^{\mathbf{N}}_{ A} =
		\inf_{
			\substack{
				\pi \in \cP(\cX) \\
				\forall i, n, \,
				\int_{\cX}
				\phi^i_{n}(x_i) \dd \pi(x_1, ..., x_M) \\
				= \int_{\cX_i} \phi^i_{n} (x) \dd \mu_i( x) \\
				\int_{\cX} \sum_{i=1}^M \theta_i(|x_i|) \dd \pi(x_1, ..., x_M) \leq A
			}
		}
		\left\{
			\int_{\cX} c(x_1, ..., x_M) \dd \pi( x_1, ... x_M)
		\right\}.
	\end{equation}
Then, there exists $A_0>0$ such that for all $A \geq A_0$, $I^{\mathbf{N}}_{ A}$ is finite and is a minimum. Moreover, for all $A \ge A_0$, there exists a minimizer 
$\pi^{\mathbf{N}}_A$ for the problem \eqref{eqn:approxProblemGeneral} such that $\pi^{\mathbf{N}}_{A} = \sum_{k = 1}^{K} p_k \delta_{x^k_1,\dots,x^k_M}$, 
for some $0 < K \leq \sum_{i=1}^M N_i + 2$, with $p_k\geq 0$ and  $x^k_i \in \Sigma_{\mu_i}$ for all $1\le i\le M$ and $1\leq k \leq K$.
\end{proposition}
An interesting point to remark in Proposition~\ref{prop:multimarginal} is that the number of weighted points of the discrete measure $\pi^{\mathbf{N}}_{ A}$ is linear with respect to the number of moment constraints. In particular, if we take the same number of moments $N_i=N$ for each marginal, the number of weighted points is equal to $2+MN$ and thus grows linearly with respect to~$M$. Since each points has $dM$ coordinates, the dimension of the discrete measure is in $O(M^2)$. For this reason, the development of algorithms for minimizing $\pi^{\mathbf{N}}_{ A}$ by using finite discrete measures may be a way to avoid the curse of dimensionality when $M$ is getting large. 

\medskip

We make here a specific focus on the multimarginal optimal transport problem which arises in quantum chemistry applications~\cite{seidl2007strictly,cotar2015infinite}. In this particular case, 
the multi-marginal optimal transport of interest reads as~(\ref{eq:multimarg}), with $\cX_1 = \cdots \cX_M = \R^3$, $N_1 = \cdots= N_M = N$ for some $N\in \N^*$, $\mu_1 = \cdots  = \mu_M = \mu$ for some $\mu \in \cP(\R^3)$ 
and $c$ is given by the Coulomb cost
$$c(x_1,\cdots, x_M):= \sum_{1\leq i < j \leq M} \frac{1}{|x_i -x_j|}.$$
The integer $M$ represents here the number of electrons in the system of interest. The inherent symmetries of the system yield interesting 
consequences on the MCOT problem~(\ref{eq:MCOTmulti}), which are summarized in the following proposition.

\begin{proposition}\label{prop:quantum}
Let $M\in \N^*$, $N\in \N^*$, $\mu \in \cP(\cX)$ and $\Sigma_{\mu}\subset \cX$ a Borel set such that $\mu(\Sigma_{\mu})=1$. We assume that $c: \cX^M \to \Reel^+ \cup \{ \infty \}$ is a symmetric l.s.c. cost function. More precisely, we denote 
by $\mathcal{S}_M$ the set of permutations of the set $\{1,\cdots,M\}$ and assume that
$$\forall \sigma \in \mathcal{S}_M, \quad c(x_{\sigma(1)}, \cdots, x_{\sigma(M)}) = c(x_1,\cdots,x_M), \quad \mbox{ for almost all }x_1,\cdots, x_M\in \cX.$$
For all $1\leq n \leq N$, let $\phi_n\in  L^1(\cX,\mu;\R)$. We define $\phi_n^i:= \phi_n$ for all $1\leq i \leq M$ and assume the set of test functions $\phi_n^i$ for $n\in\{1,\dots ,N\}$ and $i\in \{1,\cdots,M\}$ 
is admissible for $(\mu,\dots,\mu,c)$. We make the following assumptions.
  \begin{enumerate}
  \item For all $n$, the function  $\phi_n$ is continuous.
  \item There exists $\theta:\R_+\rightarrow \R_+$ a non-decreasing continuous function such that $\theta(r)\underset{r\to +\infty}{\mathop{\longrightarrow}} +\infty$  and such that
there exist $C >0$ and $0 < s < 1$ such that for all $1\leq n\leq N$, we have
\begin{equation}
\forall x\in \cX , \ |\phi_n(x)| \leq C(1 + \theta(|x|))^s  .
\end{equation}
    \end{enumerate}
We consider the following problem
	\begin{equation}\label{eq:MCOTmultisym}
		I^{N}_{ A} =
		\inf_{
			\substack{
				\pi \in \cP(\cX^M) \\
				\forall n,i, \,
				\int_{\cX^M} \phi_{n}(x_i) \dd \pi(x_1, ..., x_M) \\
				= \int_{\cX} \phi_{n} (x) \dd \mu( x) \\
				\int_{\cX^M} \sum_{i=1}^M \theta(|x_i|) \dd \pi(x_1, ..., x_M) \leq A
			}
		}
		\left\{
			\int_{\cX} c(x_1, ..., x_M) \dd \pi( x_1, ... x_M)
		\right\}.
	\end{equation}
Then, it holds that
\begin{equation}\label{eq:MCOTmultisym22}
I^{N}_{ A} =
		\inf_{
			\substack{
				\pi \in \cP(\cX^M) \\
				\forall n, \,
				\int_{\cX^M} \left(\frac{1}{M}\sum_{i=1}^M \phi_{n}(x_i)\right) \dd \pi(x_1, ..., x_M) \\
				= \int_{\cX} \phi_{n} (x) \dd \mu( x) \\
				\int_{\cX^M} \sum_{i=1}^M \theta(|x_i|) \dd \pi(x_1, ..., x_M) \leq A
			}
		}
		\left\{
			\int_{\cX} c(x_1, ..., x_M) \dd \pi( x_1, ... x_M)
		\right\},
\end{equation}
and there exists $A_0>0$ such that for all $A \geq A_0$, $I^{N}_{ A}$ is finite and is a minimum. Moreover, for all $A \ge A_0$, there exists a minimizer $\pi^{N}_A$ for the problem~\eqref{eq:MCOTmultisym22} such that $\pi^{N}_{A} = \sum_{k = 1}^{K} p_k \delta_{x^k_1,\dots,x^k_M}$, 
for some $0 < K \leq N + 2$, with $p_k\geq 0$ and  $x^k_i \in \Sigma_{\mu}$ for all $1\le i\le M$ and $1\leq k \leq K$. Besides, the symmetric measure
\begin{equation}\label{eq:defpisym}
\pi_{\rm sym,A}^{N} := \frac{1}{M!}\sum_{\sigma \in \mathcal{S}_M}\sum_{k = 1}^{K} p_k \delta_{x^k_{\sigma(1)},\dots,x^k_{\sigma(M)}}
\end{equation}
is a minimizer to (\ref{eq:MCOTmultisym}) and (\ref{eq:MCOTmultisym22}).
\end{proposition}

\begin{proof}
It is obvious that the right hand side of~\eqref{eq:MCOTmultisym22} is smaller than the right hand side of~\eqref{eq:MCOTmultisym}. By using the same arguments as in the proof of Theorem~\ref{prop:approxProblmDiscreteMeasGeneral}, there exists $A_0>0$ such that for all $A \geq A_0$ the infimum of~\eqref{eq:MCOTmultisym22} is finite, is a minimum that is attained by some discrete probability measure  $\pi^{N}_{A} = \sum_{k = 1}^{K} p_k \delta_{x^k_1,\dots,x^k_M}$, 
for some $0 < K \leq N + 2$ with  $x^k_i \in \Sigma_{\mu}$ for all $1\le i\le M$ and $1\leq k \leq K$. Then, since $c$ is symmetric and the set of constraints is also symmetric, we get that $\pi_{\rm sym,A}^{N}$ also realizes the minimum. Besides, it satisfies $\int_{\cX^M} \phi_{n}(x_i) \dd \pi_{\rm sym,A}^{N}(x_1, ..., x_M) = \int_{\cX} \phi_{n} (x) \dd \mu( x)$ for all $n,i$, which shows that it is also the minimizer of~\eqref{eq:MCOTmultisym}. 
\end{proof}

Proposition~\ref{prop:quantum} is particularly interesting for the design of numerical schemes for the resolution of multimarginal optimal transport problems with Coulomb cost arising in quantum chemistry applications. Indeed, 
the latter read as (\ref{eq:MCOTmultisym}) and the number of charged points of the minimizer $\pi^N_A$ of~\eqref{eq:MCOTmultisym22} only scales at most like $N+2$, and that the dimension of the optimal discrete measure is in $dM(N+2)$. This result states that such formulation of the multimarginal optimal transport problem does not suffer from the curse of dimensionality. Let us mention that this result is close in spirit to the recent work~\cite{friesecke2018breaking}, where multimarginal optimal transport problems with Coulomb cost are studied on finite state spaces.

\subsubsection{Martingale OT problem}\label{sect:martCase}

In this paragraph, we assume  $\cX = \cY = \Reel^d$ with $d \in \Natural^*$, and consider two probability measures $\mu,\nu \in \cP(\R^d)$ such that
$$\int_{\R^d} |y|\dd \nu(y)< \infty $$ and
$\mu$ is lower than $\nu$ for the convex order, i.e. \begin{equation}\int_{\R^d} \varphi(x)\dd \mu(x)\le \int_{\R^d} \varphi(y)\dd \nu(y),\label{cvx_order}
\end{equation}
for any convex function $\varphi:\R^d\rightarrow \R$ non-negative or integrable with respect to $\mu$ and $\nu$. This latter condition is equivalent, by Strassen's theorem~\cite{Strassen}, to the existence of a martingale coupling between $\mu$ and $\nu$, i.e.
$$\exists \pi \in \Pi(\mu,\nu), \  \forall x \in {\R^d}, \  \int_{\R^d} y \dd \pi(x, y) = x. $$
The original martingale optimal transport consists then in the minimization problem
\begin{equation}\label{OT_MG_PB}
\inf_{
		\substack{
			\pi \in \Pi(\mu,\nu) \\
			\forall x \in {\R^d}, \, \int_{\R^d} y \dd \pi(x, y) = x
		}}
		\left\{
			\int_{{\R^d} \times {\R^d}} c(x,y) \dd \pi(x, y)
		\right\},
\end{equation}
with $c:\R^d \times \R^d\rightarrow \R_+\cup\{\infty\}$ being a l.s.c. cost function. This problem has recently got a great attention in mathematical finance since the work of Beiglb\"ock et al.~\cite{BeHLPe}, because it is related to the calculation of model-independent option price bounds on an arbitrage free market.

We consider a set of test functions $(\phi_m)_{1\leq m \leq N}\in L^1({\R^d},\mu;\R)^N $ and $(\psi_n)_{1\leq n \leq N} \in  L^1({\R^d},\nu;\R)^N$, and the following problem:
$$I^{N}=\inf_{		\substack{
			\pi \in \Pi(\mu,\nu;(\phi_m)_{1\leq m \leq N}, (\psi_n)_{1\leq n \leq N} ) \\
			\forall x \in {\R^d}, \, \int_{\R^d} y \dd \pi(x, y) = x
		}} 		\left\{
			\int_{{\R^d} \times {\R^d}} c(x,y) \dd \pi(x, y)
		    \right\}.$$
           Suppose for simplicity that there exist some minimizer to this problem~$\pi^*$. Then, by using Theorem~5.1 in Beiglb\"ock and Nutz~\cite{BeNu} that is an extension of Tchakaloff's theorem to the martingale case, there exists a probability measure $\tilde{\pi}^*$ weighting at most $(d+2N+2)^2$ points such that $\tilde{\pi}^*\in  \Pi(\mu,\nu;(\phi_m)_{1\leq m \leq N}, (\psi_n)_{1\leq n \leq N} )$, $$\forall x \in {\R^d}, \, \int_{\R^d} y \dd \tilde{\pi}^*(x, y) = x$$ and $$\int_{{\R^d} \times {\R^d}} c(x,y) \dd \tilde{\pi}^*(x, y)=\int_{{\R^d} \times {\R^d}} c(x,y) \dd \pi^*(x, y)=I^{N}.$$
However, the minimization problem for $I^N$ still has the martingale constraints. To get a problem that is similar to the MCOT, we then relax in addition the martingale constraint.  This constraint is equivalent to have
\begin{equation*}
	\int_{{\R^d} \times {\R^d}} f(x)(y-x) \dd \pi(x, y) = 0,
\end{equation*}
for all bounded measurable functions $f:\R^d\to \R$, and also for all function $f:\R^d\to \R$ such that $\int_{\R^d} |xf(x)|\dd \mu(x)<\infty$. Then, it is natural to consider $N'$ test functions
$\chi_l : \R^d \to \R$,  $1 \leq l \leq N'$, such that
\begin{equation}\label{integrabilite_chi}
	\int_{\R^d} | x \chi_l(x) | \dd \mu(x) < \infty,
\end{equation}
and then to consider the following minimization problem
\begin{equation}\label{MgINN'}
  I^{N,N'}=\inf_{		\substack{
			\pi \in \Pi(\mu,\nu;(\phi_m)_{1\leq m \leq N}, (\psi_n)_{1\leq n \leq N} ) \\
			\forall l, \, \int_{\R^d \times \R^d }y \chi_l(x) \dd \pi(x, y) = \int_{\R^d } x \chi_l(x) \dd \mu(x)
		}} 		\left\{
			\int_{\R^d \times \R^d} c(x,y) \dd \pi(x, y)
		    \right\}.
\end{equation}
We will say that the test functions $(\phi_m)_{1\leq m \leq N}$, $(\psi_n)_{1\leq n \leq N}$ and $(\chi_l)_{1\leq l \leq N'}$ are admissible for the martingale problem of $(\mu,\nu,c)$ if $I^{N,N'}<\infty$. Similarly to Theorem~\ref{prop:approxProblmDiscreteMeasGeneral}, we get the following result. 

\begin{proposition}\label{prop:generalUnboundedMartProblem}
Let $\mu \in \cP(\R^d)$, $\nu \in \cP(\R^d)$ and $c: \R^d \times \R^d \to \R_+ \cup \{ +\infty\}$ a l.s.c. function. Let $\Sigma_\mu,\Sigma_\nu\subset \R^d$ be Borel sets such that $\mu(\Sigma_\mu)=\nu(\Sigma_\nu)=1$. Let $N\in \mathbb{N}^*$ and 
let $ (\phi_m)_{1\leq m \leq N} \in L^1(\R^d,\mu;\R)^N$, $(\psi_n)_{1\leq n \leq N} \in L^1(\R^d,\nu;\R)^N$ and $(\chi_l)_{1\leq l \leq N'}$ satisfying~\eqref{integrabilite_chi} be an admissible set of test functions for the martingale problem of $(\mu,\nu,c)$. We make the following assumptions.
\begin{enumerate}
\item For all $n\in \{1,\dots,N\}$, $l\in \{1,\dots,N'\}$, the functions $\phi_n$, $\psi_n$ and $\chi_l$ are continuous.
\item There exist $\theta_\mu:\Reel_+ \to \Reel_+$ and $\theta_\nu: \Reel_+ \to \Reel_+$ two non-negative non-decreasing continuous functions such that $\displaystyle \theta_\mu(r)\mathop{\longrightarrow}_{r\to +\infty} +\infty$ 
  and $\displaystyle \theta_\nu(r) \mathop{\longrightarrow}_{r\to +\infty} +\infty$, and such that
there exist $C >0$ and $0 < s < 1$ such that for all $1\leq n\leq N$, $1\leq l\leq N'$, and all $(x,y)\in {\R^d}\times {\R^d}$,
\begin{equation}\label{eq:ass_mom}
|\phi_n(x)| +  |\psi_n(y)| + |y\chi_l(x)| \leq C(1 + \theta_\mu(|x|) + \theta_\nu(|y|))^s.
\end{equation}
 
\end{enumerate}
For all $A>0$, let us introduce
 \begin{equation}\label{eqn:approxProblemGeneral_MG}
   I^{N,N'}_{A}=\inf_{		\substack{
			\pi \in \Pi(\mu,\nu;(\phi_m)_{1\leq m \leq N}, (\psi_n)_{1\leq n \leq N} ) \\
			\forall l, \, \int_{\R^d \times \R^d }y \chi_l(x) \dd \pi(x, y) = \int_{\R^d } x \chi_l(x) \dd \mu(x)
	 \\
				\int_{{\R^d} \times {\R^d}} (\theta_\mu(|x|) + \theta_\nu(|y|))\dd \pi(x,y) \leq A	}} 		\left\{
			\int_{\R^d \times \R^d} c(x,y) \dd \pi(x, y)
		    \right\}.
\end{equation}
Then, there exists $A_0>0$ such that for all $A \geq A_0$, $I^{N,N'}_{ A}$ is finite and is a minimum. Moreover, for all $A \ge A_0$, there exists a minimizer 
$\pi^{N,N'}_A$ for Problem~\eqref{eqn:approxProblemGeneral_MG} such that $\pi^{N,N'}_A = \sum_{k = 1}^{K} p_k \delta_{x_k, y_k}$, 
for some $0 < K \leq 2 N +N' + 2$, with $p_k\geq 0$, $x_k \in \Sigma_\mu$ and $y_k \in \Sigma_\nu$ for all $1\leq k \leq K$.
\end{proposition}
The proof of Proposition~\ref{prop:generalUnboundedMartProblem} follows exactly the same line as the proof of Theorem~\ref{prop:approxProblmDiscreteMeasGeneral}, since the relaxation of the martingale moment constraints only brings new moment constraints. Let us stress that the minimizer $\pi^{N,N'}_A$ does not satisfy in general the martingale constraint. Also, we do not impose in Proposition~\ref{prop:generalUnboundedMartProblem} to have~\eqref{cvx_order}, i.e. $\mu$ smaller than $\nu$ for the convex order. In fact, the admissibility condition already ensures that $I^{N,N'}<\infty$ and thus, by using Proposition~\ref{cor:Tchakaloff} that $I^{N,N'}_A<\infty$ for $A$ large enough. Nonetheless, if we assume in addition that $\mu$ smaller than $\nu$ for the convex order and that $I$, the infimum of Problem~\ref{OT_MG_PB}, is finite, then we have $I^{N,N'}_A<\infty$ and $I^{N,N'}_A\le I$ for any $A\ge \int_{{\R^d} \times {\R^d}} (\theta_\mu(|x|) + \theta_\nu(|y|))\dd \pi^{1}(x,y)$, where $\pi^1\in\Pi(\mu,\nu)$ is such that $\int_{\R^d}y\dd \pi^1(x,y)=x$ and $\int_{\R^d \times \R^d} c(x,y) \dd \pi^1(x, y)\le I+1$.

\section{Convergence of the MCOT problem towards the OT problem}\label{sect:cvgMCOTtoOT}

The aim of this section is to prove that when the number of test functions $N \to +\infty$, the minimizer of the MCOT problem converges towards a minimizer of
the OT problem, under appropriate assumptions and up to the extraction of a subsequence.

\subsection{Continuous test functions on unbounded domains}

Let us consider two sequences of continuous real-valued test functions $(\phi_m)_{m \in \Natural^*}$ and $(\psi_n)_{n \in \Natural^*}$ defined on $\cX$ (resp. $\cY$) and make the following assumptions.

Let us first assume that there exist
continuous non-decreasing non-negative
functions $\theta_\mu : \Reel_+ \to \Reel_+ $ and $\theta_\nu : \Reel_+ \to \Reel_+$
such that
\begin{equation}\label{eqn:nonCmpctCvgHypOnTheta}
 \theta_\mu(|x|) \xrightarrow[|x| \to + \infty]{} +\infty \quad \text{and} \quad
 \theta_\nu(|y|) \xrightarrow[|y| \to + \infty]{} +\infty
\end{equation}
and
\begin{equation}\label{eqn:nonCmpctCvgHypFiniteMoment}
  \int_\cX \theta_\mu(|x|) \dd \mu(x) < \infty \quad \text{and} \quad
  \int_\cY \theta_\nu(|y|) \dd \nu(y) < \infty.
\end{equation}
In the sequel, we set \begin{equation}\label{def_A0}
  A_0:= \int_\cX \theta_\mu(|x|) \dd \mu(x) + \int_\cY \theta_\nu(|y|) \dd \nu(y).
\end{equation}
We assume moreover that there exist $(s^\mu_m)_{m\in \bN^*}, (s^\nu_n)_{n\in \bN^*} \in (0,1)^{\bN^*}$ and $(C^\mu_m)_{m\in \bN^*}, (C^\nu_n)_{n\in \bN^*} \in (\Reel_+^*)^{\bN^*}$ such that 
\begin{align}\label{eq:dom1}
  \forall m \in \Natural^*, \ \forall x \in \cX, \quad | \phi_m(x) | &\le C^\mu_m (1+ \theta_\mu(|x|))^{s^\mu_m}, \\\label{eq:dom2}
  \forall n \in \Natural^*, \ \forall y \in \cY, \quad | \psi_n(y) | &\le C^\nu_n (1 + \theta_\nu(|y|))^{s^\nu_n}.\\\nonumber
\end{align}
Last, we assume that the probability measures $\mu$ and $\nu$ are fully characterized by their moments:
\begin{align}
&  \forall \eta\in \cP(\cX), \left( \forall m\in \N^*, \int_{\cX}\phi_m(x)\dd \eta(x)=\mu_m \right) \implies \eta=\mu, \label{carmu} \\
 &  \forall \eta\in \cP(\cY), \left( \forall n\in \N^*, \int_{\cY}\psi_n(x)\dd \eta(x)=\nu_n \right) \implies \eta=\nu. \label{carnu}
\end{align}

We consider the optimal cost for the OT problem~\eqref{eqn:OTDef} that we restate here for convenience
\begin{equation}\label{eqn:nonCmpctCvgOTOptCost}
  I = \inf_{\pi \in \Pi(\mu, \nu)}
    \left\{ \int_{\cX \times \cY} c(x,y) \dd \pi(x, y) \right\},
\end{equation}
and for all $N\in \bN^*$, we define
the $N^{th}$ MCOT problem,
\begin{equation}\label{eqn:nonCmpctCvgApproxProbOptCost}
  I^N_{A_0} = \min_{
		\substack{
			\pi \in \Pi(\mu,\nu;(\phi_m)_{1\le m \le N},(\psi_n)_{1\le n \le N} )\\
      \int_{\cX \times \cY} (\theta_\mu(|x|) + \theta_\nu(|y|)) \dd \pi(x, y) \le A_0 
		}
	}
	\left\{
		\int_{\cX \times \cY} c(x,y) \dd \pi(x, y)
	\right\}.
\end{equation}

\begin{theorem}\label{prop:cvgApproxPbtoOTPb}
Let $\mu \in \cP(\cX)$ and $\nu \in \cP(\cY)$ satisfying (\ref{eqn:nonCmpctCvgHypFiniteMoment}) for some continuous non-decreasing functions $\theta_\mu: \Reel_+ \to \Reel_+$ and $\theta_\nu: \Reel_+ \to \Reel_+$ satisfying 
(\ref{eqn:nonCmpctCvgHypOnTheta}). Let $c: \cX \times \cY \to \Reel_+ \cup \{+\infty\}$ a l.s.c. function. Let $(\phi_m)_{m \in \bN^*} \subset L^1(\cX,\mu;\R) $ and $(\psi_n)_{n \in \bN^*}\subset L^1(\cY,\nu;\R)$ be continuous functions satisfying \eqref{eq:dom1}, \eqref{eq:dom2}, \eqref{carmu} and~\eqref{carnu}.
Let us finally assume that $I$, defined by (\ref{eqn:nonCmpctCvgOTOptCost}) is finite. 

Then, for all $N\in \bN^*$, there exist at least one minimizer for Problem (\ref{eqn:nonCmpctCvgApproxProbOptCost}) and 
$$
I^N_{A_0} \mathop{\longrightarrow}_{N\to +\infty} I. 
$$
Besides, from every sequence $(\pi^N)_{N \in \bN^*}$ such that
  for all $N$, $\pi^N\in \cP(\cX\times \cY)$ is a minimizer for~\eqref{eqn:nonCmpctCvgApproxProbOptCost},
  one can extract a subsequence which converges towards a minimizer $\pi^\infty\in \cP(\cX\times \cY)$ to problem (\ref{eqn:nonCmpctCvgOTOptCost}).
\end{theorem}

\begin{proof}
  
From Theorem~\ref{prop:approxProblmDiscreteMeasGeneral} and Remark~\ref{rk:Ifinite},  We know that there exists at least one minimizer $\pi^N\in \cP(\cX\times \cY)$ to \eqref{eqn:nonCmpctCvgApproxProbOptCost}. Since we have
$$\forall N,  \int_{\cX \times \cY} (\theta_\mu(|x|) + \theta_\nu(|y|)) \dd \pi^N(x, y) \le A_0,$$
and~\eqref{eqn:nonCmpctCvgHypOnTheta}, we get that the sequence $(\pi^N)_{N\in\bN^*}$ is tight. Thus, up to the extraction of a subsequence, still denoted $(\pi^N)_{N \in \bN^*}$ for the sake of simplicity,
  there exists a measure $\pi^\infty \in \cP(\cX \times \cY)$ such that
  $\pi^N \xrightharpoonup[N \to \infty]{} \pi^\infty$. With the same argument as in the proof of Theorem~\ref{prop:approxProblmDiscreteMeasGeneral}, we get that  for all $m,n \in \Natural^*$,
  \begin{equation*}
    \int_{\cX\times \cY} \phi_m(x) \dd \pi^\infty(x,y) = \mu_m \quad \text{and} \quad
    \int_{\cX\times \cY} \psi_n(x) \dd \pi^\infty(x,y) = \nu_n.
  \end{equation*}
  Then, Properties~\eqref{carmu} and~\eqref{carnu}  give $\pi^\infty \in \Pi(\mu, \nu)$.
  Therefore,
  \begin{equation}\label{eqn:cvgApproxProb1stIneg}
    \int_{\cX \times \cY} c(x,y) \dd \pi^\infty(x, y) \ge I.
  \end{equation}

  We now establish that $\int_{\cX \times \cY} c(x,y) \dd \pi^\infty(x, y) \le I$ which concludes the proof.
  The Skorokhod representation theorem states that there exist a space
  $(\Omega, \mathcal{F}, \bP)$
  and random variables $(X_N, Y_N)_{N\in \N\cup\{\infty\}}$ 
  such that  $(X_N, Y_N)$ is distributed according to $\pi^N$ and $(X_N, Y_N) \to (\tilde{X}, \tilde{Y})$
  $\bP$-a.s.

  Furthermore, as $c$ is l.s.c,
  \begin{equation*}
    \liminf_{N \to \infty} c(X_N, Y_N) \geq c(X_\infty, Y_\infty), \, \bP\mathrm{-a.s.}
  \end{equation*}
  and by Fatou's lemma we get
  $\liminf_{N \to \infty} \Esp{c({X}_N, {Y}_N)} \geq \Esp{c(X_\infty, Y_\infty)}$,
  i.e.
  \begin{equation}\label{eqn:cvgApproxProbFatouSkor}
    \liminf_{N \to \infty} \int_{\cX \times \cY} c(x,y) \dd \pi^N(x, y) \geq
      \int_{\cX \times \cY} c(x,y) \dd \pi^\infty( x, y),
  \end{equation}

  Furthermore, note that $(I^N)_{N \in \Natural}$ is a non-decreasing sequence
  and that for all $N \in \bN^*$, $I^N \le I$.
  Thus, there exists $I^\infty \le I$ such that
  $I^N \xrightarrow[N \to \infty]{} I^\infty$.
  Recall that
  \begin{equation*}
    \liminf_{N \to \infty} \int_{\cX \times \cY} c(x,y) \dd \pi^N(x, y)
    = \lim_{N \to \infty} \int_{\cX \times \cY} c(x,y) \dd \pi^N(x, y)
    = I^\infty,
  \end{equation*}
  then \eqref{eqn:cvgApproxProbFatouSkor} implies,
  \begin{equation}\label{eqn:cvgApproxProb2ndIneg}
    I \ge I^\infty \ge \int_{\cX \times \cY} c(x,y) \dd \pi^\infty(x, y).
  \end{equation}

  Using equations \eqref{eqn:cvgApproxProb1stIneg} and \eqref{eqn:cvgApproxProb2ndIneg}, one gets
  \begin{equation*}
    \int_{\cX \times \cY} c(x,y) \dd \pi^\infty(x, y) = I
    \quad \text{and} \quad
    I^\infty = I.
  \end{equation*}
  Thus, as $\pi^\infty \in \Pi(\mu, \nu)$,
  $\pi^\infty \in \argmin_{
  		\pi \in \Pi(\mu, \nu)
  	}
  	\left\{
  		\int_{\cX \times \cY} c(x,y) \dd \pi(x, y)
  	\right\}$,
  and $I^N \xrightarrow[N \to \infty]{}  I$.
\end{proof}

\subsection{Bounded test functions on compact sets}\label{sect:discTestFnCmpctSet}

We now assume that $\cX$ and $\cY$ are compact subsets of $\Reel^{d_x}$ and $\Reel^{d_y}$. We state a result analogous to Theorem~\ref{prop:cvgApproxPbtoOTPb} that holds without the additional moment constraint and for possibly discontinuous test functions. We consider  two sequences of bounded measurable real-valued test functions $(\phi_m)_{m \in \Natural^*}\subset L^\infty(\cX)$ and $(\psi_n)_{n \in \Natural^*} \subset L^\infty(\cY)$ that satisfy
\begin{equation}\label{eqn:CmpctCvgHypOnDiscTestFncSpanMu}
\forall f \in C^0(\cX), \mathop{\inf}_{v_N \in  \mathrm{Span}\left\{ \phi_m, \, 1\le m\le N \right\}} \| f - v_N\|_\infty \mathop{\longrightarrow}_{N\to +\infty} 0 
\end{equation}
and
\begin{equation}\label{eqn:CmpctCvgHypOnDiscTestFncSpanNu}
\forall f \in C^0(\cY),  \mathop{\inf}_{v_N \in  \mathrm{Span}\left\{ \psi_n, \, 1\le n\le N \right\}} \| f- v_N\|_\infty  \mathop{\longrightarrow}_{N\to +\infty} 0.
\end{equation}
It is easy then to see that the properties~\eqref{carmu} and~\eqref{carnu} are satisfied for any $\mu\in \cP(\cX)$ and $\nu \in \cP(\cY)$.
For any $N\ge 1$, we consider  the following MCOT problem:
\begin{equation}\label{eqn:nonCmpctCvgApproxProbOptCostDisc}
  I^N = \inf_{
		\substack{
			\pi \in \Pi(\mu,\nu;(\phi_m)_{1\le m \le N}, (\psi_n)_{1\le n \le N})		
	}}
	\left\{
		\int_{\cX \times \cY} c(x,y) \dd \pi(x, y)
	\right\}.
\end{equation}

\begin{proposition}\label{prop:cvgApproxPbtoOTPbDiscTestFn}
Let us assume that $\cX$ and $\cY$ are compact sets and let $\mu\in \cP(\cX)$ and $\nu \in \cP(\cY)$. 
Let $(\phi_m)_{m\in\bN^*}\subset L^\infty(\cX)$ and $(\psi_n)_{n\in\bN^*} \subset L^\infty(\cY)$ satisfying \eqref{eqn:CmpctCvgHypOnDiscTestFncSpanMu} and \eqref{eqn:CmpctCvgHypOnDiscTestFncSpanNu}.
Let us assume that $I<+\infty$. Then, it holds that $I^N\le I$ and
  \begin{equation*}
    I^N \xrightarrow[N \to \infty]{} I.
  \end{equation*}
Moreover, from every sequence $(\pi^N)_{N \in \Natural}$ such that
  for all $N\in \bN^*$, $\pi^N\in \Pi(\mu,\nu;(\phi_m)_{1\le m \le N}, (\psi_n)_{1\le n \le N}) $ satisfies 
\begin{equation}\label{eq:presque}
\int_{\cX\times \cY} c(x,y)\,d\pi^N(x,y) \leq I_N + \epsilon_N,
\end{equation}
with   $\epsilon_N\underset{n\to +\infty}\longrightarrow 0$,
one can extract a subsequence which converges towards a measure $\pi^\infty \in \cP(\cX\times \cY)$ which is a minimizer to Problem \eqref{eqn:nonCmpctCvgOTOptCost}.
\end{proposition}

\begin{remark} From Proposition~\ref{cor:Tchakaloff},  there exists $0\leq K_N \leq 2N+2$, $x_1,\cdots, x_{K_N}\in \cX$, $y_1, \cdots, y_{K_N}\in \cY$ and $w_1, \cdots, w_{K_N}\geq 0$ 
such that $\gamma^N:= \sum_{k=1}^{K_N}w_k \delta_{(x_k, y_k)} \in \Pi(\mu,\nu;(\phi_m)_{1\le m \le N}, (\psi_n)_{1\le n \le N})		$
and 
\begin{equation}\label{eq:presque_disc}
\int_{\cX\times \cY} c(x,y)\,d\gamma^N(x,y) = \int_{\cX\times \cY} c(x,y)\,d\pi^N(x,y) \leq I_N + \epsilon_N.
\end{equation}
In other words, any sequence $(\pi^N)_{N\in\bN^*}$ satisfying the assumptions of Proposition~\ref{prop:cvgApproxPbtoOTPbDiscTestFn} can be chosen as a discrete measure charging at most $2N+2$ points.
\end{remark}

\begin{proof}[Proof of Proposition \ref{prop:cvgApproxPbtoOTPbDiscTestFn}]
  Since $\cX$ and $\cY$ are compact, the sequence $(\pi^N)$ is tight and we can assume, up to the extraction of a subsequence, that it weakly converges to $\pi^\infty$.
For $N\in \N^*\cup \{\infty\}$, we denote the marginal laws of $\pi^N$ respectively
  by $\dd \mu^N(x) := \int_\cY \dd \pi^N(x, y)$ and
  $\dd \nu^N(y) := \int_\cX \dd \pi^N(x, y)$. For $f \in C^0(\cX)$, it holds that
  \begin{equation*}
    \int_\cX f \dd \mu^N \xrightarrow[N \to \infty]{} \int_\cX f \dd  \mu^\infty.
  \end{equation*}
  Let $\epsilon > 0$. Using the density condition \eqref{eqn:CmpctCvgHypOnDiscTestFncSpanMu}, one can find
  $M \in \bN^*$ and  $\lambda_1, ..., \lambda_M \in \Reel$
  such that $\sup_{x \in \cX} \left|f(x) - \sum_{i=1}^M \lambda_i \phi_i (x)\right| \le \epsilon$.
  Thus,
  \begin{equation}\label{eqn:cvgApproxPbtoOTDiscrProof1}
    \left| \int_\cX f \dd \mu - \sum_{i=1}^M  \lambda_i \mu_i \right| \le \epsilon
  \end{equation}
  and for $K > M$, $ \left| \int_\cX f \dd \mu^K - \sum_{i=1}^M  \lambda_i \int_\cX \phi_i \dd \mu^K \right| \le \epsilon$, 
  i.e.
  \begin{equation}\label{eqn:cvgApproxPbtoOTDiscrProof2}
    \left| \int_\cX f \dd \mu^K - \sum_{i=1}^M  \lambda_i \mu_i \right| \le \epsilon.
  \end{equation}
  Then, \eqref{eqn:cvgApproxPbtoOTDiscrProof1} and \eqref{eqn:cvgApproxPbtoOTDiscrProof2}
  imply that $    \left| \int_\cX f \dd \mu^K - \int_\cX f \dd \mu \right| \le 2 \epsilon$,
  and taking $K \to \infty$ leads to
  \begin{equation}\label{eqn:cvgApproxPbtoOTDiscrProof3}
    \left| \int_\cX f \dd \mu^\infty - \int_\cX f \dd \mu \right| \le 2 \epsilon.
  \end{equation}
  As \eqref{eqn:cvgApproxPbtoOTDiscrProof3} holds for any $\epsilon > 0$, one gets
  that for any $f \in C^0(\cX)$,
  \begin{equation*}
    \int_\cX f \dd \mu^\infty = \int_\cX f \dd \mu,
  \end{equation*}
which yields that $\mu^\infty = \mu$. Similarly, it holds that $\nu^\infty = \nu$. Therefore, $\pi^\infty \in \Pi(\mu, \nu)$
  and
  \begin{equation}
    \int_{\cX \times \cY} c(x,y) \dd \pi^\infty(x, y) \ge I.
  \end{equation}
Now, we use the same arguments as in the proof of Proposition~\ref{prop:cvgApproxPbtoOTPb} to get $\int_{\cX \times \cY} c(x,y) \dd \pi^\infty(x, y) \le I$, which gives the result.
\end{proof}

\subsection{Convergence for Martingale Optimal Transport problems}

In this subsection, we study  the convergence of $I^{N,N'}_A$ defined by~\eqref{MgINN'} when the number of test functions for the martingale condition $N'\rightarrow +\infty$ towards the following minimization problem:
\begin{equation}\label{MgINmg}
  I^{N,mg}_A=\inf_{		\substack{
			\pi \in \Pi(\mu,\nu;(\phi_m)_{1\leq m \leq N}, (\psi_n)_{1\leq n \leq N} ) \\
		\forall x \in \R^d, 	\int_{\R^d }y \dd \pi(x, y) = x\\
      \int_{\cX \times \cY} (\theta_\mu(|x|) + \theta_\nu(|y|)) \dd \pi(x, y) \le A_0  
		}} 		\left\{
			\int_{\R^d \times \R^d} c(x,y) \dd \pi(x, y)
		    \right\}.
\end{equation}
This convergence is particularly interesting for the practical application in finance: the marginal laws $\mu,\nu$ are in general not observed and market data 
only provide some moments. For $d=1$, market data give the prices of European put (or call) options that corresponds to $\phi_m(x)=(K_m-x)^+$ and $\psi_n(y)=(K'_m-y)^+$. Let us assume 
for simplicity zero interest rates. Then, by taking $\theta_\mu(|x|)=\theta_\nu(|x|)=|x|$, we have from the martingale assumption  $\int_{\cX \times \cY} (|x|+ |y|) \dd \pi(x, y) = 2 S_0$, where $S_0>0$ is the 
current price of the underlying asset. Then, a natural choice would be to take $A_0=2S_0$. Therefore, the convergence stated in Proposition~\ref{prop:CVMartProblem} gives a way to approximate option price bounds by taking into account that only some moments are known, while the few existing numerical methods for Martingale Optimal Transport in the literature assume that the marginal laws are known~\cite{ACJ1,ACJ2,GuOb}.

\begin{proposition}\label{prop:CVMartProblem}
  Let $\mu \in \cP(\R^d)$ lower than $\nu \in \cP(\R^d)$ for the convex order and $c: \R^d \times \R^d \to \R_+ \cup \{ +\infty\}$ a l.s.c. function. We assume $|x|\le \theta_\mu(|x|)$, $|y|\le \theta_\nu(|y|)$ and suppose  $A_0<\infty$ 
  with $A_0$ defined by~\eqref{def_A0}.  We assume that the test functions $(\chi_l,l\in \N^*)$ are bounded and such that for any function $f:\R^d\rightarrow \R$ continuous with compact support, we have
\begin{equation}\label{eqn:Cond_chi}
\mathop{\inf}_{g \in  \mathrm{Span}\left\{ \chi_l, \, 1\le l\le N' \right\}} \| f - g\|_\infty \mathop{\longrightarrow}_{N'\to +\infty} 0.
\end{equation}
 Let the assumptions of Proposition~\ref{prop:generalUnboundedMartProblem} hold for any $N'\ge 1$. Then, we have $I_{A_0}^{N,N'}\underset{N'\to +\infty}{\longrightarrow} I^{N,mg}_{A_0}<\infty$.
\end{proposition}
\begin{proof}
  Since $A_0<\infty$, any martingale coupling between $\mu$ and $\nu$ satisfies the constraints of $I_{A_0}^{N,N'}$. By using Tchakaloff's theorem and the fact that $c$ is finite-valued, we get that $I^{N,N'}_{A_0}$ is finite for any $N'$ and is attained by 
  a measure denoted by $\pi^{N'}$ according to Proposition~\ref{prop:generalUnboundedMartProblem}. 
  Similarly, using Tchakaloff's theorem for the martingale case, Theorem 5.1~\cite{BeNu}, we get that $I^{N,mg}_{A_0}<\infty$. Note that from the inclusion of the constraints, 
  we clearly have $I^{N,N'_1}_{A_0}\le I^{N,N'_2}_{A_0} \le I^{N,mg}_{A_0}$ for $N'_1\le N'_2$. We can then repeat the arguments in the proof of Theorem~\ref{prop:cvgApproxPbtoOTPb} to get that $(\pi^{N'})$ is tight 
  and any limit~$\pi^\infty$ of a weakly converging subsequence satisfies $I^{N,mg}_{A_0}=\int_{\R^d \times \R^d} c(x,y) \dd \pi^\infty(x, y)$.

  The only thing to prove is that $\int_{\R^d \times \R^d}(y-x) f(x) d\pi^\infty(x, y) =0$ for any function $f:\R^d\rightarrow \R$ continuous with compact support. Let $\epsilon>0$. By assumption, there exists $M\in \N^*$ and $\lambda_1,\dots,\lambda_M \in \R$ such that $\sup_{x \in \R^d} |f(x)-\sum_{l=1}^M \lambda_l \chi_l(x)|\le \epsilon$. Therefore, for $N'\ge M$, we have
  \begin{align*}
    \left|\int_{\R^d \times \R^d} f(x) (y-x) d\pi^{N'}(x, y)\right| &=\left|\int_{\R^d \times \R^d} \left(f(x)-\sum_{l=1}^M \lambda_l \chi_l(x)\right)(y-x) d\pi^{N'}(x, y)\right|\\
&    \le \epsilon \int_{\R^d \times \R^d}|y-x|d\pi^{N'}(x, y)\le \epsilon A_0,
  \end{align*}
 by  using  the triangle inequality and the fact that $|x|\le \theta_\mu(|x|)$, $|y|\le \theta_\nu(|y|)$. We conclude then easily letting $N'\rightarrow \infty$. 
\end{proof}

Let us mention that we can obtain using similar arguments that $I_{A_0}^{N,mg}$ and $I_{A_0}^{N,N'}$ converge towards (\ref{OT_MG_PB}) as $N$ and $N'$ go to infinity.

\section{Rates of convergence for particular sets of test functions}\label{Sec:rateCV}

Throughout this section, we assume that $$\cX = \cY = [0,1]$$
and for all $N\in \bN^*$, we define the intervals
\begin{equation}
T^N_{1} = \left[0 , \frac{1}{N}\right], \ \forall 2 \le m \le N, \, T^N_m = \left(\frac{m-1}{N}, \frac{m}{N}\right].
\end{equation}

We investigate in this section the rate of convergence of $I^N$ defined by~\eqref{eqn:nonCmpctCvgApproxProbOptCostDisc} towards $I$ defined by~\eqref{eqn:nonCmpctCvgOTOptCost}, when the test functions are piecewise constant (resp. piecewise linear) on $T^N_m$. We obtain, under suitable assumptions a convergence rate of $O(1/N)$ (resp. $O(1/N^2)$). This shows, as one may expect, the importance of the choice of test functions to approximate the Optimal Transport problem.

\subsection{Piecewise constant test functions on compact sets}\label{subsec_P0}

In this section, we assume that the cost function $c: [0,1]^2 \to \Reel_+$ is Lipschitz:
\begin{equation}\label{Lipcost}|c(x,y)-c(x',y')|\le K \max(|x-x'|,|y-y'|).
\end{equation}
We define, for $\pi \in \cP([0,1]^2)$, $I(\pi)=\int_{\cX\times \cY} c(x,y) \dd \pi(x,y)$ and 
\begin{equation}\label{eqn:P0OT}
  I = \inf_{ \pi \in \Pi(\mu, \nu) } I( \pi).
\end{equation}
We introduce the piecewise constant test functions $$\forall N\ge 1, 1\le m\le N, \ \ \phi^N_m = \Ind{T^N_m},$$ and consider the MCOT problem:
\begin{equation}\label{eqn:P0MCOT}
	I^N = \inf_{
		\substack{
			\pi \in \Pi(\mu,\nu;(\phi^N_m)_{1\le m \le N}, (\phi^N_n)_{1\le n \le N})
		}
	}
	\left\{
		\int_{[0,1]^2} c(x, y) \dd \pi(x,y)
	\right\}.
\end{equation}
Then, Proposition~\ref{prop:cvgP0MCOTtoOT} establishes the rate of convergence of the sequence $(I^N)_{N\in\bN^*}$ to $I$ as $N$ increases. 
problem $N$ increases.

\begin{proposition}\label{prop:cvgP0MCOTtoOT}
Let $\mu,\nu \in \cP([0,1])$ and $c : [0,1]^2 \to \Reel_+$ a Lipschitz function with Lipschitz constant $K>0$.
Then, for all $N\in \bN^*$,
\begin{equation}\label{eqn:cvgMCOTtoTOP0}
I^N \le I \le I^N + \frac{K}{N}.
\end{equation}
\end{proposition}

\begin{remark}
  Let us note that we are not exactly in the framework of Section~\ref{sect:cvgMCOTtoOT}, since the test functions depends on~$N$. However, we have 
$$\mathrm{Span}\left\{ \phi_m^{N}, \, 1\le m\le N \right\} \subset \mathrm{Span}\left\{ \phi_m^{2N}, \, 1\le m\le 2N \right\}$$ and thus
 Proposition~\ref{prop:cvgApproxPbtoOTPbDiscTestFn} gives for any $L\in\N^*$, 
$$
I^{L2^k} \mathop{\longrightarrow}_{k\to +\infty} I.
$$
\end{remark}

Before proving Proposition~\ref{prop:cvgP0MCOTtoOT}, we state a result that bounds the distance between an MCOT optimizer and the minimizer of the OT problem~\eqref{eqn:P0OT}. We define that for $p\ge 1$, the $W_p$-Wasserstein distance between $\eta_1,\eta_2 \in \cP(\R^d)$ as $W^p_p(\eta_1,\eta_2)=\inf_{\pi \in \Pi(\eta_1,\eta_2)} \int_{\R^d\times \R^d} \|x_1-x_2\|_p^p \dd \pi(x_1,x_2)$, i.e. we take the $\|\|_p$-norm for $W_p$.  
\begin{proposition}\label{prop_dist_optimum}
 Let $p>1$. Let $\mu \in \cP([0,1])$. If $\mu^N\in \cP([0,1])$ is such that $\int_0^1 \phi^N_m(x) \dd \mu^N(x)=\int_0^1 \phi^N_m(x) \dd \mu(x)$ for all $m\in \{1,\dots,N\}$, then
  $$W_p(\mu,\mu^N)\le \frac{1}{N}. $$
  Let us assume besides that the cost function satisfies $c(x,y)=H(y-x)$ with $H:\R \rightarrow \R_+$ strictly convex.  There exists then a unique minimizer of~\eqref{eqn:P0OT} that we denote $\pi^*$.\\ Let $\pi^N\in \Pi(\mu,\nu;(\phi^N_m)_{1\le m \le N}, (\phi^N_n)_{1\le n \le N})$, $\mu^N$ and $\nu^N$ the marginal laws of $\pi^N$ and assume that $$\int_{[0,1]^2}c(x,y) \dd \pi^N(x,y)=\min_{\pi \in \Pi(\mu^N,\nu^N)}\int_{[0,1]^2}c(x,y) \dd \pi(x,y).$$
  Then, we have  $W_p(\pi^N,\pi^*)\le \frac{2^{1/p}}{N},$ where $W_p$ is defined using the $\|\|_p$ norm on $\R^2$. 
\end{proposition}
\begin{proof}
  For $\eta \in  \cP(\R)$, we define $F_\eta^{-1}(u)=\inf\{x \in \R: \eta((-\infty,x])\ge u \}$, that coincides with the usual inverse when $F_\eta$ is increasing. Let $p>1$. By Theorem~2.9~\cite{santambrogio2015optimal}, we have
    \begin{align*}&W^p_p(\mu,\mu^N)=\int_0^1 |F_{\mu}^{-1}(u)-F_{\mu^N}^{-1}(u)|^pdu \\&
    = \int_{0}^{F_{\mu}\left(0 \right)} |F_{\mu}^{-1}(u)-F_{\mu^N}^{-1}(u)|^pdu + \sum_{m=1}^N \int_{F_{\mu}\left(\frac{m-1}{N} \right)}^{F_{\mu}\left(\frac{m}{N} \right)} |F_{\mu}^{-1}(u)-F_{\mu^N}^{-1}(u)|^pdu. 
    \end{align*}
    If $F_{\mu}\left(\frac{m}{N} \right)=F_{\mu}\left(\frac{m-1}{N} \right)$, we clearly have $\int_{F_{\mu}\left(\frac{m-1}{N} \right)}^{F_{\mu}\left(\frac{m}{N} \right)} |F_{\mu}^{-1}(u)-F_{\mu^N}^{-1}(u)|^pdu=0$. Otherwise, we have $F_{\mu^N}\left(\frac{m-1}{N} \right)= F_{\mu}\left(\frac{m-1}{N} \right)< F_{\mu}\left(\frac{m}{N} \right)=F_{\mu^N}\left(\frac{m}{N} \right)$, and therefore
    $$\forall u\in \left( F_{\mu}\left(\frac{m-1}{N} \right),  F_{\mu}\left(\frac{m}{N} \right)\right), \quad F_{\mu}^{-1}(u),F_{\mu^N}^{-1}(u)\in \left[\frac{m-1}{N},\frac{m}{N} \right].$$
    This gives $|F_{\mu}^{-1}(u)-F_{\mu^N}^{-1}(u)|\le 1/N$. Since $F_\mu(0)=F_{\mu^N}(0)$, we get that $F_{\mu}^{-1}(u)=F_{\mu^N}^{-1}(u)=0 $ for $u\in (0,F_{\mu}\left(0 \right))$. We finally get $W^p_p(\mu,\mu^N)\le N^{-p}$. 

    Now, let $U\sim \mathcal{U}([0,1])$ be a uniform random variable on~$[0,1]$. Still by Theorem~2.9~\cite{santambrogio2015optimal}, we have $(F_\mu^{-1}(U), F_\nu^{-1}(U))\sim \pi^*$ and $(F_{\mu^N}^{-1}(U), F_{\nu^N}^{-1}(U))\sim \pi^N$. This gives a coupling between $\pi^*$ and $\pi^N$, and thus
    $$W_p^p(\pi^N,\pi^*)\le \E[|F_{\mu^N}^{-1}(U)-F_\mu^{-1}(U)|^p]+\E[|F_{\nu^N}^{-1}(U)-F_\nu^{-1}(U)|^p]\le \frac{2}{N^p}.$$
    
\end{proof}

In order to prove Proposition~\ref{prop:cvgP0MCOTtoOT},
let us introduce the following auxiliary problem.
For all $N\in \bN^*$, let us define
\begin{align*}
  \bar{\Pi}^N (\mu, \nu) : = &\bigg\{ (\bar{\pi}_{m,n})_{1\leq m,n \leq N}\, |  \, \forall 1\leq m,n\leq N, \; \bar{\pi}_{m,n} \ge 0, \\
 &   \forall m, \sum_{n=1}^{N} \bar{\pi}_{m,n} =   \mu(T^N_m), \, \forall n, \sum_{m=1}^{N} \bar{\pi}_{m,n} =  \nu(T^N_n) \bigg\}
\end{align*}
and
\begin{equation}\label{eqn:P0DiscretePbDef}
 J^N
 := \inf_{ \bar{\pi} \in \bar{\Pi}^N(\mu, \nu) } \sum_{m,n=1}^{N}
  c \left( \frac{m - \frac{1}{2}}{N}, \frac{n - \frac{1}{2}}{N} \right) \bar{\pi}_{m,n}.
\end{equation}

Let us introduce the following applications:
\begin{equation}\label{eqn:P0AppDDef}
  \begin{array}{lccl}
    D:  & \Pi(\mu, \nu) & \to & \bar{\Pi}^N(\mu,\nu) \\
     & \pi & \mapsto & (\pi(T^N_m\times T^N_n))_{1\leq m,n \leq N}
  \end{array}
\end{equation}
and
\begin{equation}\label{eqn:P0AppJDef}
  \begin{array}{lccl}
    J : & \bar{\Pi}^N(\mu,\nu) & \to & \mathbb{R}^+ \\
    & \bar{\pi} & \mapsto & \sum_{m,n=1}^{N}
    c \left( \frac{m - \frac{1}{2}}{N}, \frac{n - \frac{1}{2}}{N} \right) \bar{\pi}_{m,n}.
  \end{array}
\end{equation}

\begin{lemma}\label{lem:P0MCOTDiscrPbEquiv}
 Let $N\in \bN^*$. It holds that
\begin{equation}\label{estimIJ}
 \forall \pi \in \cP([0,1]), \  \left| I(\pi) - J(D(\pi)) \right| \le \frac{K}{2N}.
\end{equation}
Besides, we have 
  \begin{equation}
    I^N \le J^N \le I^N + \frac{K}{2N}.
  \end{equation}
\end{lemma}

\begin{proof}[Proof of Lemma \ref{lem:P0MCOTDiscrPbEquiv}]
Let $\pi \in \cP([0,1]^2)$.
Then, we write
  \begin{equation*}
   \begin{split}
    I(\pi) & = \int_{[0,1]^2}  c(x,y) \dd \pi(x,y) = \sum_{m,n = 1}^{N} \int_{T_m^N \times T_n^N} c(x,y) \dd \pi (x, y )
   \\   & = \sum_{m,n = 1}^{N} c\left(\frac{m - \frac{1}{2}}{N}, \frac{n - \frac{1}{2}}{N} \right)D_{mn}(\pi)
   \\   &\quad  + \sum_{m,n =1}^{N} \int_{T_m^N \times T_n^N}
            \left(c(x,y) - c\left(\frac{m - \frac{1}{2}}{N}, \frac{n - \frac{1}{2}}{N} \right) \right)
            \dd \pi(x,y),
   \end{split}
 \end{equation*}
  and get $\left| I(\pi) - J(D(\pi)) \right| \le \frac{K}{2N}$ since $|c(x,y) - c\left(\frac{m - \frac{1}{2}}{N}, \frac{n - \frac{1}{2}}{N} \right)|\le\frac{K}{2N}$ for $(x,y)\in T_m^N \times T_n^N$.

  Let $N\in \bN^*$. For all $\bar{\pi} \in \bar{\Pi}(\mu, \nu)$, defining $\pi:= \sum_{m,n=1}^N \bar{\pi}_{mn} \delta_{\frac{m - \frac{1}{2}}{N}, \frac{n - \frac{1}{2}}{N} }$, one obtains that 
$\pi \in \cP([0,1]^2)$, $D(\pi) = \bar{\pi}$ and $I(\pi) = J(\bar{\pi})$; this implies that $I^N \leq J^N$.

  Conversely, if  $\pi \in \Pi(\mu,\nu;(\phi^N_m)_{1\le m \le N}, (\phi^N_n)_{1\le n \le N})$ is chosen to satisfy $I(\pi) \leq I^N + \epsilon$ for some $\epsilon>0$, one gets
$
J^N \le J(D(\pi)) \le I(\pi) + \frac{K}{2N} = I^N + \frac{K}{2N} + \epsilon
$.
  Letting $\epsilon \to 0$ provides the wanted result.
\end{proof}

We also need the following auxiliary lemma. 

\begin{lemma}\label{lem:P0OTDiscrPbEquiv}
  For all $\bar{\pi} \in \bar{\Pi}^N(\mu, \nu)$, there exists $\bar{\pi}^* \in \Pi(\mu,\nu)$ such that
$\bar{\pi} = D(\bar{\pi}^*)$.
\end{lemma}

\begin{proof}[Proof of Lemma \ref{lem:P0OTDiscrPbEquiv}]
  Let $\bar{\pi} \in \bar{\Pi}(\mu, \nu)$. We define $\bar{\pi}^*$ by
  $$\dd\bar{\pi}^*(x,y)=\dd \mu(x) \sum_{m=1}^N \Ind{T^N_m}(x)\sum_{n=1}^N \frac{\bar{\pi}_{m,n}}{\sum_{n'=1}^N \bar{\pi}_{m,n'}} \frac{\Ind{T^N_n}(y) \dd \nu(y)}{\nu(T^N_n)}.  $$
  Since $\sum_{n'=1}^N \bar{\pi}_{m,n'}=\mu(T^N_m)$ and $\sum_{m=1}^N \bar{\pi}_{m,n}=\nu(T^N_n)$, we have
  \begin{align*}
    \int_{\cX}\dd\bar{\pi}^*(x,y)&=\sum_{m=1}^N \mu(T^N_m)\sum_{n=1}^N \frac{\bar{\pi}_{m,n}}{\mu(T^N_m)} \frac{\Ind{T^N_n}(y) \dd \nu(y)}{\nu(T^N_n)}\\
    &=\sum_{n=1}^N \left(\sum_{m=1}^N \bar{\pi}_{m,n}\right) \frac{\Ind{T^N_n}(y) \dd \nu(y)}{\nu(T^N_n)}=\sum_{n=1}^N \Ind{T^N_n}(y) \dd \nu(y)=\dd \nu(y) .
  \end{align*}
Also, we have $    \int_{\cY}\dd\bar{\pi}^*(x,y)=\dd \mu(x) \sum_{m=1}^N \Ind{T^N_m}(x)\sum_{n=1}^N \frac{\bar{\pi}_{m,n}}{\sum_{n'=1}^N \bar{\pi}_{m,n'}}=\dd \mu(x),$
which gives $\bar{\pi}^* \in \Pi(\mu,\nu)$. Last, we have
$$\int_{T^N_m\times T^N_n}\dd\bar{\pi}^*(x,y)=\mu(T^N_m)\frac{\bar{\pi}_{m,n}}{\sum_{n'=1}^N \bar{\pi}_{m,n'}} =\bar{\pi}_{m,n},$$
which precisely gives $\bar{\pi}=D(\bar{\pi}^*)$.
\end{proof}
  
We are now in position to give the proof of Proposition~\ref{prop:cvgP0MCOTtoOT}.

\begin{proof}[Proof of Proposition \ref{prop:cvgP0MCOTtoOT}]
The inclusion $\Pi(\mu,\nu;(\phi^N_m)_{1\le m \le N}, (\phi^N_n)_{1\le n \le N}) \subset \Pi(\mu,\nu)$ gives $I^N\le I$.
 
   Lemma \ref{lem:P0OTDiscrPbEquiv} implies that for all $\bar{\pi} \in \bar{\Pi}^N(\mu, \nu)$, there exists  $\bar{\pi}^*\in \Pi(\mu,\nu)$ such that $D(\bar{\pi}^*)=\bar{\pi}$, and we get by Lemma~\ref{lem:P0MCOTDiscrPbEquiv}
$    \left| J(\bar{\pi}) - I(\bar{\pi}^*) \right| \le \frac{K}{2N}
$.
Let now $\bar{\pi}\in\bar{\Pi}^N(\mu, \nu)$ such that $J( \bar{\pi}) \leq J^N +\epsilon$ for some $\epsilon>0$. 
Then one gets that $J^N + \frac{K}{2N} +\epsilon \geq I \left(\bar{\pi}^*\right) \ge I$. Letting $\epsilon$ go to zero yields that
    \begin{equation}\label{eqn:P0PropProof2ndIneq}
  I\le  J^N + \frac{K}{2N}.
  \end{equation}
 Furthermore, Lemma \ref{lem:P0MCOTDiscrPbEquiv} gives $J^N \le I^N+ \frac{K}{N}$ and thus
$I \le I^N + \frac{K}{N}$.
\end{proof}

\begin{remark}Proposition~\ref{prop:cvgP0MCOTtoOT} can be easily extended to higher dimensions and in the multi-marginal case. Let us assume that $c:([0,1]^d)^M \to \R_+$ be such that
  $$|c(x_1,\dots,x_M)-c(x'_1,\dots,x'_M)|\le K \max_{i\in \{1,\dots, M\}}\|x_i-x'_i\|_\infty.$$ 
  For $N\in \N^*$ and ${\bf m}\in \{1,\dots,N\}^d=:\mathcal{E}_N$, we consider the test function $\phi^N_{\bf m}(x)=\prod_{i=1}^d \phi^N_{m_i}(x_i)$ for $x\in [0,1]^d$. Then, with  $I=\inf_{ \pi \in \Pi(\mu_1,\dots,\mu_M) } \int_{([0,1]^d)^M} c(x_1,\dots,x_M) \dd \pi(x_1,\dots,x_M)$ and
  $$
	I^N = \inf_{
		\substack{
		  \pi : \forall {\bf m}, k , \int_{[0,1]^d} \phi^N_{\bf m}(x) \dd \mu_k(x)=\int_{([0,1]^d)^M} \phi^N_{\bf m}(x_k) \dd \pi(x_1,\dots,x_M)
		}
	}
	\left\{ \int_{([0,1]^d)^M} c(x_1,\dots,x_M) \dd \pi(x_1,\dots,x_M)	
	\right\},$$
we get similarly (it is straightforward to generalize Proposition~\ref{prop_dist_optimum} and we can extend the result of Lemma~\ref{lem:P0OTDiscrPbEquiv} by induction on~$M$)
  \begin{equation*}
    I^N \le I^* \le I^N + \frac{K}{N}.
  \end{equation*} 
Since the number of moments (i.e. of test functions) is $MN^d$, we see that there is a curse of dimensionality with respect to~$d$, not with respect to~$M$. 
\end{remark}

\subsection{Piecewise affine test functions in dimension 1 on a compact set}

The test functions considered are discontinuous piecewise affine functions, identical on each space. For all $N\in \bN^*$ and all $1\leq m \leq N$, let us define the following discontinuous piecewise affine functions
\begin{align*}
  	\phi^N_{m,1}(x) &= \left\{
  		\begin{array}{ll}
  			N \left(x - \frac{m-1}{N} \right) & \text{if} \quad x \in T_m^N, \\
  			0 & \text{otherwise},\\
  		\end{array}
  	\right. \\
    \phi^N_{m,2}(x) &= \left\{
  		\begin{array}{ll}
  			1 - N \left( x - \frac{m-1}{N} \right)& \text{if} \quad x \in T_m^N, \\
  			0 & \text{otherwise},\\
  		\end{array}
  	\right.\\
\end{align*}
  and for all $i=1,2$,
  $$
  \mu^N_{m,i}:= \int_{\cX} \phi^N_{m,i}\,d\mu \quad \mbox{ and }   \nu^N_{m,i}:= \int_{\cY} \phi^N_{m,i}\,d\nu.
  $$

  \begin{lemma}\label{lem:P1MomentsEqConsqces}    
  Let $\mu_1, \mu_2 \in \cP([0,1])$. Let $N\in \bN^*$ and let us assume that
for all $1 \leq m \leq N$ and $i=1,2$,
$$
\int_{[0,1]}\phi^N_{m,i}(x) \dd \mu_1(x) = \int_{[0,1]} \phi^N_{m,i}(x) \dd \mu_2(x).
$$

  Then, denoting by $F_1 : [0,1] \to [0,1]$ (resp. $F_2 : [0,1] \to [0,1]$)
  the cumulative distribution function of $\mu_1$ (resp. $\mu_2$),
  one gets that
  \begin{equation}\label{eqn:P1MomentEqCsq1}
   \forall 1\leq m \leq N, \quad  \int_{T_m^N} F_1(x) \dd x = \int_{T_m^N} F_2(x) \dd x,
  \end{equation}
  and
  \begin{equation}\label{eqn:P1MomentEqCsq2}
    \forall 1\leq m \leq N,\quad F_1\left(\frac{m}{N}\right) = F_2\left(\frac{m}{N}\right).
  \end{equation}
\end{lemma}

\begin{proof}
We have $\phi_{m,1}+\phi_{m,2}=\Ind{T^N_m}$ and thus, for $2\le m\le N$, $F_1\left(\frac m N \right)-F_1\left(\frac {m-1} N \right)=F_2\left(\frac m N \right)-F_2\left(\frac {m-1} N \right)$. Since $F_1(1)=F_2(1)=1$, this gives~\eqref{eqn:P1MomentEqCsq2}. Now, let $l=1,2$. 
An integration by parts yields for $1\le m\le N$
	\begin{align}\nonumber
			\int_{[0,1]} \phi^N_{m,1}(x) \dd  \mu_l(x)
			& = \int_{\frac{m-1}N}^{\frac m N} (x-\frac{m-1}N) \dd  \mu_l(x)\\ \nonumber
			&= \frac{1}{N} F_l\left(\frac{m}{N}\right) -  \int_{\frac{m-1}N}^{\frac m N} F_l(x) \dd x  \label{eq:test1}		
	\end{align}
Using~\eqref{eqn:P1MomentEqCsq2}, this gives~\eqref{eqn:P1MomentEqCsq1}.
\end{proof}
Let us remark that we may have $F_1(0)\not=F_2(0)$ under the assumptions of Lemma~\ref{lem:P1MomentsEqConsqces}, since $\mu_1$ and $\mu_2$ may charge differently $0$.

Let us now explain with a rough calculation why considering these test functions may lead to a convergence rate of $O(1/N^2)$ when $c$ is $C^1$ with a Lipschitz gradient. Let 
\begin{equation}\label{eqn:P1MCOT}
	I^N = \inf_{
		\substack{
			\pi \in \Pi(\mu,\nu;(\phi^N_{m,i}), (\phi^N_{n,i}))
		}
	}
	\left\{
		\int_{\cX \times \cY} c(x, y) \dd \pi(x,y)
	\right\}.
\end{equation}
We have $I^N\le I$ and, for any $\pi \in \Pi(\mu,\nu;(\phi^N_{m,i}), (\phi^N_{n,i}))$,
\begin{align*}
  I(\pi)= \sum_{m,n = 1}^{N} \int_{T_m^N \times T_n^N} &c\left(\frac{m - \frac{1}{2}}{N}, \frac{n - \frac{1}{2}}{N} \right) +\partial_x c\left(\frac{m - \frac{1}{2}}{N}, \frac{n - \frac{1}{2}}{N} \right)(x-\frac{m - \frac{1}{2}}{N})\\&+\partial_y c\left(\frac{m - \frac{1}{2}}{N}, \frac{n - \frac{1}{2}}{N} \right)(y-\frac{m - \frac{1}{2}}{N}) \dd x \dd y +O(1/N^2)
\end{align*}
Thus, we have \begin{align}
  I(\pi) =&\sum_{m,n = 1}^{N} \left(c- \frac 12\partial_x- \frac 12\partial_y \right) \left(\frac{m - \frac{1}{2}}{N}, \frac{n - \frac{1}{2}}{N} \right) \pi^1_{mn}\label{linearP1problem}\\
  &+\partial_x c\left(\frac{m - \frac{1}{2}}{N}, \frac{n - \frac{1}{2}}{N} \right)\pi^2_{mn}+\partial_y c\left(\frac{m - \frac{1}{2}}{N}, \frac{n - \frac{1}{2}}{N} \right)\pi^3_{mn}+O(1/N^2), \notag
\end{align}
with  $\pi^1_{mn}=\pi(T_m^N \times T_n^N)$, $N \pi^2_{mn}=\int_{T_m^N \times T_n^N}\phi^N_{m,1}(x) \dd \pi(x,y)$ and $N \pi^3_{mn}=\int_{T_m^N \times T_n^N}\phi^N_{m,1}(y) \dd \pi(x,y)$. We can thus consider the linear programming problem of minimizing the right-hand-side of~\eqref{linearP1problem} under the constraints $\sum_n \pi^1_{mn}=\mu^N_{m,1}+\mu^N_{m,2}$, $\sum_m \pi^1_{mn}=\nu^N_{m,1}+\nu^N_{m,2}$, $\sum_n \pi^2_{mn}=\mu^N_{m,1}/N$, $\sum_m \pi^3_{mn}=\nu^N_{m,1}/N$ and $\pi^i_{mn}\ge 0$. Suppose for simplicity that we can find an minimum $(\pi^{*,i}_{mn})$ to this discrete problem. If we could find (similarly as Lemma~\ref{lem:P0OTDiscrPbEquiv}) $\pi^* \in \Pi(\mu, \nu)$ such that $\pi^{*,1}_{mn}=\pi^*(T_m^N \times T_n^N)$, $N \pi^{*,2}_{mn}=\int_{T_m^N \times T_n^N}\phi^N_{m,1}(x) \dd \pi^*(x,y)$ and $N \pi^{*,3}_{mn}=\int_{T_m^N \times T_n^N}\phi^N_{m,1}(y) \dd \pi^*(x,y)$, we would get then
$$  I\le I^N+O(1/N^2).$$
Unfortunately, such kind of a result is not obvious. Besides, we see from this derivation that the smoothness of the cost function plays an important role.

Let us recall that for $p\ge 1$, the $W_p$-Wasserstein distance at the power~$p$, $W^p_p(\mu,\nu)$, corresponds to the cost function $c(x,y)=|x-y|^p$. In the following, we prove convergence results with rate $O(1/N^2)$ for $c(x,y)=|x-y|$ and $c(x,y)=|x-y|^2$. In the first case, the cost function is not smooth on the diagonal, and we need to impose an extra condition on~$\mu$ and $\nu$ to get this rate. We first state a first result, which is already interesting, but will be not sufficient to prove the desired convergence. Its proof is postponed in Appendix~\ref{sec_Appendix}.

\begin{proposition}\label{prop:P1ControlRegularMarginal}
Let $\mu_1, \mu_2\in \cP([0,1])$ be two probability measures with cumulative distribution functions $F_1$ and $F_2$, respectively, such that $F_1, F_2 \in C^2([0,1])$.
Let us assume that
for all $1 \leq m \leq N$ and $i=1,2$,
$$
\int_{[0,1]}\phi^N_{m,i}(x) \dd \mu_1(x) = \int_{[0,1]} \phi^N_{m,i}(x) \dd \mu_2(x).
$$
Then,
\begin{equation}\label{eq:W1}
W_1(\mu_1, \mu_2) \leq \frac{\| F_1^{\prime \prime} \|_\infty + \|F_2^{\prime \prime} \|_\infty }{3N^2}.
\end{equation}
In addition, let $m_1 := \min_{u \in [0,1]}F_1'(u)$ and $m_2 = \min_{u \in [0,1]}F'_2(u)$ and let us assume that $m_1 > 0$ and $m_2 >0$. Then, 
for all $p>1$, it holds that
\begin{equation}\label{eq:Wp}
W_p(\mu_1, \mu_2) 
		\leq \frac{\|F_1^{\prime \prime} \|_\infty + \|F_2^{\prime \prime} \|_\infty}{3N^2}
			(p!)^{\frac{1}{p}}
			\left(\frac{5}{2} \left(\frac{1}{m_1} + \frac{1}{m_2}\right) \right)^{\frac{p-1}{p}}.
	\end{equation}
\end{proposition}

\begin{remark}\label{rem:P1ControlRegularMarginal}
  The result of Proposition \ref{prop:P1ControlRegularMarginal}
  can be extended through the triangle inequality in order
  to treat regular measures with different piecewise affine moments.
  Indeed, for $p \geq 1$:
	\begin{equation*}
		W_p(\mu, \nu) \leq W_p(\mu, \tilde{\mu}) + W_p(\tilde{\mu}, \tilde{\nu}) + W_p(\tilde{\nu}, \nu),
	\end{equation*}
	thus
	\begin{equation}
		\left| 	W_p(\mu, \nu) - W_p(\tilde{\mu}, \tilde{\nu}) \right|
		\leq W_p(\mu, \tilde{\mu}) + W_p(\tilde{\nu}, \nu).
	\end{equation}
	Thus, using Proposition \ref{prop:P1ControlRegularMarginal}, one gets that
  for $\mu$, $\nu$ two measures with cumulative distribution functions
	$F$ and $G$, respectively, such that $F, G \in C^2([0,1])$
	and $\tilde{\mu}$, $\tilde{\nu}$ two measures with cumulative distribution functions
	$\tilde{F}$ and $\tilde{G}$, respectively, such that $\tilde{F}, \tilde{G} \in C^2([0,1])$;
	If $\mu$ and $\tilde{\mu}$ (respectively $\nu$ and $\tilde{\nu}$)
	have the same $2N$ piecewise affine moments of step $1/N$, then
	\begin{equation}
		\left| W_1(\mu, \nu) - W_1(\tilde{\mu}, \tilde{\nu}) \right|
		\leq \frac{\|F''\|_\infty + \|\tilde{F}''\|_\infty + \|G''\|_\infty + \|\tilde{G}''\|_\infty}{3N^2}.
	\end{equation}
	Besides, if $m_\mu = \min_{u \in [0,1]}F'(u)$,
	$m_{\tilde{\mu}} = \min_{u \in [0,1]}\tilde{F}'(u)$,
	$m_\nu = \min_{u \in [0,1]}G'(u)$ and
	$m_{\tilde{\nu}} = \min_{u \in [0,1]}\tilde{G}'(u)$,
	are positive, one has for all $p\in\N^*$,
	\begin{multline}
		\left| W_p(\mu, \nu) - W_p(\tilde{\mu}, \tilde{\nu}) \right| \\
		\leq \frac{\|F^{\prime \prime} \|_\infty + \|\tilde{F}^{\prime \prime} \|_\infty}{3N^2}
			\left(\frac{5}{2} \left(\frac{1}{m_\mu} + \frac{1}{m_{\tilde{\mu}}}\right) \right)^{\frac{p-1}{p}} (p!)^{\frac{1}{p}}
		\\ + \frac{\|G^{\prime \prime} \|_\infty + \|\tilde{G}^{\prime \prime} \|_\infty}{3N^2}
			\left(\frac{5}{2} \left(\frac{1}{m_\nu} + \frac{1}{m_{\tilde{\nu}}}\right) \right)^{\frac{p-1}{p}} (p!)^{\frac{1}{p}}.
	\end{multline}
\end{remark}

Unfortunately,   Proposition \ref{prop:P1ControlRegularMarginal} can not be extended   to non-smooth measures, as Example \ref{exmpl:RegularityForP1Control} below shows. However, the $O(1/N^2)$ convergence obtained in Remark \ref{rem:P1ControlRegularMarginal} may stay true even for non-smooth measures~$\tilde{\mu}$ and $\tilde{\nu}$.  This is important in our context to treat the case where $\tilde{\mu}$ and $\tilde{\nu}$ are not smooth since the solution of the MCOT problem may typically be a discrete measure that match respectively the moments of $\mu$ and $\nu$. We tackle this issue for $W_1$ and $W_2$ in the two following paragraphs.

\begin{example}\label{exmpl:RegularityForP1Control}
	In Proposition \ref{prop:P1ControlRegularMarginal}, if one of the
	measures (let us say $\tilde{\mu}$) is not regular enough, then the convergence in $O(1/N^2)$ may not be true, as shown thereafter.

	We consider $\mu \sim \mathcal{U}([0,1])$ and
	\begin{equation*}
		\tilde{\mu}_N = \frac{1}{N} \sum_{i = 1}^N \delta_{\frac{1}{2N} + \frac{i-1}{N}}.
	\end{equation*}
	Then, for all $1\le m\le N$, we have
	\begin{equation*}
		\tilde{F}\left(\frac{m}{N}\right) = \frac{m}{N} = F\left(\frac{m}{N}\right),
	\end{equation*}
	and
	\begin{equation*}
		\int_{T^N_m} \tilde{F} = \frac{m-1}{N^2} + \frac{1}{2N}\frac{1}{N}
		= \int_{T^N_m} F.
	\end{equation*}

	However, we have
	\begin{equation*}
		\begin{split}
			W_1(\mu, \tilde{\mu}_N) &= N \int_0^{1/N} \left| u - \frac{1}{2N} \right| \dd u 
			= 2N \left( \frac{1}{2N} \right)^2 \frac{1}{2} 
			= \frac{1}{4N}.
		\end{split}
	\end{equation*}
\end{example}

\subsubsection{Convergence speed for $W_1$}\label{sect:cvgSpeedP1UnderEsttoOT}

\begin{proposition}\label{prop:cvgP1MCOTtoOT}
Let $\mu,\nu, \tilde{\mu}, \tilde{\nu} \in \cP([0,1])$.
Let us assume that $\mu$ and $\nu$ are absolutely continuous with respect to the Lebesgue measure and let us denote by $\rho_\mu$ and $\rho_\nu$ their density probability functions.
Let us denote by $F_\mu$, $F_\nu$, $F_{\tilde{\mu}}$ and $F_{\tilde{\nu}}$ the cumulative distribution functions of $\mu$, $\nu$, $\tilde{\mu}$ and $\tilde{\nu}$ respectively.
Let $N\in \bN^*$. Let us assume that
  \begin{equation}\label{eqn:P1EqualityMoments}
    \forall 1\leq m \leq N, \quad \int_{T_m^N} F_\mu = \int_{T_m^N} F_{\tilde{\mu}}
    \quad \text{and} \quad
    \int_{T_m^N} F_\nu = \int_{T_m^N} F_{\tilde{\nu}}.
  \end{equation}
Let us assume in addition that the function $F_\mu - F_\nu$ changes sign at most $Q$ times for some $Q\in \bN$. More precisely, denoting by $G:= F_\mu - F_\nu$, we assume that there exist
  $x_0 = 0 < x_1< x_2 < \cdots < x_{Q} < x_{Q+1} = 1 \in [0,1]$ such that for all $1\leq q \leq Q+1$, 
\begin{equation}\label{eq:sign1}
\forall x,y \in [x_{q-1}, x_q],\; G(x)G(y) \geq 0,
\end{equation}
and for all $1\leq q \leq Q$, 
\begin{equation}\label{eq:sign2}
\forall x\in [x_{q-1}, x_q],\; \forall z\in [x_q, x_{q+1}], \;  G(x)G(z) \leq 0.
\end{equation}
Let us also assume that $\rho_{\mu}- \rho_\nu\in L^\infty([0,1],\dd x;\R)$.
Then,
  \begin{equation}\label{eq:toprove}
    W_1(\mu, \nu) \le W_1(\tilde{\mu}, \tilde{\nu}) + 2 \|\rho_\mu - \rho_\nu \|_\infty \frac{Q}{N^2} .
  \end{equation}
\end{proposition}
The key thing to notice is that we only assume regularity on the measures $\mu,\nu$, not on
$\tilde{\mu}, \tilde{\nu}$. The assumption that $F_\mu - F_\nu$ changes sign at most $Q$ times is related to the fact that $c(x,y)=|x-y|$ is not smooth on the diagonal: an optimal coupling is given by the inverse transform coupling, and  $F_\mu^{-1} - F_\nu^{-1}$ changes sign at most $Q$ times as well. Last, remarkably, we do not need for this result to assume $F_\mu(m/N)=F_{\tilde{\mu}}(m/N)$ and $F_\nu(m/N)=F_{\tilde{\nu}}(m/N)$. Thus, it is sufficient to work with continuous piecewise affine test functions.

More precisely, for all $N\in \bN^*$, let us define
$$
\forall x\in [0,1], \quad \psi^N_1(x) = \left\{
      \begin{array}{ll}
        1 - Nx & \text{if} \quad x \in T^N_{1} \\
        0  & \text{elsewhere},
      \end{array}       
      \right. $$
and for all $2\leq m \leq N$,
$$
\psi^N_{m}(x) = \left\{
      \begin{array}{ll}
        N \left(x - \frac{m-2}{N} \right) & \text{if} \quad x \in T^N_{m-1} \\
       1 - N\left(x - \frac{m-1}{N}\right) & \text{if} \quad x \in T^N_{m} \\
        0  & \text{elsewhere}.
        \end{array}\right.
$$
      We can check by integration by parts that $\int_{[0,1]}\psi^N_1(x) \dd \mu(x)=N \int_{T^N_1}F_\mu(x) \dd x$ and $\int_{[0,1]}\psi^N_m(x) \dd \mu(x)=N \int_{T^N_m}F_\mu(x) \dd x- N \int_{T^N_{m-1}}F_\mu(x) \dd x$ for $2\leq m \leq N$. Therefore, we get
      \begin{align}
        &\forall m\in \{1,\dots,N\}, \int_{[0,1]}\psi^N_m(x) \dd \mu(x)=\int_{[0,1]}\psi^N_m(x) \dd \tilde{\mu}(x) \notag \\
        \iff& \forall m\in \{1,\dots,N\}, \int_{T^N_m}F_\mu(x) \dd x=\int_{T^N_m}F_{\tilde{\mu}(x)} \dd x. \label{equiv_psimN}
      \end{align}
      Last, let us remark that $\psi^N_1=\phi^N_{1,2}$ and $\psi^N_m=\phi^N_{m-1,1}-\phi^N_{m,2}$ for $2\le m\le N$ so that $\mathrm{Span}\left\{ \psi^N_n, \, 1\le n\le N \right\} \subset \mathrm{Span}\left\{ \phi^N_{n,1}, \phi^N_{n,2}, \, 1\le n\le N \right\} $ and
      $$  \Pi(\mu,\nu;(\phi^N_{n,l}),(\phi^N_{n,l}))\subset \Pi(\mu,\nu;(\psi^N_n) ,(\psi^N_n) )  .$$

\begin{corollary}\label{cor:cvgP1MCOTtoOT}
Let $\mu,\nu \in \cP([0,1])$. Let us assume that $\mu$ and $\nu$ are absolutely continuous with respect to the Lebesgue measure and let us denote by $\rho_\mu$ and $\rho_\nu$ their density probability functions. 
Let $F_\mu$ and $F_\nu$ be their cumulative distribution functions.
For all $N\in \bN^*$, let us define
  \begin{equation}\label{eqn:cvgP1MCOT}
    I^N = \inf_{
      \substack{
        \pi \in \Pi(\mu,\nu;(\psi^N_m)_{1 \le m \le N},(\psi^N_n)_{1 \le n \le N})      }
    }
    \left\{
      \int_{[0,1]^2} |x-y| \dd \pi(x, y)
    \right\}.
  \end{equation}
There exists a minimizer for~\eqref{eqn:cvgP1MCOT}.   Let us assume in addition that the function $F_\mu - F_\nu$ changes sign at most $Q$ times for some $Q\in \bN$ (in the sense of Proposition~\ref{prop:cvgP1MCOTtoOT}) and that  $\rho_{\mu}- \rho_\nu\in L^\infty([0,1],\dd x;\R)$.
  Then,
  \begin{equation}\label{eqn:cvgP1MCOTtoOTRemark}
    I^N \le W_1(\mu, \nu) \le I^N + 2  \| \rho_\mu - \rho_\nu \|_\infty \frac{Q}{N^2}
  \end{equation}
\end{corollary}
In fact, looking at the proof of Proposition \ref{prop:cvgP1MCOTtoOT}, it even is sufficient to assume that $\rho_\mu-\rho_\nu$ is bounded on a neighborhood of the points at which $F_\mu - F_\nu$ changes sign. For simplicity of statements, we have assumed in Proposition~\ref{prop:cvgP1MCOTtoOT} and Corollary~\ref{cor:cvgP1MCOTtoOT} that $\rho_\mu-\rho_\nu$ is bounded on~$[0,1]$.

\begin{proof}[Proof of Corollary \ref{cor:cvgP1MCOTtoOT}]
From the inclusion $ \Pi(\mu,\nu;(\psi^N_m)_{1 \le m \le N},(\psi^N_n)_{1 \le n \le N})\subset \Pi(\tilde{\mu}, \tilde{\nu})$, we clearly have $I^N\le W_1(\mu,\nu)$. Using Theorem~\ref{prop:approxProblmDiscreteMeasGeneral} together with Remark~\ref{Rk:prop31}, since the functions $\psi_m^N$ are continuous on $[0,1]$ for all $1\leq m \leq N$, there exists 
$\pi^N \in \Pi(\mu,\nu;(\psi^N_m)_{1 \le m \le N},(\psi^N_n)_{1 \le n \le N})$ which is a minimizer to Problem~(\ref{eqn:cvgP1MCOT}). Let us denote by $\tilde{\mu}$ and $\tilde{\nu}$ the marginal laws of $\pi^N$. 
First, we  remark that
$$ \Pi(\mu,\nu;(\psi^N_m)_{1 \le m \le N},(\psi^N_n)_{1 \le n \le N})= \Pi(\tilde{\mu},\tilde{\nu};(\psi^N_m)_{1 \le m \le N},(\psi^N_n)_{1 \le n \le N})\subset \Pi(\tilde{\mu}, \tilde{\nu})$$
and thus

  \begin{equation*}
    \begin{split}
    I^N =   \int_0^1 |x - y| \dd \pi^N(x, y)
      &= \min_{\pi \in \Pi(\tilde{\mu}, \tilde{\nu})}
        \left\{
          \int_0^1 |x - y| \dd \pi(x, y)
        \right\}= W_1(\tilde{\mu}, \tilde{\nu}).
    \end{split}
  \end{equation*}
Besides,  it holds that for all $1\leq m \leq N$, 
$$
\int_{[0,1]} \psi^N_m(x) \dd \tilde{\mu}(x) = \int_{[0,1]} \psi^N_m(x) \dd\mu (x), \quad \int_{[0,1]} \psi^N_m(y) \dd \tilde{\nu}(y) = \int_{[0,1]} \psi^N_m(y) \dd \nu(y),
$$
and we therefore get~\eqref{eqn:P1EqualityMoments} from~\eqref{equiv_psimN}. We can thus apply Proposition \ref{prop:cvgP1MCOTtoOT} and get the desired result.
\end{proof}

\begin{proof}[Proof of Proposition \ref{prop:cvgP1MCOTtoOT}]

Let $1\leq m \leq N$. If for all $1\leq q \leq Q$, $x_q \notin T_m^N$, then $F_\mu - F_\nu$ remains non-negative or non-positive on $T_m^N$. Thus, using \eqref{eqn:P1EqualityMoments}, it holds that
\begin{align*}
\int_{T_m^N} \left| F_\mu  - F_\nu \right|
    &= \epsilon \int_{T_m^N} \left( F_\mu - F_\nu \right)\\
   & = \epsilon \int_{T_m^N} \left( F_{\tilde{\mu}} - F_{\tilde{\nu}} \right)
    = \left| \int_{T_m^N} \left( F_{\tilde{\mu}} - F_{\tilde{\nu}}\right)  \right|
    \leq \int_{T_m^N} \left| F_{\tilde{\mu}} - F_{\tilde{\nu}} \right|,
\end{align*}
where $\epsilon = 1$ if $F_\mu - F_\nu \geq 0$ on $T_m^N$ and $\epsilon = -1$ if $F_\mu - F_\nu \leq 0$ on $T_m^N$. 
On the other hand, if there exists $1\leq q \leq Q$, such that $x_q \in T_m^N$, one gets
  \begin{equation*}
    \begin{split}
      \int_{T_m^N} \left| F_\mu  - F_\nu \right|
      &= \int_{T_m^N} \left( F_\mu - F_\nu \right) + 2 \int_{T_m^N} \left( F_\mu - F_\nu \right)^- \\
      &= \int_{T_m^N} \left( F_{\tilde{\mu}} - F_{\tilde{\nu}} \right) + 2 \int_{T_m^N} \left( F_\mu - F_\nu \right)^- \\
      &\le \int_{T_m^N} \left| F_{\tilde{\mu}} - F_{\tilde{\nu}} \right| + 2 \int_{T_m^N} \left( F_\mu - F_\nu \right)^- \\
      &\le \int_{T_m^N} \left| F_{\tilde{\mu}} - F_{\tilde{\nu}} \right|
        + 2 \| \rho_\mu - \rho_\nu \|_\infty   \frac{1}{N^2},
    \end{split}
  \end{equation*}
  since  for $x\in T^N_m$, $F_\mu(x) - F_\nu(x)=\int_{x_q}^x  \rho_\mu - \rho_\nu $ and $|x-x_q|\le 1/N$.

  Thus, as there are at most $Q$ intervals of that last type, we get
  \begin{equation*}
\int_0^1 \left| F_\mu  - F_\nu \right| \le \int_0^1 \left| F_{\tilde{\mu}} - F_{\tilde{\nu}} \right| + 2 \| \rho_\mu - \rho_\nu \|_\infty  \frac{Q}{N^2},
  \end{equation*}
  i.e. $W_1(\mu, \nu) \le W_1(\tilde{\mu}, \tilde{\nu}) + 2  \| \rho_\mu - \rho_\nu \|_\infty  \frac{Q}{N^2}$.
\end{proof}

\subsubsection{Convergence speed for $W_2$}

\begin{proposition}\label{prop:cvgP1_W2}
Let $\mu,\nu, \tilde{\mu}, \tilde{\nu} \in \cP([0,1])$.
Let us assume that $\mu(\dd x)=\rho_\mu(x)\dd x$ and $\nu(\dd x)=\rho_\nu(x)\dd x$  with $\rho_\mu,\rho_\nu \in L^\infty([0,1],\dd x;\R_+)$.
Let us denote by $F_\mu$, $F_\nu$, $F_{\tilde{\mu}}$ and $F_{\tilde{\nu}}$ the cumulative distribution functions of $\mu$, $\nu$, $\tilde{\mu}$ and $\tilde{\nu}$ respectively.
Let $N\in \bN^*$. Let us assume that
  \begin{align}
    \forall 1\leq m \leq N, F_\mu\left(\frac m N\right)=F_{\tilde{\mu}}\left(\frac m N\right)  \quad \text{and} \quad F_\nu\left(\frac m N\right)=F_{\tilde{\nu}}\left(\frac m N\right),\label{eqn:P1_EqMoments2}\\
    \label{eqn:P1_EqMoments}
    \forall 1\leq m \leq N, \quad \int_{T_m^N} F_\mu = \int_{T_m^N} F_{\tilde{\mu}}
    \quad \text{and} \quad
    \int_{T_m^N} F_\nu = \int_{T_m^N} F_{\tilde{\nu}} .   \end{align}
Then,
  \begin{equation}\label{eq:toprove_2}
    W^2_2(\mu, \nu) \le W^2_2(\tilde{\mu}, \tilde{\nu}) + \frac 73\frac{\|\rho_\mu\|_\infty+\|\rho_\nu\|_\infty }{N^2} .
  \end{equation}
\end{proposition}
This proposition plays the same role as Proposition~\ref{prop:cvgP1MCOTtoOT} for $W_1$. Again, the important point is that no regularity assumption is made on  $\tilde{\mu}$ and $\tilde{\nu}$. We note that we no longer have restriction on the number of points where $F_\mu-F_\nu$ changes sign, which is related to the fact that $c(x,y)=(x-y)^2$ is smooth. Contrary to Proposition~\ref{prop:cvgP1MCOTtoOT}, we need here the condition~\eqref{eqn:P1_EqMoments2}.

\begin{corollary}\label{cor:cvgP1MCOTtoOT_W2}
Let $\mu,\nu \in \cP([0,1])$. Let us assume that $\mu(\dd x)=\rho_\mu(x)\dd x$ and $\nu(\dd x)=\rho_\nu(x)\dd x$  with $\rho_\mu,\rho_\nu \in L^\infty([0,1],\dd x;\R_+)$. Let $F_\mu$ and $F_\nu$ be their cumulative distribution functions.
For all $N\in \bN^*$, let us define
  \begin{equation}\label{eqn:cvgP1MCOT_W2}
    I^N = \inf_{
      \substack{
        \pi \in \Pi(\mu,\nu;(\phi^N_{m,l})_{\scriptsize \substack{1 \le m \le N \\ 1\le l\le 2}},(\phi^N_{m,l})_{\scriptsize \substack{1 \le m \le N \\ 1\le l\le 2}})      }
    }
    \left\{
      \int_{[0,1]^2} (x-y)^2 \dd \pi(x, y)
    \right\}.
  \end{equation}
  Then,
  \begin{equation}\label{eqn:cvgP1MCOTtoOTRemark_W2}
    I^N \le W_2^2(\mu, \nu) \le I^N + \frac 73\frac{\|\rho_\mu\|_\infty+\|\rho_\nu\|_\infty }{N^2} .
  \end{equation}
\end{corollary}
We omit the proof of Corollary \ref{cor:cvgP1MCOTtoOT_W2} since it follows the same line as the one of Corollary \ref{cor:cvgP1MCOTtoOT}. The only difference is that we do not know here if the infimum is a minimum and have to work for an arbitrary $\epsilon>0$ with  $\pi\in \Pi(\mu,\nu;(\phi^N_{m,l})_{\scriptsize \substack{1 \le m \le N \\ 1\le l\le 2}},(\phi^N_{m,l})_{\scriptsize \substack{1 \le m \le N \\ 1\le l\le 2}})$ such that $ \int_{[0,1]^2} (x-y)^2 \dd \pi(x, y)\le I^N +\epsilon$.
Let us also mention here that we can use Proposition~\ref{prop_dist_optimum} to bound the distance between  an MCOT minimizer and an OT minimizer since $\phi_{m,1}+\phi_{m,2}=\Ind{T^N_m}$.

\begin{proof}[Proof of Proposition \ref{prop:cvgP1_W2}]
  From Lemma B.3~\cite{Jourdain2013}, we have
  \begin{align*}
    W_2^2(\mu, \nu)&=\int_0^1\int_0^1 \Ind{x<y}([F_\mu(x)-F_\nu(y)]^++[F_\nu(x)-F_\mu(y)]^+) \dd x \dd y \\
    &=\sum_{k=1}^N\sum_{l=k+1}^N \int_{T^N_k}\int_{T^N_l}([F_\mu(x)-F_\nu(y)]^++[F_\nu(x)-F_\mu(y)]^+)\dd x \dd y\\
    & \ + \sum_{k=1}^N  \int_{T^N_k}\int_{T^N_k}\Ind{x<y}([F_\mu(x)-F_\nu(y)]^++[F_\nu(x)-F_\mu(y)]^+)\dd x \dd y.
  \end{align*}
  The two terms $[F_\mu(x)-F_\nu(y)]^+$ and $[F_\nu(x)-F_\mu(y)]^+$ can be analyzed in the same way by exchanging $\mu$ and $\nu$, and we focus on the first one.  Thus, we consider for $k\le l$ the term $\alpha_{kl}:=\int_{T^N_k}\int_{T^N_l}\Ind{x<y}[F_\mu(x)-F_\nu(y)]^+\dd x \dd y$ and denote $\tilde{\alpha}_{kl}=\int_{T^N_k}\int_{T^N_l}\Ind{x<y}[F_{\tilde{\mu}}(x)-F_{\tilde{\nu}}(y)]^+\dd x \dd y$.
  
  \noindent  $\bullet$ If $F_\mu(k/N)\le F_\nu((l-1)/N)$, then from~\eqref{eqn:P1_EqMoments2}, we have also $F_{\tilde{\mu}}(k/N)\le F_{\tilde{\nu}}((l-1)/N)$ (note that if $l=1$, $F_{\tilde{\nu}}(0)\ge 0=F_\nu(0)$). Thus, $\alpha_{kl}=\tilde{\alpha}_{kl}=0$.
  
 \noindent   $\bullet$ If $F_\nu(l/N)\le F_\mu((k-1)/N)$, then from~\eqref{eqn:P1_EqMoments2}, we have also $F_{\tilde{\nu}}(l/N)\le F_{\tilde{\mu}}((k-1)/N)$, and using~\eqref{eqn:P1_EqMoments} we get for $k<l$
    $$\alpha_{kl}=\int_{T^N_k}\int_{T^N_l}F_\mu(x)-F_\nu(y) \dd x \dd y=\int_{T^N_k}\int_{T^N_l}F_{\tilde{\mu}}(x)-F_{\tilde{\nu}}(y)\dd x \dd y=\tilde{\alpha}_{kl}.$$
    For $k=l$, we have by using~\eqref{eqn:P1_EqMoments} and Lemma~\ref{lem:cvgP1_W2} for the inequality
    \begin{align*}
      \alpha_{kk}&=\int_{T^N_k} \left(\int_{\frac{k-1}{N}}^xF_\mu - \int_x^{\frac{k}{N}}F_\nu\right)\dd x\\
      &=\int_{T^N_k} \left(\int_{\frac{k-1}{N}}^xF_\mu + \int_{\frac{k-1}{N}}^xF_\nu\right)\dd x -\frac 1N \int_{T^N_k}F_\nu \\
      &\le \int_{T^N_k} \left(\int_{\frac{k-1}{N}}^xF_{\tilde{\mu}} + \int_{\frac{k-1}{N}}^xF_{\tilde{\nu}}\right)\dd x -\frac 1N \int_{T^N_k}F_{\tilde{\nu}} + \frac{\|\rho_\mu\|_\infty+\|\rho_\nu\|_\infty}{6N^3} \\
      &= \tilde{\alpha}_{kk}+ \frac{\|\rho_\mu\|_\infty+\|\rho_\nu\|_\infty}{6N^3}.
    \end{align*}
    
 \noindent    $\bullet$ We now consider the case where $ F_\mu(k/N)> F_\nu((l-1)/N)$ and $F_\nu(l/N)> F_\mu((k-1)/N)$. We can thus find $x_0\in T^N_k$ and $y_0\in T^N_l$ such that $F_\mu(x_0)=F_\nu(y_0)$. We then have $\forall x\in T^N_k,y\in T^N_l, |F_\mu(x)-F_\nu(y)|\le|F_\mu(x)-F_\mu(x_0)|+|F_\nu(y_0)-F_\nu(y)|\le \|\rho_\mu\|_\infty |x-x_0|+ \|\rho_\nu\|_\infty |y-y_0|$, and thus using that $\int_{T^N_k}|x-x_0|\dd x \le \frac{1}{2N^2}$, 
    $$\alpha_{kl}\le \frac{\|\rho_\mu\|_\infty+\|\rho_\nu\|_\infty}{2N^3}\le \tilde{\alpha}_{kl}+\frac{\|\rho_\mu\|_\infty+\|\rho_\nu\|_\infty}{2N^3}.$$
    For $1\le k \le N$, we note $\{l_k,l_k+1,,\dots,l_k+n_k-1\}\subset \{1,\dots,N\}$ the set of~$l$ such that
  $F_\mu((k-1)/N)<F_{\nu}(l/N)$ and $F_\mu(k/N)> F_\nu((l-1)/N)$. We necessarily have $l_{k+1}\ge  l_k+n_k-1$ since $F_\nu((l_k+n_k-2)/N)<F_\mu(k/N)<F_\nu(l_{k+1}/N)$. Therefore, there is at most one element overlap between two consecutive sets, and thus $\sum_{k=1}^N n_k\le 2N$.

  Combining all cases, and taking into account the contribution of the symmetric term $[F_\nu(x)-F_\mu(y)]^+$ in the integral, we finally get
  \begin{align*}
    W_2^2(\mu,\nu)&\le W_2^2(\tilde{\mu},\tilde{\nu})+2\left(N \frac{\|\rho_\mu\|_\infty+\|\rho_\nu\|_\infty}{6N^3}+2N \frac{\|\rho_\mu\|_\infty+\|\rho_\nu\|_\infty}{2N^3} \right),
  \end{align*}
  which gives~\eqref{eq:toprove_2}
\end{proof}

\section{Numerical algorithms to approximate optimal transport problems}\label{sect:algoMkngUseOfPropTchack}
This section presents the implementation of two algorithms for the approximation of the Optimal Transport cost. Both algorithms rely on  Theorem~\ref{prop:approxProblmDiscreteMeasGeneral}, i.e. that the optimum of the MCOT problem is attained by a finite discrete measure $\sum_{k=1}^{2N+2}p_k\delta_{(x_k,y_k)}$. The two algorithms corresponds to the following choices:
\begin{enumerate}
\item piecewise constant test functions,
\item (regularized) piecewise linear test functions.
\end{enumerate}
In the first case, the precise positions $(x_k,y_k)$ are useless to satisfy the moment constraints: only matters in which cell $(x_k,y_k)$ belongs. Thus the optimization problem is essentially discrete on a (large) finite space, for which Metropolis-Hastings algorithms are relevant. In the second case, we implement a penalized gradient algorithm to optimize the positions $(x_k,y_k)$ and the weights~$p_k$.

The goal of these numerical tests is only illustrative to see the potential relevance of this approach. We do not claim that these algorithms are more efficient than other existing methods in the literature, and the improvement of our algorithms is left for future research.

\subsection{Metropolis-Hastings algorithm on a finite state space}

We expose in the following the principles of the Metropolis-Hastings algorithm used to compute
an approximation of the OT cost. For simplicity, we do so in the case of two uni-dimensional marginal laws. However, the algorithm principles can be adapted to solve a Multimarginal MCOT
problem with marginal laws defined on spaces of any finite dimension.

\subsubsection{Description of the algorithm}

For this algorithm, we consider the framework of Subsection~\ref{subsec_P0}, i.e. $N$ piecewise constant functions $\phi^N_m=\mathbf{1}_{T^N_m}$, $1\le m\le N$. and the MCOT problem~\eqref{eqn:P0MCOT}. As mentioned above, if $(x_k,y_k)$ belongs to the cell $T^N_i\times T^N_j$, its position in this cell is useless regarding the moment constraints. We can therefore assume that the position minimizes the cost in this cell. For $c(x,y)=|y-x|^2$, this amounts to take
\begin{equation*}
   c\left(x_k,y_k\right) = \tilde{c}(i,j) \text{ with } \tilde{c}(i,j)= \left\{
    \begin{array}{lll}
      c\left(\frac{i}{N}, \frac{j+1}{N}\right) & \text{if} & i>j \\
      c\left(\frac{i}{N}, \frac{j}{N}\right)   & \text{if} & i=j \\
      c\left(\frac{i+1}{N}, \frac{j}{N}\right) & \text{if} & i<j.
    \end{array}
  \right.
\end{equation*}

We consider then $2N +2$ distinct cells $T^N_{i_k}\times T^N_{j_k}$, $k\in\{1,\dots, 2N+2\}$. The weights associated to each cell is determined as the solution of the linear optimization of the cost associated
under the constraint that the weights satisfy the moments constraints:
\begin{equation}
  (p_1, ..., p_{2N+2})=
    \argmin_{\substack{
      p_k \ge 0, \
      \sum_{k=1}^{2N+2} p_k = 1 \\
      \forall 1 \le m \le N, \, \sum_{k=1}^{2N+2} p_k \mathbf{1}_{i_k=m} = \mu_m \\
      \forall 1 \le n \le N, \, \sum_{k=1}^{2N+2} p_k \mathbf{1}_{j_k=n} = \nu_n
    }}
    \sum_{k=1}^{2N+2} p_k \tilde{c}\left(i_k, j_k\right).
\end{equation}
Note that this set of constraints may be void. To start with an initial configuration $(i_k,j_k)_{1\le k\le 2N+2}$ that allows the existence of weights which satisfy the constraints, we use the inverse transform sampling between the distributions given by $(\mu_k)_{1\le k\le N}$ and  $(\nu_k)_{1\le k\le N}$ on $\{1,\dots,N\}$. This gives in fact the optimal solution $(p_k,(i_k,j_k))_{1\le k\le 2N+2}$ for~\eqref{eqn:P0MCOT} that satisfy thus in particular the constraints. Since we want here to test the relevance of the Metropolis-Hastings algorithm in this framework, we do not want to start from the optimal solution: thus, we consider a random permutation~$\sigma$ on $\{1,\dots,N\}$ and then the inverse transform sampling between the distributions given by $(\mu_k)_{1\le k\le N}$ and  $(\nu_{\sigma(k)})_{1\le k\le N}$ on $\{1,\dots,N\}$. This gives a configuration that satisfy the constraints and is not a priori optimal.

We now have to specify how the Markov chain defining the algorithm moves from one state $(i_k,j_k)_{1\le k\le 2N+2}$ to another. We denote by $N(i_k,j_k)=\{(i_k+1,j_k),(i_k-1,j_k),(i_k,j_k+1),(i_k,j_k-1)\}$ the neighboring cells of $(i_k,j_k)$ and $$FN(i_k,j_k)=N(i_k,j_k)\cap(\{1,\dots,N\}^2 \setminus (\cup_{k'\not=k}\{(i_{k'},j_{k'} )\}),$$
the  neighboring cells that are free, i.e. that are not in the current configuration. 
We choose randomly and uniformly a cell $l\in\{1,\dots, 2N+2\}$. If $FN(i_l,j_l)=\emptyset$, we pick randomly another one. This rejection method amounts to choose randomly  and uniformly a cell $l$ among those such that $FN(i_l,j_l)\not =\emptyset$. Then, we select $(i'_{l},j'_{l})$ uniformly on $FN(i_l,j_l)$ and set $(i'_k,j'_k)=(i_k,j_k)$ for $k\not= l$, and we accept the new configuration $(i'_k,j'_k)_{1\le k \le 2N+2}$ only if it allows to satisfy the constraints  and with an acceptance ratio described in Algorithm~\ref{algo:MetropolisHastings}. In practice, we run this Algorithm with $K\ge 2N+2$ cells, in order to increase the chance that the new configuration is compatible with the constraints. 

\begin{algorithm}[htp]
  \caption{Metropolis-Hastings algorithm \label{algo:MetropolisHastings}}
  \begin{algorithmic}
    \STATE Fix a temperature $\beta \in \Reel^+$ and take $2N+2\le K\le N^2$.
    \STATE Initialize cells $(i_k, j_k)_{1 \le k \le K}$ and compute the actual optimal cost $c_{\mathrm{actual}}= \sum_{k=1}^{K} p_k \tilde{c}\left(i_k, j_k\right)$.
    \FOR{a given number of steps}     
      \STATE Choose randomly a particle $1 \le l \le K$ such that $FN(i_l,j_l)\not = \emptyset$.
      \STATE Compute $n_{\mathrm{actual}}=\text{Card}(FN(i_l,j_l))$ the number of free cells near $(i_l, j_l)$.
      \STATE Choose randomly a new cell  $(i'_{l},j'_{l})$ in $FN(i_l,j_l)$.
      \IF{the configuration $(i'_k,j'_k)_{1\le k \le 2N+2}$ allows to satisfy the constraints}
        \STATE Compute $c_{\mathrm{newpos}}$ the optimal cost associated to the  configuration $(i'_k,j'_k)_{1\le k \le 2N+2}$.
        \STATE Compute $n_{\mathrm{newpos}}$, the number of free cells near $(i'_l,j'_l)$ in the new configuration.
        \STATE Move the particle $l$ in $(u,v)$ with probability 
          $\min \left( 1, \displaystyle\frac{e^{- c_{\mathrm{newpos}}/\beta}}{e^{- c_{\mathrm{actual}}/\beta}}
          \frac{n_{\mathrm{actual}}}{n_{\mathrm{newpos}}}\right)$. This probability is the acceptance ratio of the Metropolis-Hastings algorithm,
          as explained in Section 2.2 of \cite{delmas2006modles}.
          \STATE Update the value of $c_{\mathrm{actual}}$ to $c_{\mathrm{newpos}}$ if the move is accepted.
     \ENDIF
    \ENDFOR
    \RETURN the lowest cost encountered throughout the loop.
  \end{algorithmic}
\end{algorithm}

The state space of the Markov Chain describing Algorithm~\ref{algo:MetropolisHastings} is the set of $K$ distinct elements of $\{1,\dots,N\}^2$. Note that we can go from any points $(i,j)$ to $(i',j')$ with at most $2N-2$ moves (a move consists in adding or removing one to one of the coordinate). If we ignore the problem of satisfying the constraints, we can therefore go from a configuration  $(i_k,j_k)_{1\le k \le 2N+2}$ to another one $(i'_k,j'_k)_{1\le k \le 2N+2}$ with at most $K(2N-2)$ moves, which let think that the Doeblin condition may be satisfied. This would ensure theoretically the convergence of the algorithm converges towards the infimum
\begin{equation}
  \inf_{\substack{\pi \in \Pi(\mu,\nu;(\phi^N_m)_{1\le m \le N},(\phi^N_n)_{1\le n \le N})   }}
  \int_0^1 \int_0^1 c(x,y) \dd \pi (x,y),
\end{equation}
and that the convergence is exponentially fast (see e.g. Section 2 of \cite{delmas2006modles}).

\subsubsection{Numerical examples}

We tested the algorithm for the marginal laws with probability  density functions
\begin{equation}
  \rho_\mu(x) = 3x^2, \quad \rho_\nu(y) = 2-2y.
\end{equation}
We consider a number of particles 
$K = 3N +2$ in order at each step to have more freedom among the configurations
which satisfy the constraints.
We present two minimizations:
\begin{itemize}
  \item $N = 20$ and $\beta = 0.000075$
  \item $N = 60$ and $\beta = 0.00002$.
\end{itemize}

The evolution of the configurations through the iterations are represented
for $N = 20$ and $N = 60$ in Figure \ref{fig:configsP0N20} and \ref{fig:configsP0N60}.
The darker the cell, the more weight it has. In green (Figures \ref{fig:configsP0N20Optim} and \ref{fig:configsP0N60Optim})
are represented the optimal configuration for the given number of moment constraints.
The convergence of the numerical cost for each minimization is represented in Figure \ref{fig:costCvP0}.
The pink line represents the cost of the Optimal Transport problem we approximate,
the dark blue line the one of the cost of the current configuration and the light blue one
the minimum numerical cost encountered during the minimization. The green line
is the cost of the optimal configuration for the given number of moment constraints, that we aim to compute.

\begin{figure}[htp]
  \centering
  \begin{subfigure}{0.48\textwidth}
    \includegraphics[width=\textwidth]{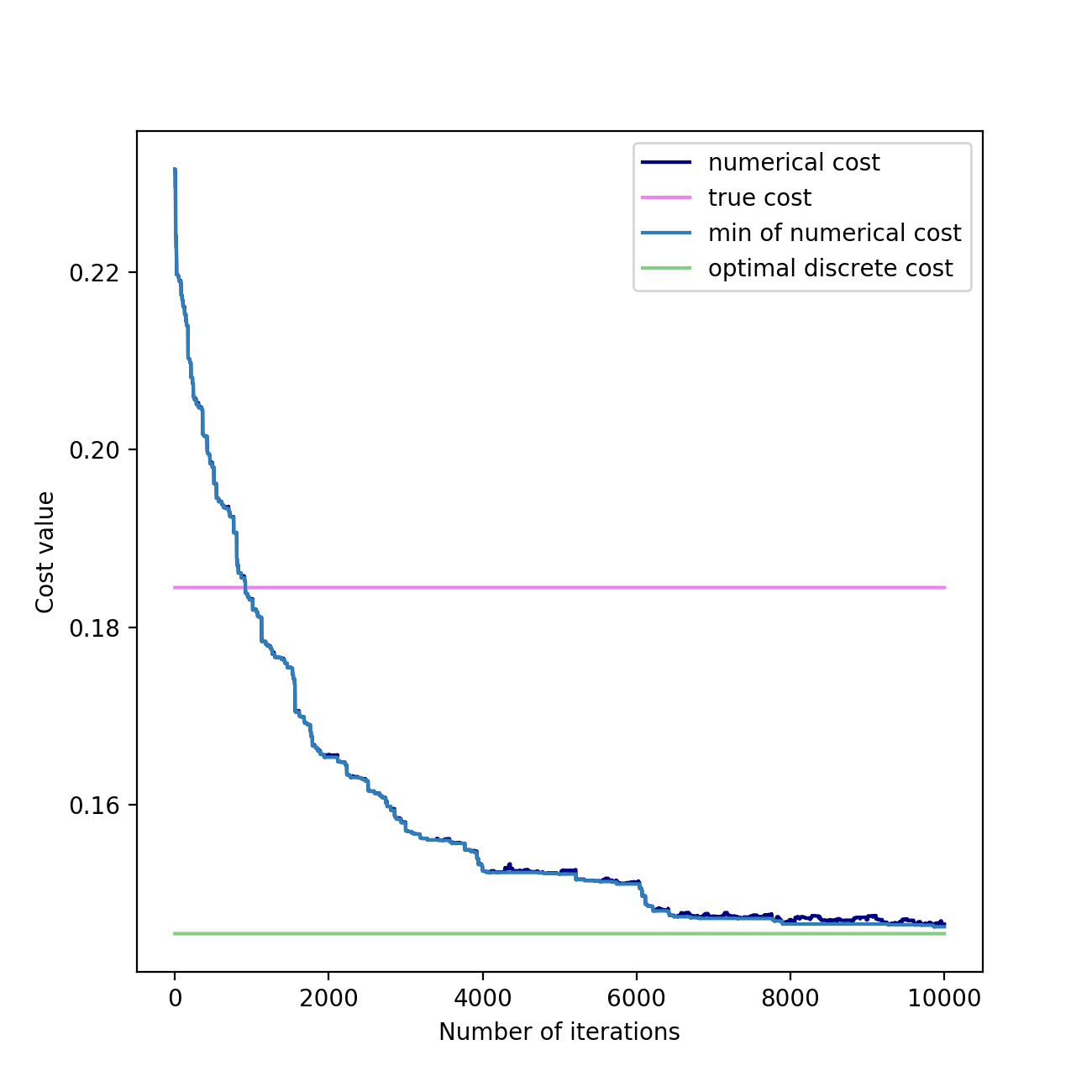}
    \caption{$N = 20$}
  \end{subfigure}
  \begin{subfigure}{0.48\textwidth}
    \includegraphics[width=\textwidth]{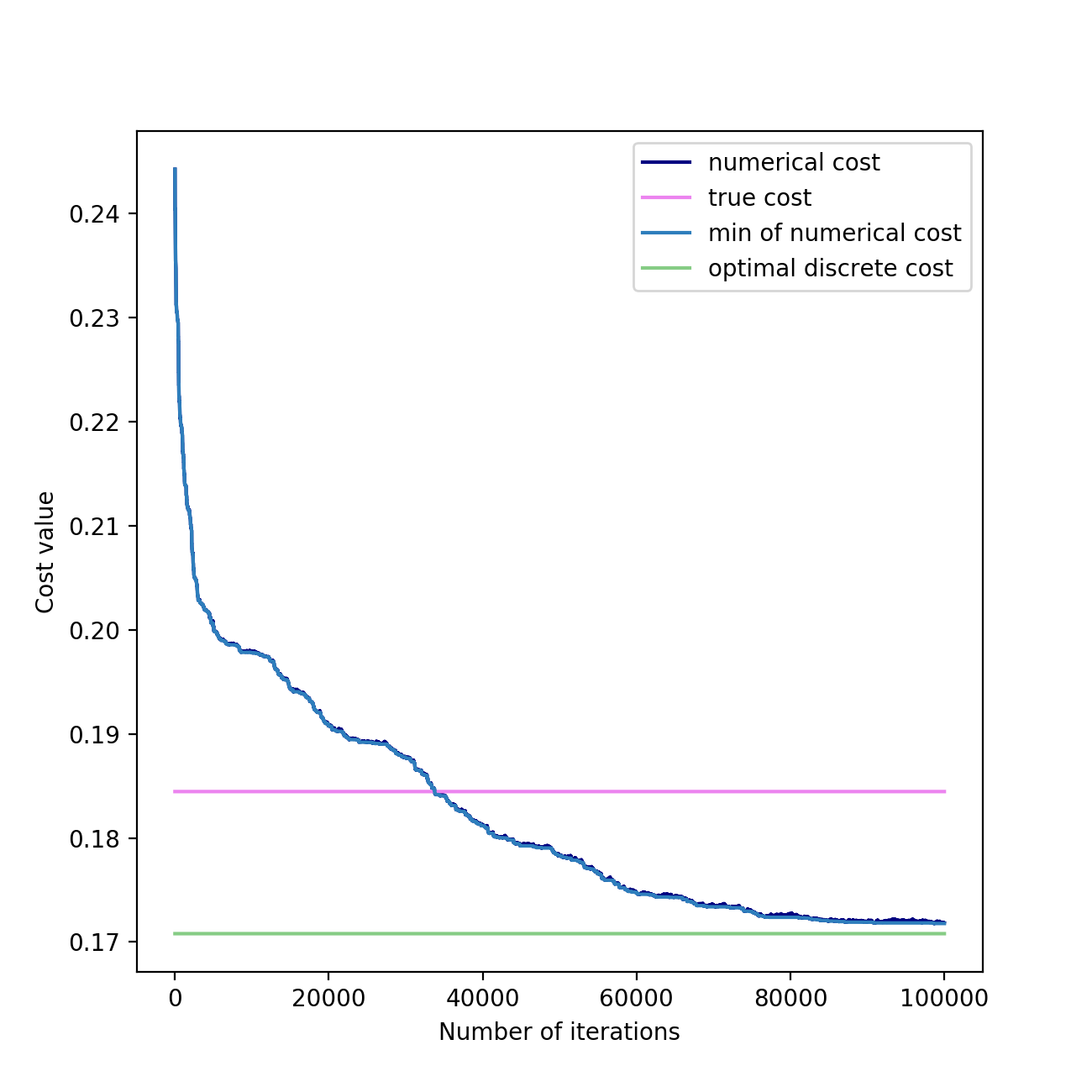}
    \caption{$N = 60$}
  \end{subfigure}
  \caption{Cost convergence \label{fig:costCvP0}}
\end{figure}

\begin{figure}[htp]
  \centering
    \begin{subfigure}[h]{0.32\textwidth}
      \includegraphics[width=\textwidth]{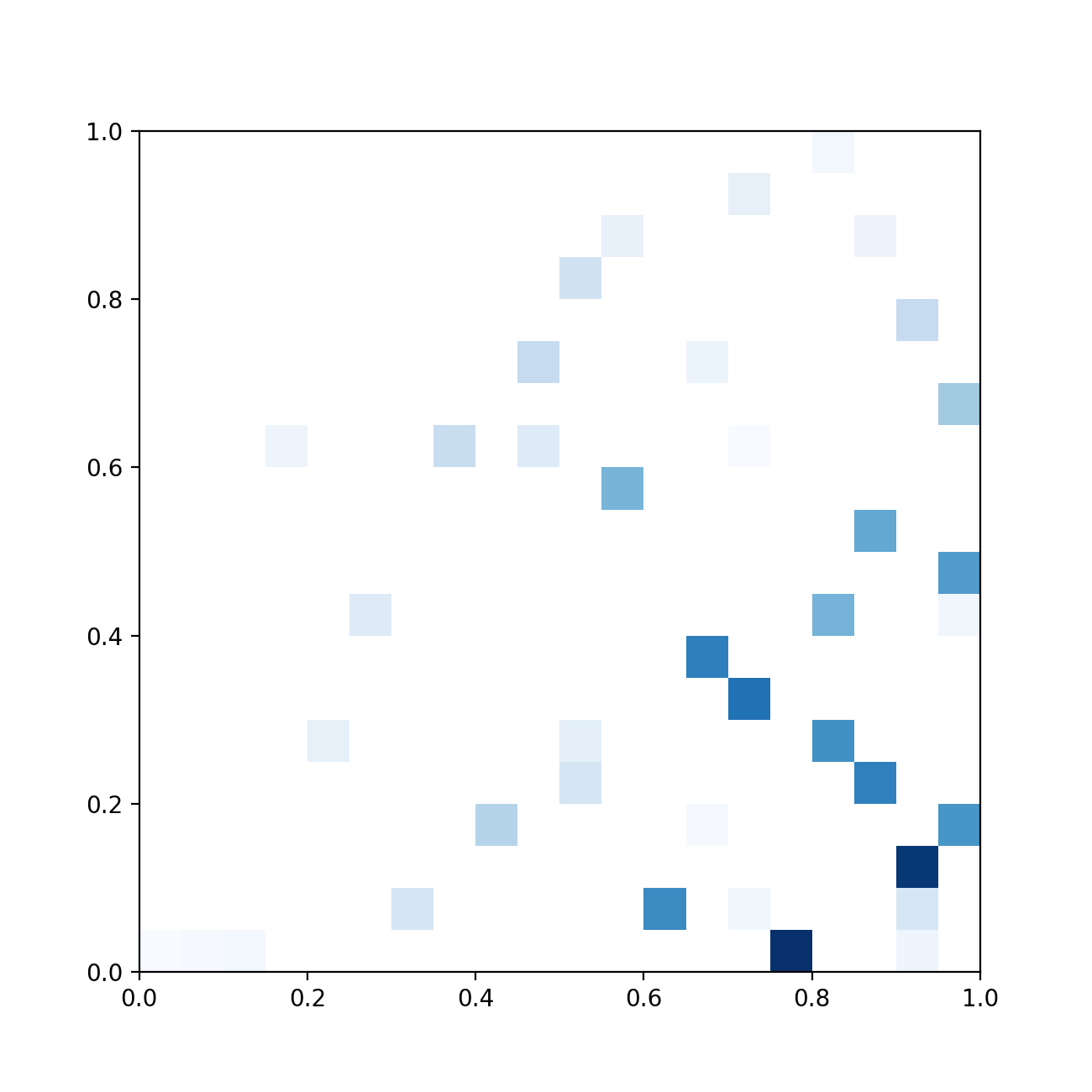}
      \caption{iteration 0}
    \end{subfigure}
    \begin{subfigure}[h]{0.32\textwidth}
      \includegraphics[width=\textwidth]{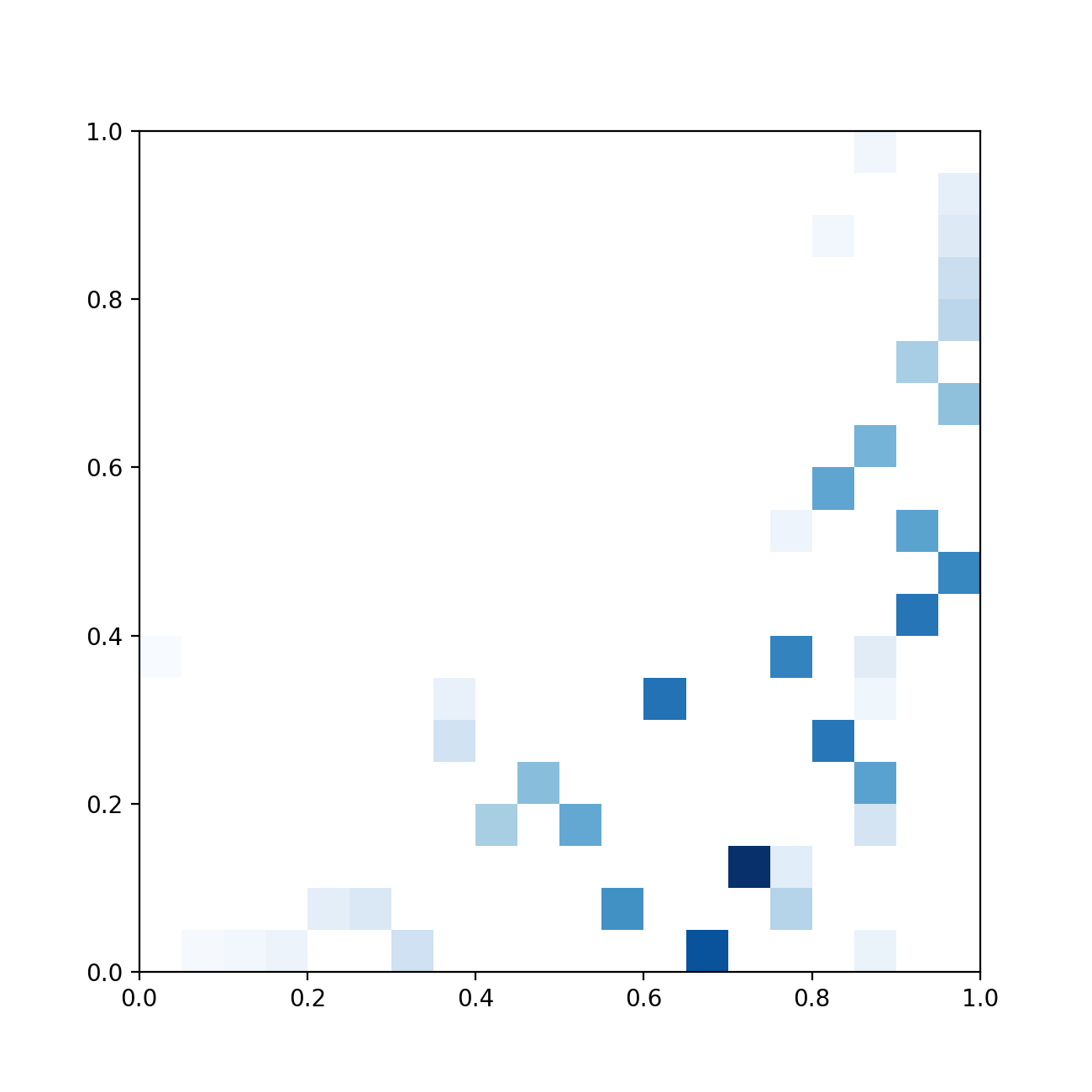}
      \caption{iteration 2000}
    \end{subfigure}
    \begin{subfigure}[h]{0.32\textwidth}
      \includegraphics[width=\textwidth]{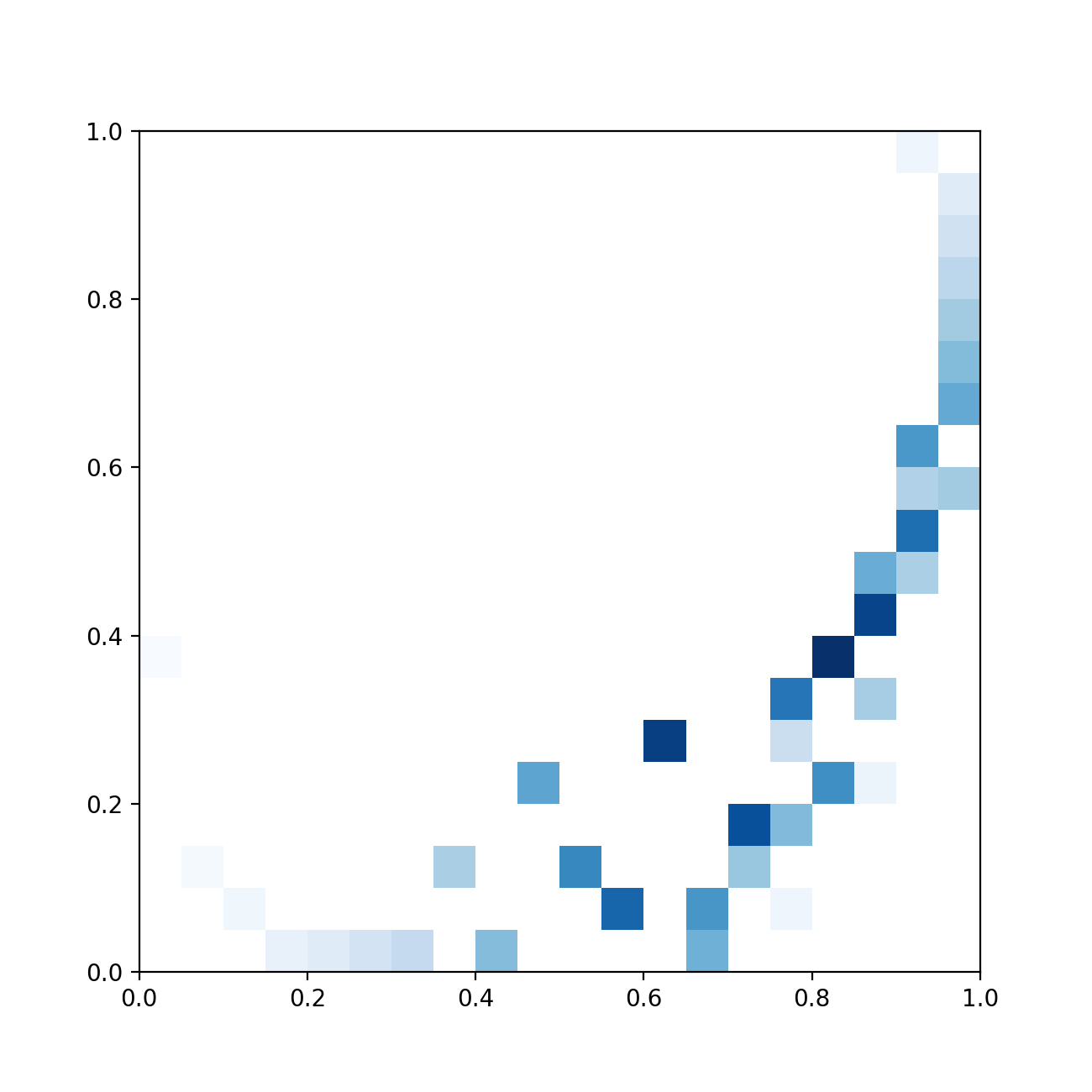}
      \caption{iteration 4000}
    \end{subfigure}
    \begin{subfigure}[h]{0.32\textwidth}
      \includegraphics[width=\textwidth]{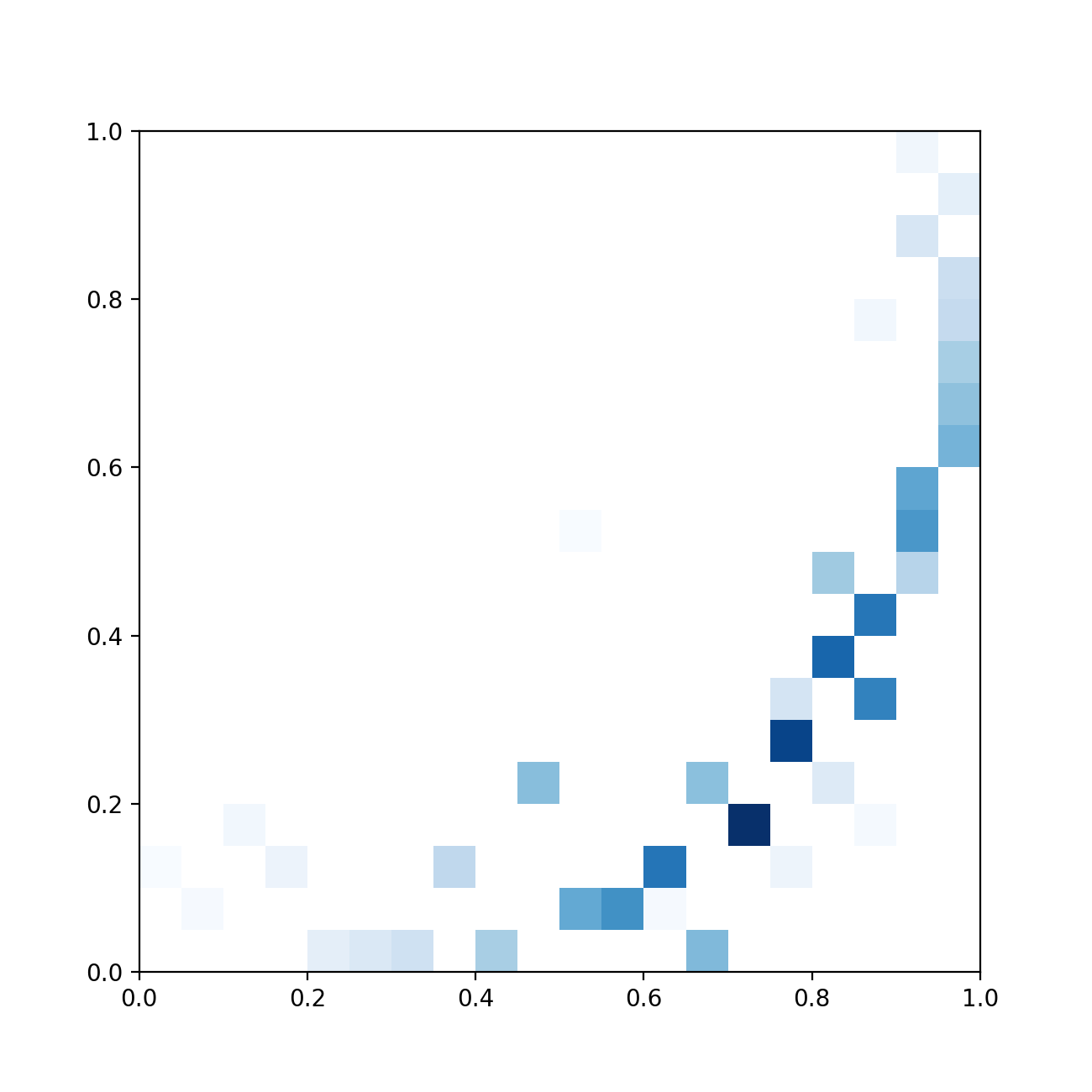}
      \caption{iteration 6000}
    \end{subfigure}
    \begin{subfigure}[h]{0.32\textwidth}
      \includegraphics[width=\textwidth]{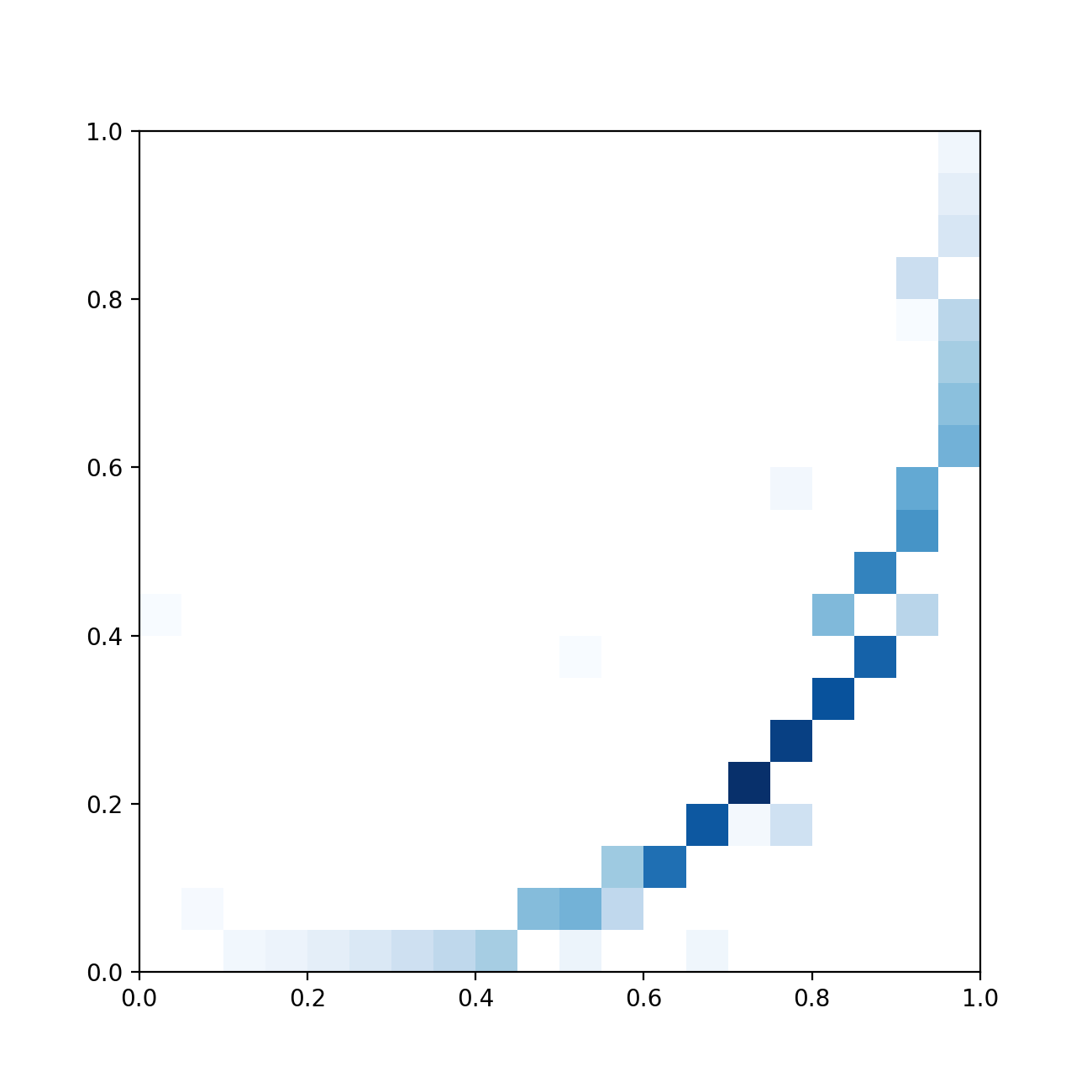}
      \caption{iteration 10000}
    \end{subfigure}
    \begin{subfigure}[h]{0.32\textwidth}
      \includegraphics[width=\textwidth]{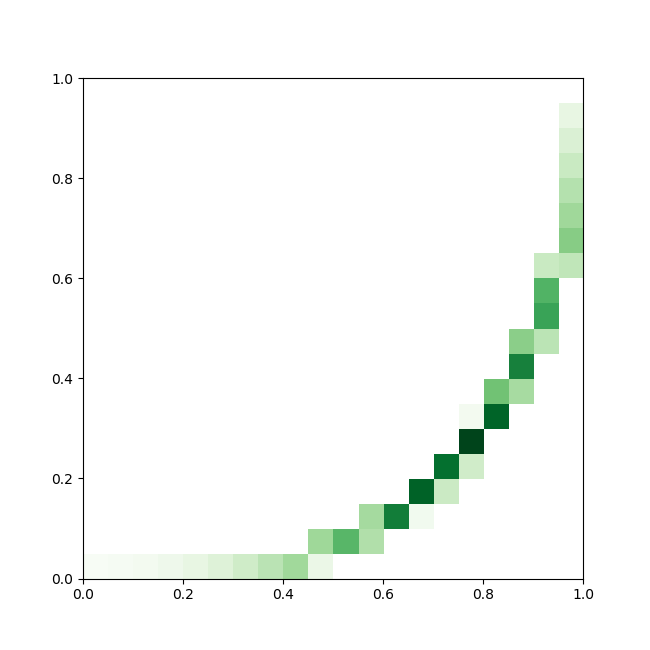}
      \caption{optimal configuration \label{fig:configsP0N20Optim}}
    \end{subfigure}
    \caption{Convergence for two 1D marginal laws with 20 test functions on each set for a quadratic cost \label{fig:configsP0N20}}
\end{figure}

\begin{figure}[htp]
  \centering
    \begin{subfigure}[h]{0.32\textwidth}
      \includegraphics[width=\textwidth]{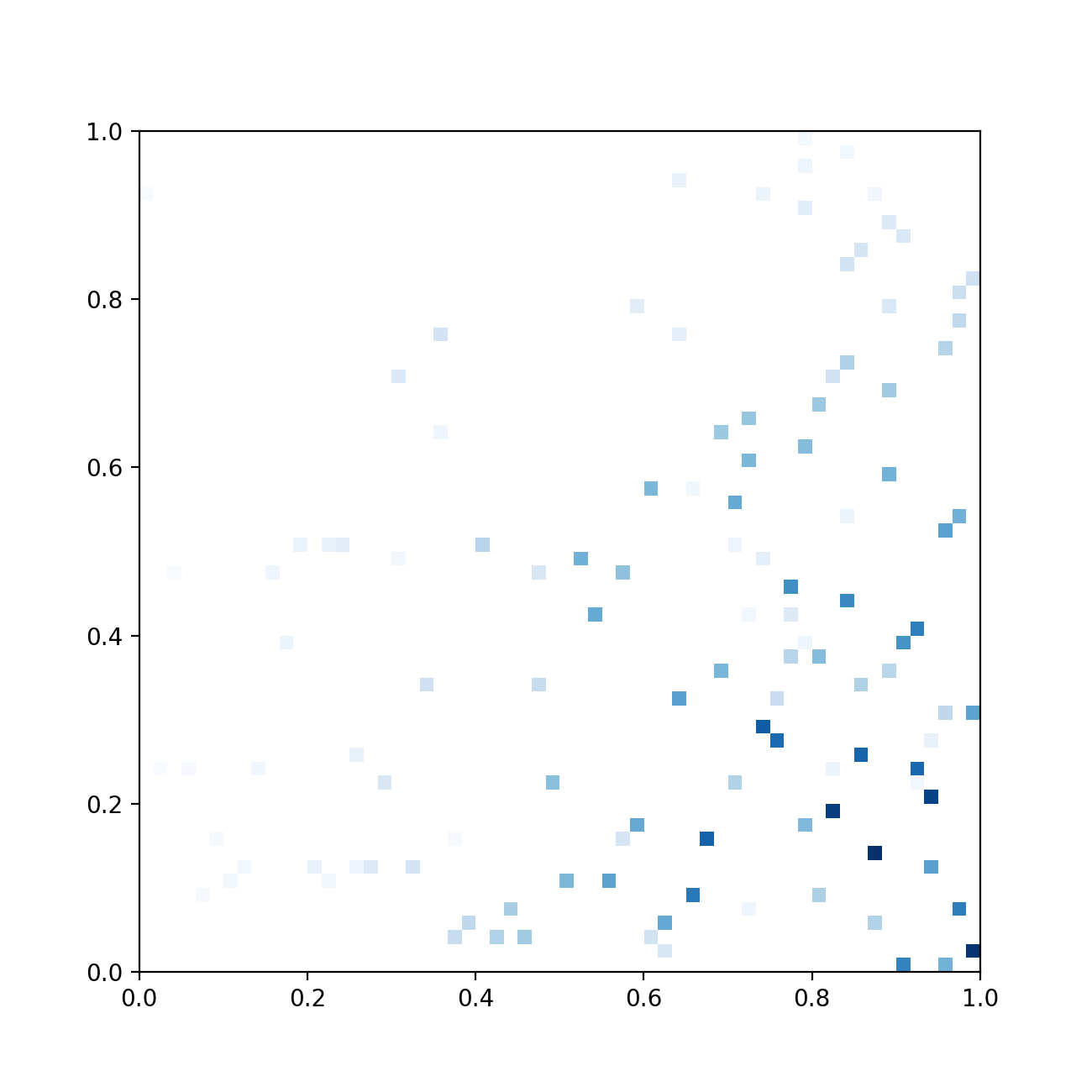}
      \caption{iteration 0}
    \end{subfigure}
    \begin{subfigure}[h]{0.32\textwidth}
      \includegraphics[width=\textwidth]{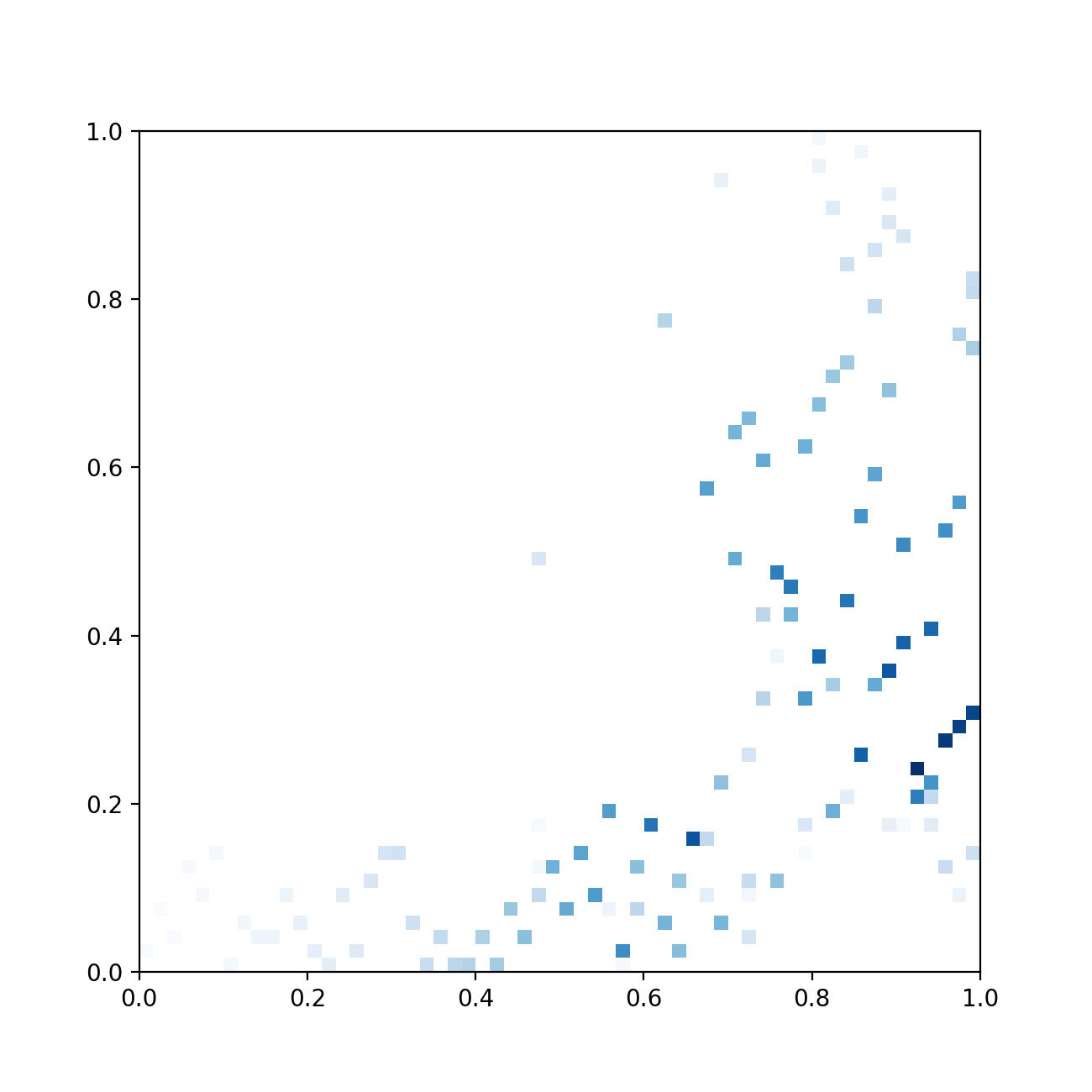}
      \caption{iteration 3000}
    \end{subfigure}
    \begin{subfigure}[h]{0.32\textwidth}
      \includegraphics[width=\textwidth]{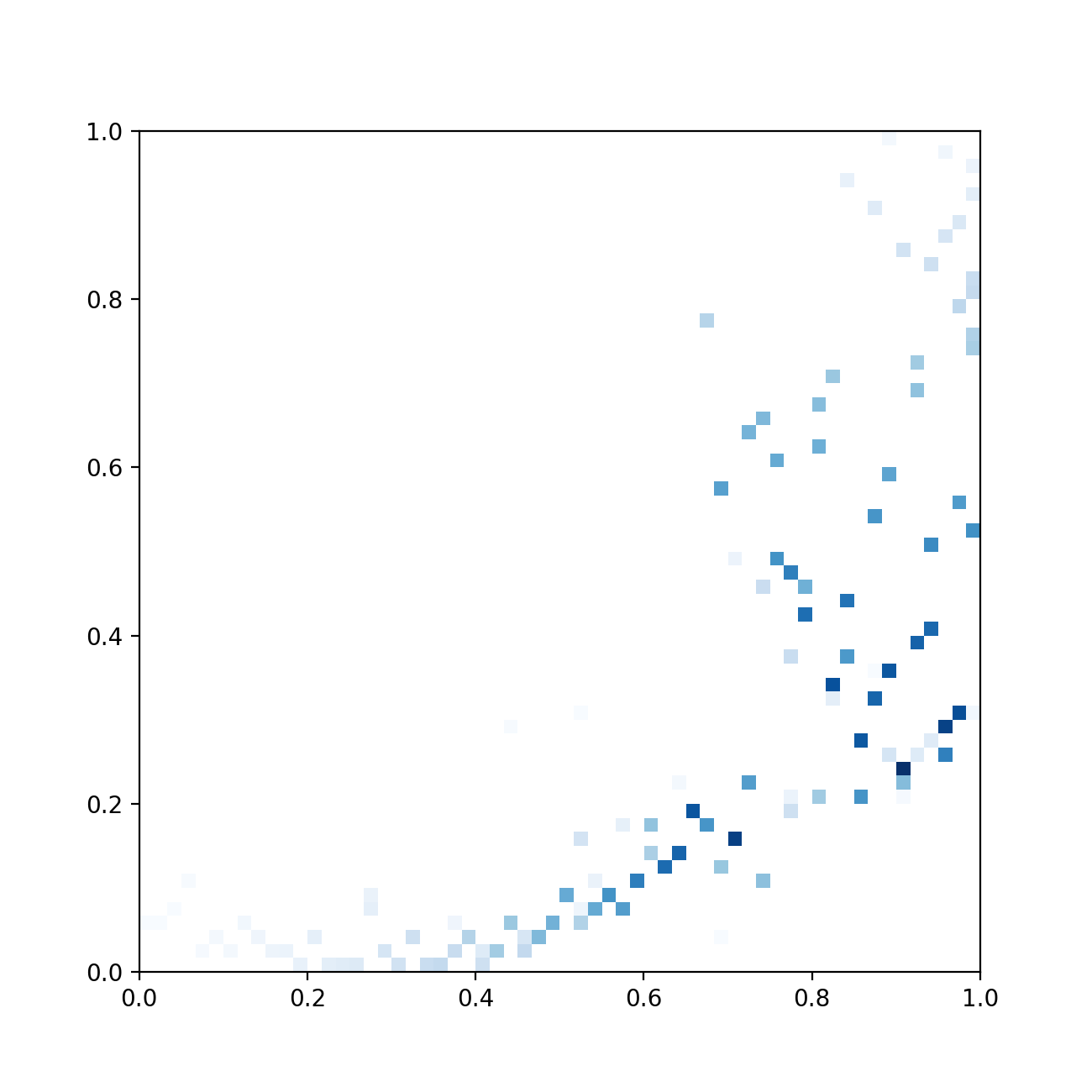}
      \caption{iteration 10000}
    \end{subfigure}
    \begin{subfigure}[h]{0.32\textwidth}
      \includegraphics[width=\textwidth]{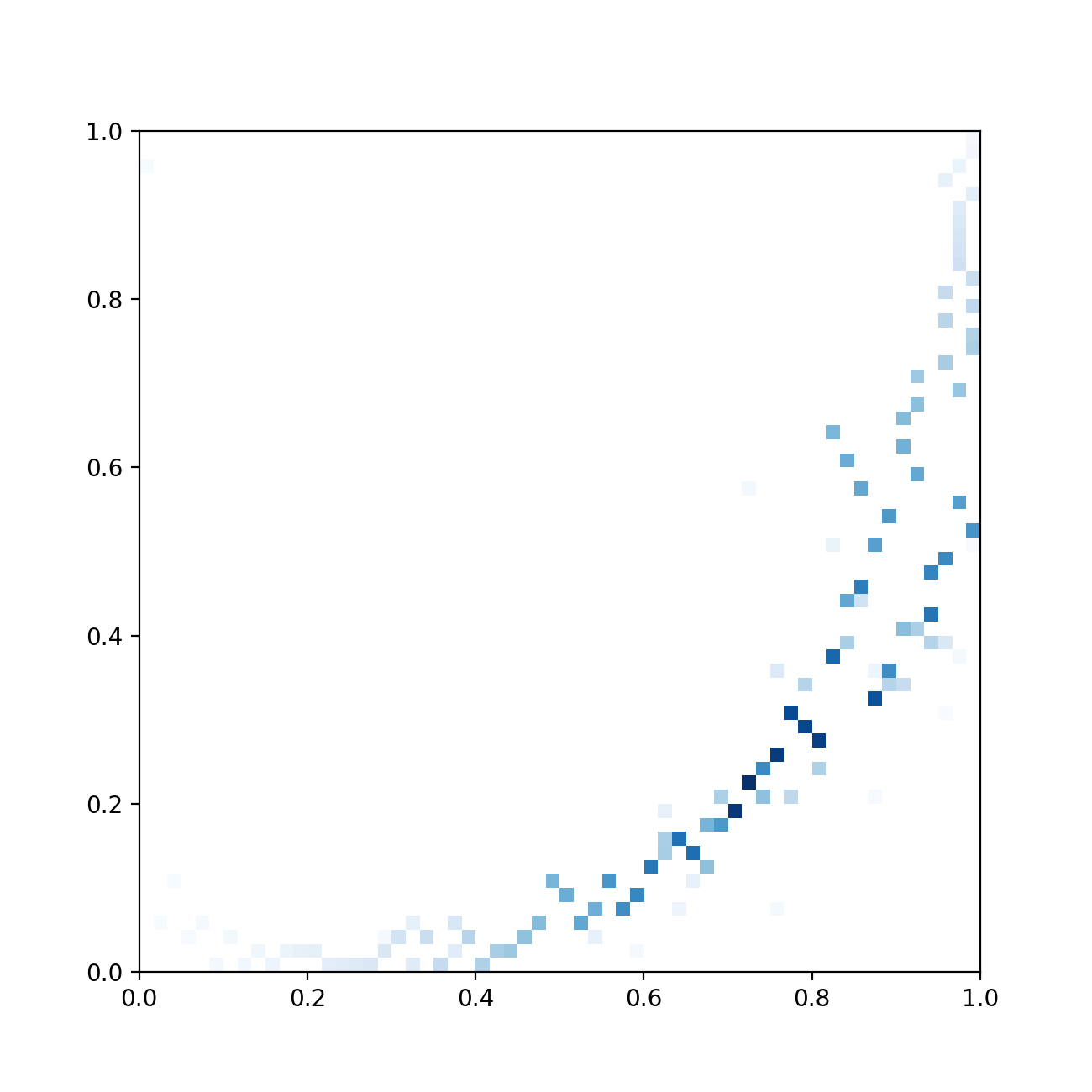}
      \caption{iteration 60000}
    \end{subfigure}
    \begin{subfigure}[h]{0.32\textwidth}
      \includegraphics[width=\textwidth]{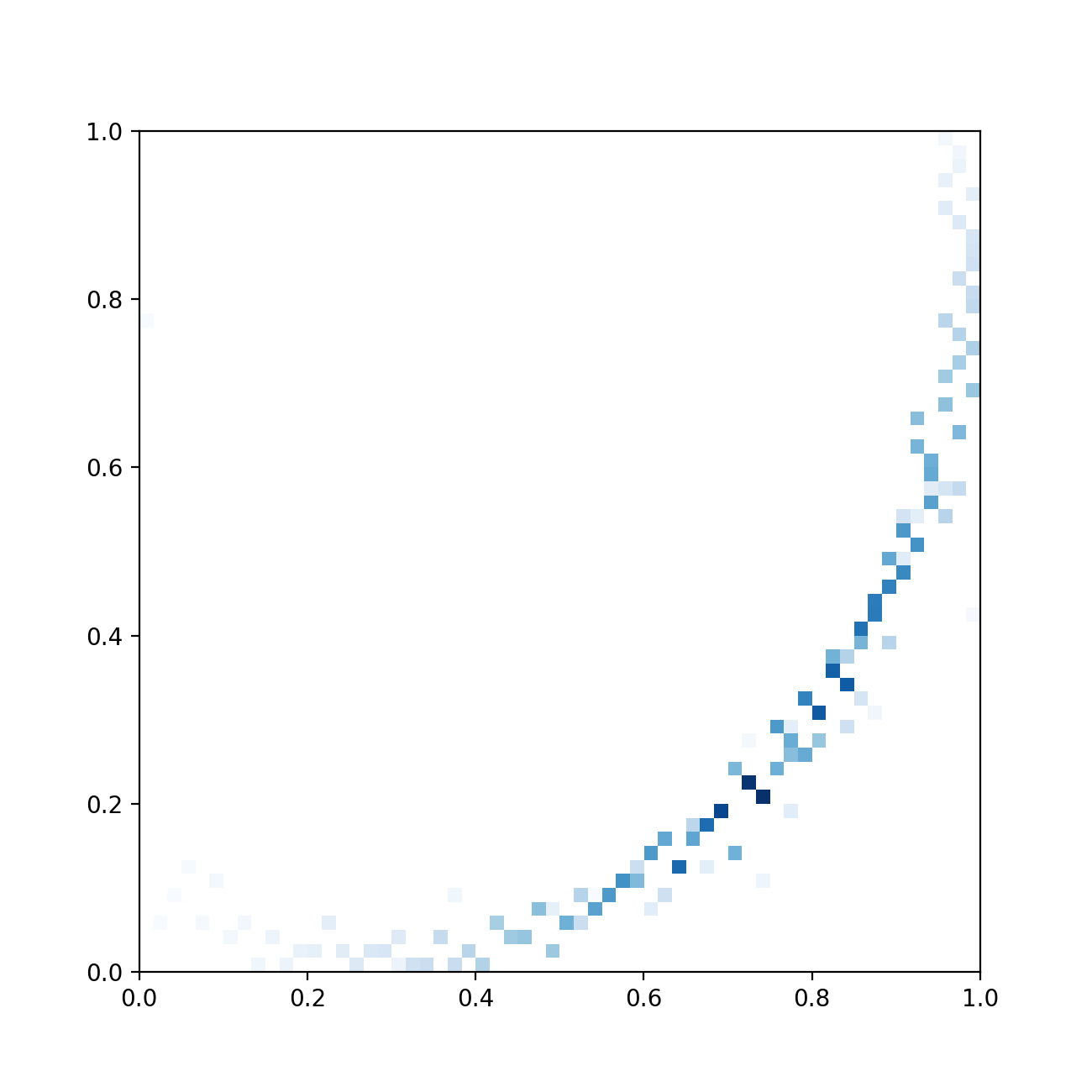}
      \caption{iteration 100000}
    \end{subfigure}
    \begin{subfigure}[h]{0.32\textwidth}
      \includegraphics[width=\textwidth]{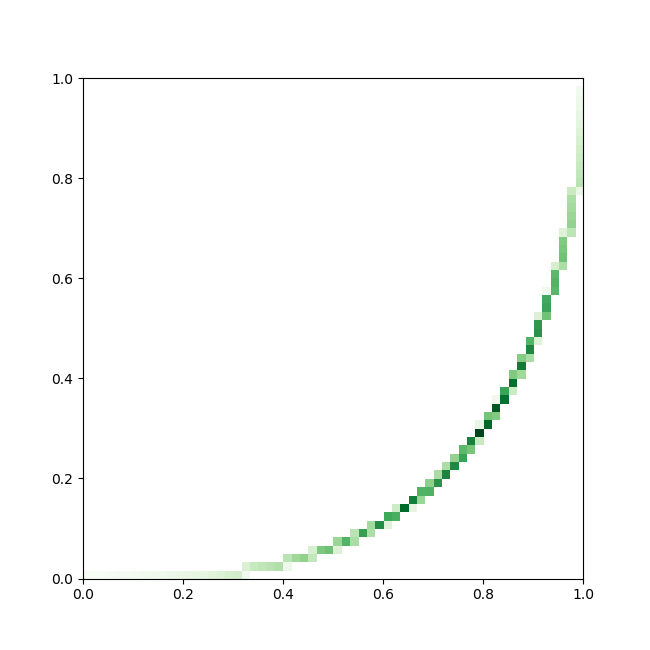}
      \caption{optimal configuration \label{fig:configsP0N60Optim}}
    \end{subfigure}
    \caption{Convergence for two 1D marginal laws with 60 test functions on each set for a quadratic cost \label{fig:configsP0N60}}
\end{figure}

\subsection{Gradient on a penalized functional \label{sect:numericsPenalizedGradient}}

\subsubsection{Principles \label{sect:gradprinciples}}

We make use of Theorem~\ref{prop:approxProblmDiscreteMeasGeneral} by searching optima of the MCOT problem with $N$ test functions on each space by looking for an optimal
probability measure which is finitely supported on at most $2N +2$ points (note that in the multimarginal case, we can look similarly for measures supported on $DN + 2$ points). 
This algorithm consists in penalizing moments constraints of the MCOT problem
for $N$ differentiable test functions on each space
($(\phi_m)_{1 \le m \le N}$ and $(\psi_n)_{1 \le n \le N}$)
and then using a gradient-type algorithm to compute the optimum.

For sake of simplicity, we consider the case of two marginal laws where the cost function~$c$ is assumed to be differentiable. Let us write the position of the $2N + 2$ particles
by $((x_k, y_k))_{1 \le k \le 2N+2}$ and their weights by $(p_k)_{1\le k \le 2N+2}$. Then, it is natural to consider the minimization of
$$  \sum_{k = 1}^{2N+2} p_k c(x_k, y_k)
    + \frac 1 \eta \left(
      \sum_{m = 1}^N \left(
        \sum_{k = 1}^{2N+2} p_k \phi_m(x_k) - \mu_m
      \right)^2 \right. \\ \left.
      + \sum_{n = 1}^N \left(
        \sum_{k = 1}^{2N +2}p_k \psi_n(y_k) - \nu_n
      \right)^2
    \right),$$
for some small parameter $\eta>0$ and under the constraints $p_k\ge0$, $\sum_{k=1}^{2N+2}p_k=1$. To avoid the handling of these latter constraints, we prefer to consider weights $p_k=\frac{e^{a_k}}{\sum_{k=1}^{2N +2} e^{a_k}}$ for some $a_k\in \R$. Since the positions $(x_k,y_k)$ are not assumed to be different from each other, the previous minimization problem is equivalent to minimize
\begin{multline}\label{eq:FfunctionGradientAlgo}
  F \left(x_1, ..., x_{2N+2}, y_1, ..., y_{2N+2}, a_1, ..., a_{2N+2} \right)\\
  = \sum_{k = 1}^{2N+2} \frac{e^{a_k}}{\sum_{l=1}^{2N +2} e^{a_l}} c(x_k, y_k)
    + \frac 1 \eta \left(
      \sum_{m = 1}^N \left(
        \sum_{k = 1}^{2N+2} \frac{e^{a_k}}{\sum_{l=1}^{2N +2} e^{a_l}} \phi_m(x_k) - \mu_m
      \right)^2 \right. \\ \left.
      + \sum_{n = 1}^N \left(
        \sum_{k = 1}^{2N +2} \frac{e^{a_k}}{\sum_{l=1}^{2N +2} e^{a_l}} \psi_n(y_k) - \nu_n
      \right)^2
    \right). 
\end{multline}
For a fixed value of $\eta>0$, we use a projected gradient algorithm (see e.g. Algorithm 1.3.16 of \cite{polak1997optimization}), to ensure that $(x_k,y_k)\in [0,1]^2$ for all $k$, together with a line search method. We implement alternated gradient steps as follows: first, a gradient step is performed on the coefficients $(a_k)_{1\le k\le 2N+2}$  with $(x_k,y_k)_{1\le k\le 2N+2}$ fixed; second, a gradient step is done on the positions $(x_k)_{1\le k\le 2N+2}$ with the other variables fixed; lastly, a gradient step is done on the positions $(y_k)_{1\le k\le 2N+2}$ with the other variables fixed. This procedure is repeated until the norm of the projected gradient is below some error threshold. The convergence of this algorithm is ensured by Wolfe theorem (see Theorem 1.2.21 of \cite{polak1997optimization}).

The example computations exposed thereafter use regularized continuous piecewise affine test functions. Remark that we do not use discontinuous piecewise affine test functions, for which we have rates of convergence for $W_1$ and $W_2$. We make this choice because the gradient algorithm that we describe above has better numerical properties for continuously differentiable test functions.

In the MCOT formulation~\eqref{eqn:MCOTDef} with $M=N$, minimizers of MCOT problems are the same if we consider test functions $(\bar{\phi}_m)_{1\le m\le N}$ and $(\bar{\psi}_m)_{1\le m\le N}$ such that ${\rm Span}((\bar{\phi}_m)_{1\le m\le N})={\rm Span}(({\phi}_m)_{1\le m\le N})$ and ${\rm Span}((\bar{\psi}_m)_{1\le m\le N})={\rm Span}(({\psi}_m)_{1\le m\le N})$. However, in the penalized version of the problem~\eqref{eq:FfunctionGradientAlgo}, the choice of the test functions has a strong impact on the convergence of the gradient algorithms. It appears that considering positive part functions (which are convex functions) greatly improves the efficiency of the procedure with respect to classical hat functions, even if both spans are identical.

Thus, for the numerical examples in 1D, we use the functions for $\epsilon > 0$ and for all $N \in \N^*$,
\begin{equation}
  \forall x \in [0,1], \quad \varphi^N_0(x) = \left\{
  \begin{array}{ll}
    - \left(x- \frac{1}{N} \right) & \text{if} \quad x - \frac{1}{N} \le - \epsilon \\
    \frac{1}{4\epsilon}\left(x - \frac{1}{N} - \epsilon\right)^2 &\text{if} \quad -\epsilon \le x - \frac{1}{N} \le \epsilon \\
    0 & \text{if} \quad x - \frac{1}{N} \ge \epsilon ,
    \end{array}\right.
\end{equation}
and for all $1 \le m \le N$,
\begin{equation}
  \forall x \in [0,1], \quad \varphi^N_m(x) = \left\{
  \begin{array}{ll}
    0 & \text{if} \quad x - \frac{m-1}{N}\le - \epsilon\\
    \frac{1}{4\epsilon} \left(x - \frac{m-1}{N} + \epsilon\right)^2 & \text{if} \quad -\epsilon \le x - \frac{m-1}{N} \le \epsilon \\
    x  - \frac{m-1}{N} & \text{if} \quad x - \frac{m-1}{N} \ge \epsilon;
    \end{array}\right.
\end{equation}
which are a regularization of the functions, for all $N \in \N^*$, and $1 \le m \le \N$,
\begin{equation*}
  \left( \cdot - \frac{1}{N} \right)^- \quad \text{and} \quad \left( \cdot - \frac{m-1}{N}\right)^+.
\end{equation*}
The vector space spanned by the restriction to~$[0,1]$ of these functions is the same as the one spanned by the classical continuous piecewise affine functions
(i.e. the functions $\psi^N_m$ introduced in Section \ref{sect:cvgSpeedP1UnderEsttoOT}).

For the example in dimension~$2$, for $N \in \N^*$, we use the following $(N + 1)^2$ test  functions defined as follows: for all $1 \le m,n \le N$ and $(x,y) \in [0,1]^2$,
\begin{equation} \label{phi_mn_1}
  \varphi^N_{m,n}(x,y) = \varphi^{2N}_{m+n-1}\left(
    \frac{x + y - \tilde{\varphi}^N_{m-n+1} (x-y) - \tilde{\varphi}^N_{n-m+1} (y-x)}{2}
  \right)
\end{equation}
where for all $q \in \mathbb{Z}$,
\begin{equation} \label{phi_mn_2}
  \forall x \in [0,1], \quad
  \tilde{\varphi}_q^N (x) = \left\{ \begin{array}{ll}
    \varphi_q^N(x) &\text{if} \quad 1 \le q \le N, \\
    0 &\text{otherwise},
  \end{array}\right.
\end{equation}
and
\begin{equation} \label{phi_mn_3}
  \varphi^N_{0,0}(x,y) = \varphi^N_{1,1}\left(\frac{1}{N}-x, \frac{1}{N}-y\right).
\end{equation}
For $1 \le m, n \le N$, we set
\begin{equation}  \label{phi_mn_4}
  \varphi^N_{m,0}(x,y) = \varphi^N_{m,1}\left(x,\frac{1}{N} - y\right) \quad \text{and} \quad \varphi^N_{0,n}(x,y) = \varphi^N_{1,n}\left(\frac{1}{N} - x,y\right).
\end{equation}
Those functions are a regularization of the functions
 $G^N_{m,n}(x,y) = \left(\min\left(x - \frac{m-1}{N}, y-\frac{n-1}{N}\right)\right)^+$ with $1 \le m, n \le N$, $G^N_{0,0}(x,y) = \left(\min\left(\frac{1}{N}- x, \frac{1}{N} - y \right)\right)^+$, $G^N_{0,m}(x,y) = \left(\min\left(x - \frac{m-1}{N}, \frac{1}{N} - y \right)\right)^+$, $G^N_{n,0}(x,y) = \left(\min\left(\frac{1}{N} - x, y-\frac{n-1}{N}\right)\right)^+$. The vector space spanned by the restriction to~$[0,1]^2$ of these functions is the same as the one spanned by the classical continuous piecewise affine functions associated to the mesh illustrated in Figure~\ref{mesh}.

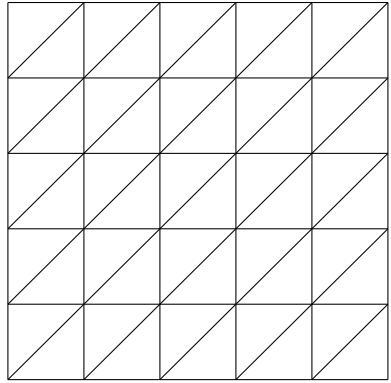
\begin{figure}[htp]
  \centering
      \begin{tikzpicture}
        \draw[-] (0,0) -- (0,5);
        \draw[-] (1,0) -- (1,5);
        \draw[-] (2,0) -- (2,5);
        \draw[-] (3,0) -- (3,5);
        \draw[-] (4,0) -- (4,5);
        \draw[-] (5,0) -- (5,5);

        \draw[-] (0,0) -- (5,0);
        \draw[-] (0,1) -- (5,1);
        \draw[-] (0,2) -- (5,2);
        \draw[-] (0,3) -- (5,3);
        \draw[-] (0,4) -- (5,4);
        \draw[-] (0,5) -- (5,5);

        \draw[-] (0,4) -- (1,5);
        \draw[-] (0,3) -- (2,5);
        \draw[-] (0,2) -- (3,5);
        \draw[-] (0,1) -- (4,5);
        \draw[-] (0,0) -- (5,5);
        \draw[-] (1,0) -- (5,4);
        \draw[-] (2,0) -- (5,3);
        \draw[-] (3,0) -- (5,2);
        \draw[-] (4,0) -- (5,1);        
      \end{tikzpicture}
  \caption{Mesh of  piecewise affine functions  on~$[0,1]^2$.}\label{mesh}
\end{figure}

\subsubsection{1D numerical example}

\paragraph{Convergence of the algorithm}
We tested the algorithm for the marginal laws with densities
\begin{equation}
  \rho_\mu(x) = 3x^2, \quad \rho_\nu(y) = 2-2y,
\end{equation}
the quadratic cost function $c(x,y)=|y-x|^2$ and a fixed penalization coefficient.
The exact optimal transport map between $\mu$ (abscissa) and $\nu$ (ordinate) is represented by the red line on Figures \ref{fig:configs1DN10} and \ref{fig:configs1DN40}.
We present two minimizations:
\begin{itemize}
  \item $N = 10$ and $1/\eta = 100$
  \item $N = 40$ and $1/\eta = 25$.
\end{itemize}
Once each minimization process has converged, the cost for $N = 10$ is 0.17805
and the one for $N = 40$ is 0.17785.
The evolution of the configurations through the iterations are represented
for $N = 10$ and $N = 40$ in Figure \ref{fig:configs1DN10} and \ref{fig:configs1DN40}.
The darker the particle $(x_k,y_k)$, the larger its weight $p_k$. 

And the convergence of the numerical cost for each one in Figure \ref{fig:costCv1D}
the pink line represents the cost of the exact Optimal Transport problem that we approximate.

\begin{figure}[htp]
  \centering
  \begin{subfigure}{0.32\textwidth}
    \includegraphics[width=\textwidth]{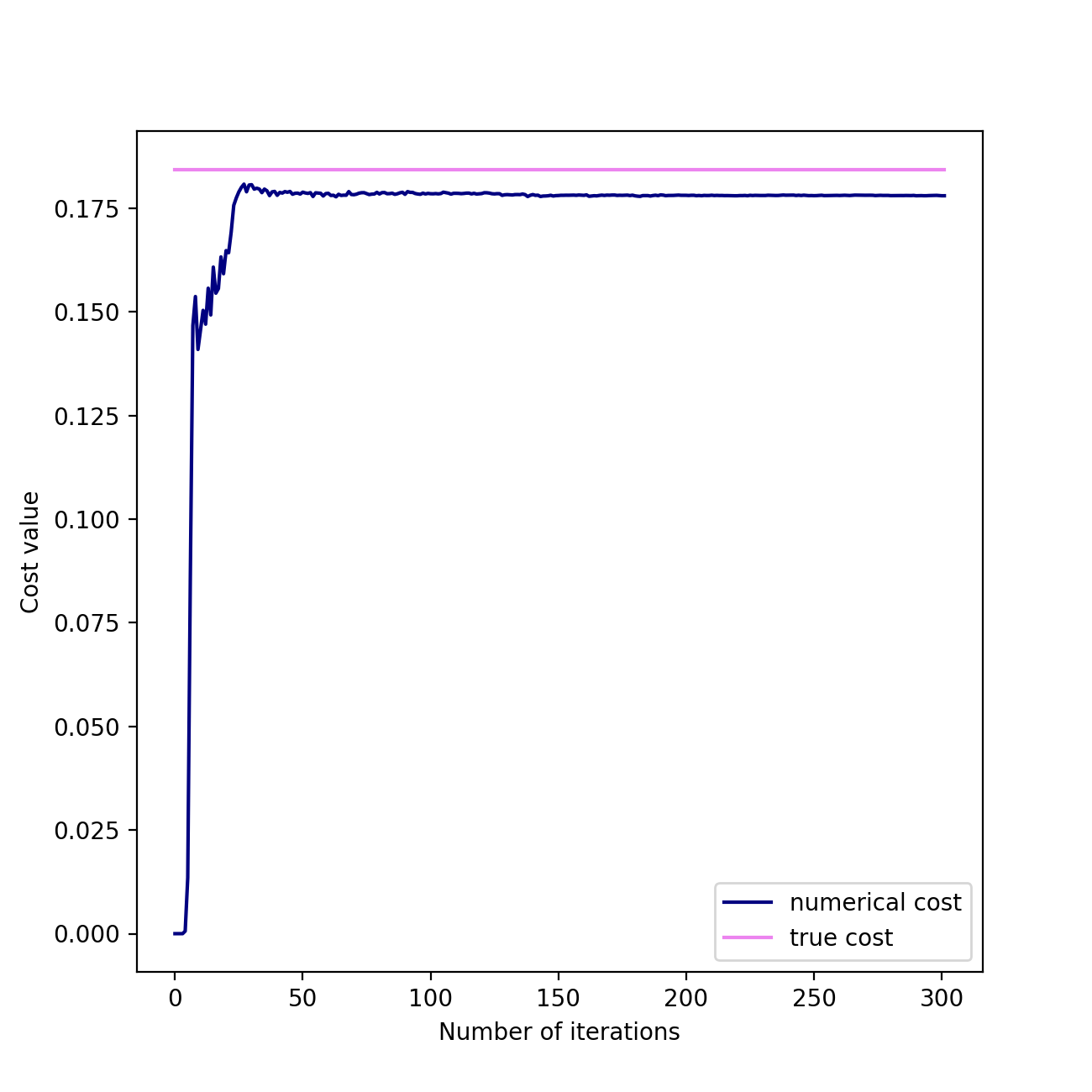}
    \caption{$N = 10$}
  \end{subfigure}
  \begin{subfigure}{0.32\textwidth}
    \includegraphics[width=\textwidth]{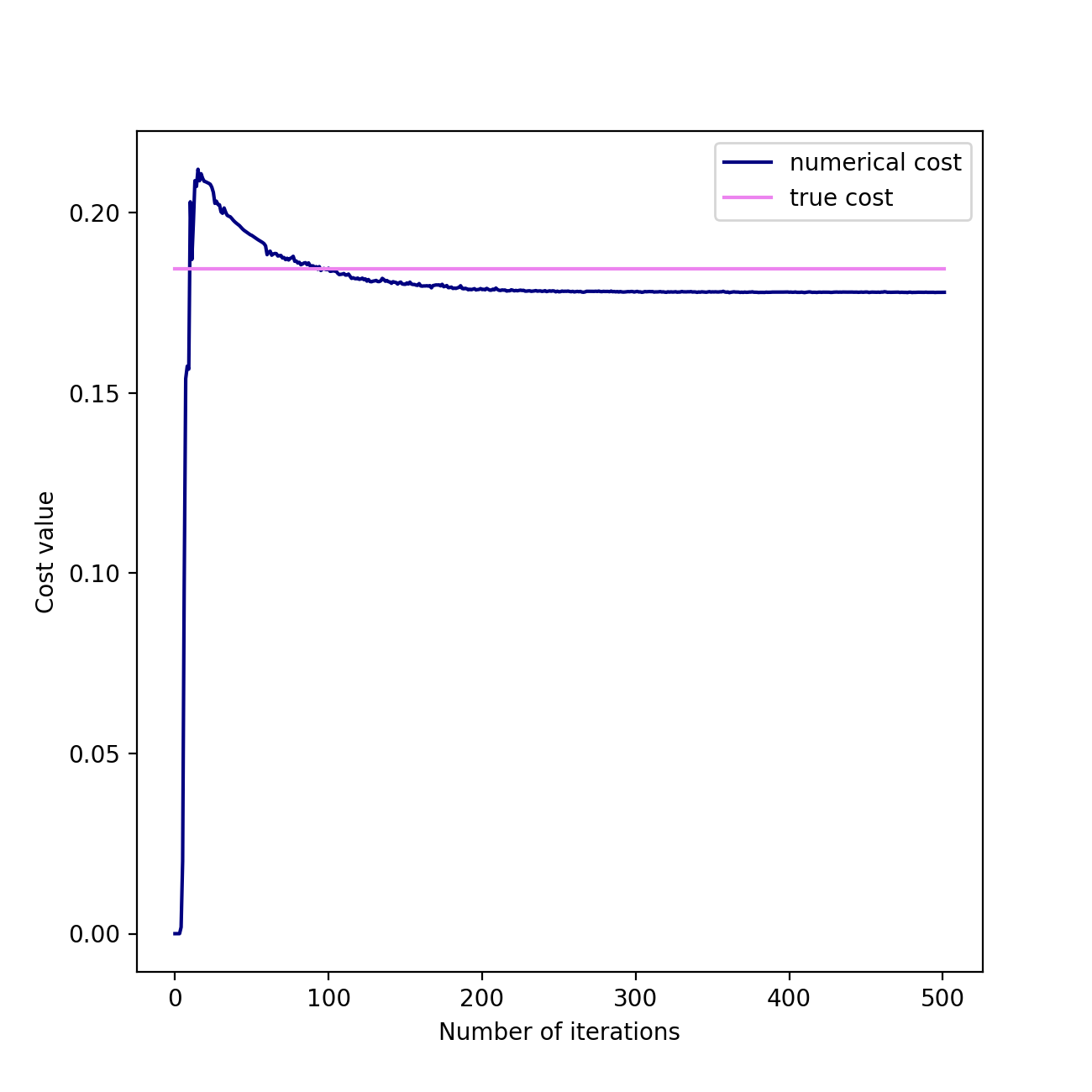}
    \caption{$N = 40$}
  \end{subfigure}
  \caption{Cost in function of the number of iterations in the gradient algorithm \label{fig:costCv1D}}
\end{figure}

\begin{figure}[htp]
  \centering
    \begin{subfigure}[h]{0.32\textwidth}
      \includegraphics[width=\textwidth]{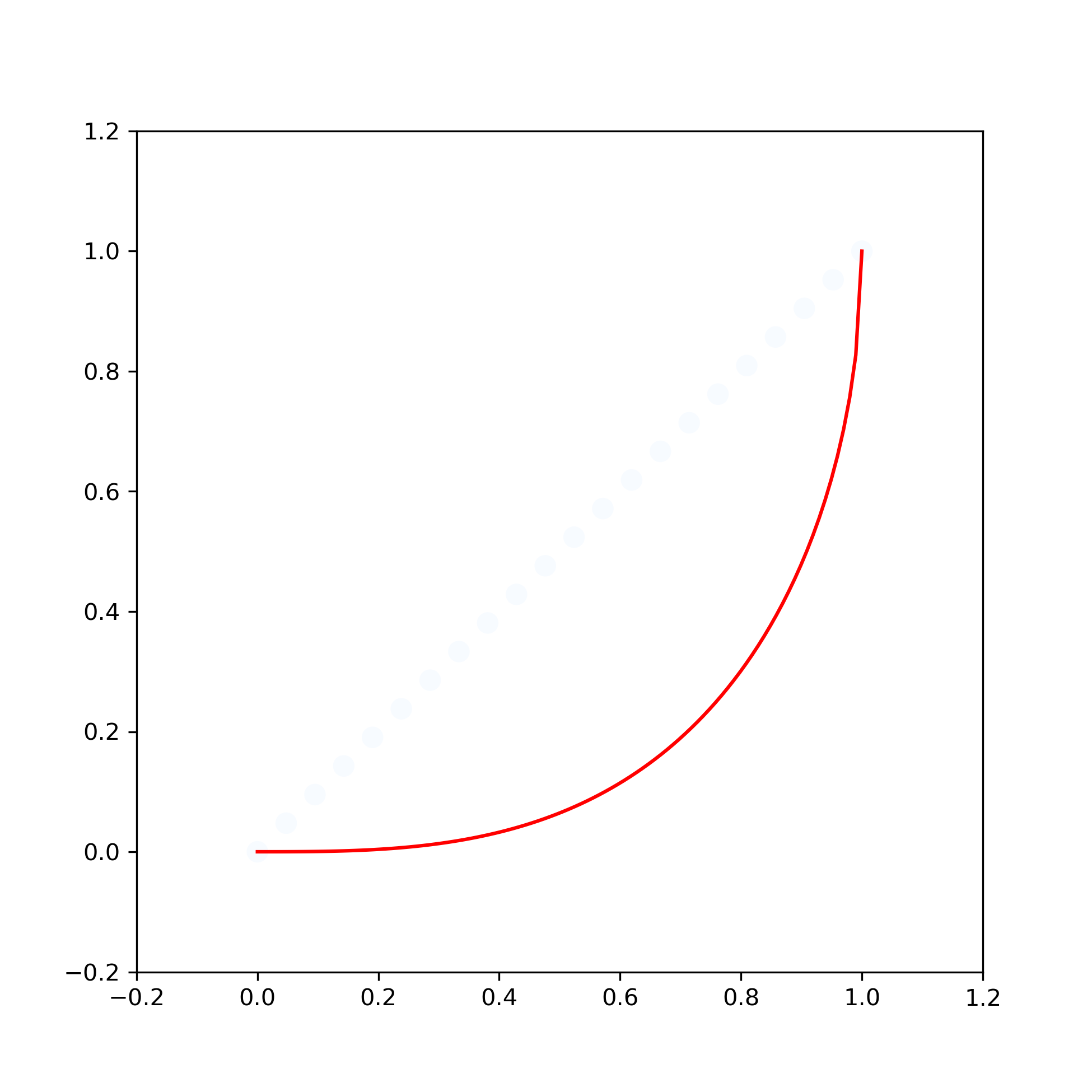}
      \caption{iteration 0}
    \end{subfigure}
    \begin{subfigure}[h]{0.32\textwidth}
      \includegraphics[width=\textwidth]{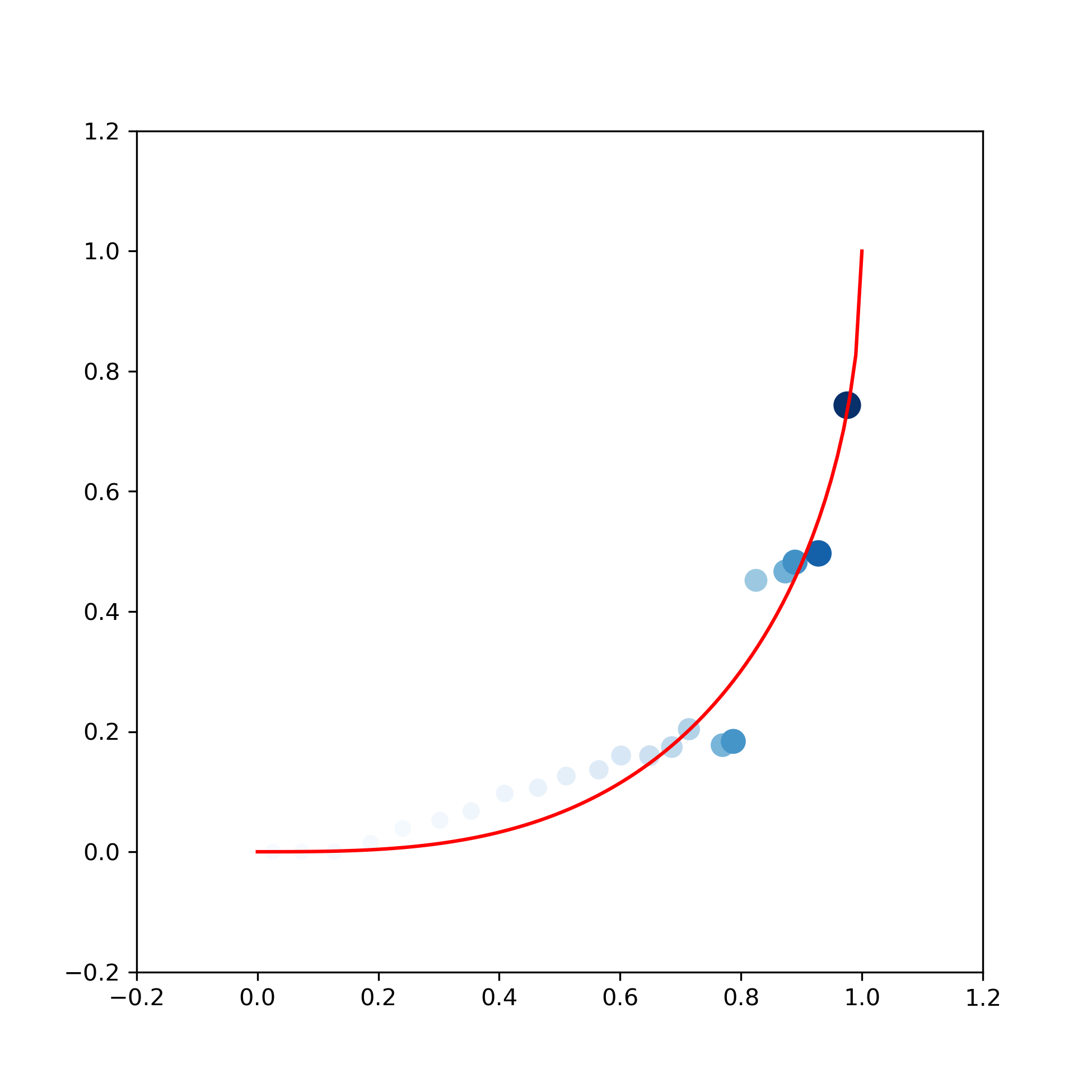}
      \caption{iteration 26}
    \end{subfigure}
    \begin{subfigure}[h]{0.32\textwidth}
      \includegraphics[width=\textwidth]{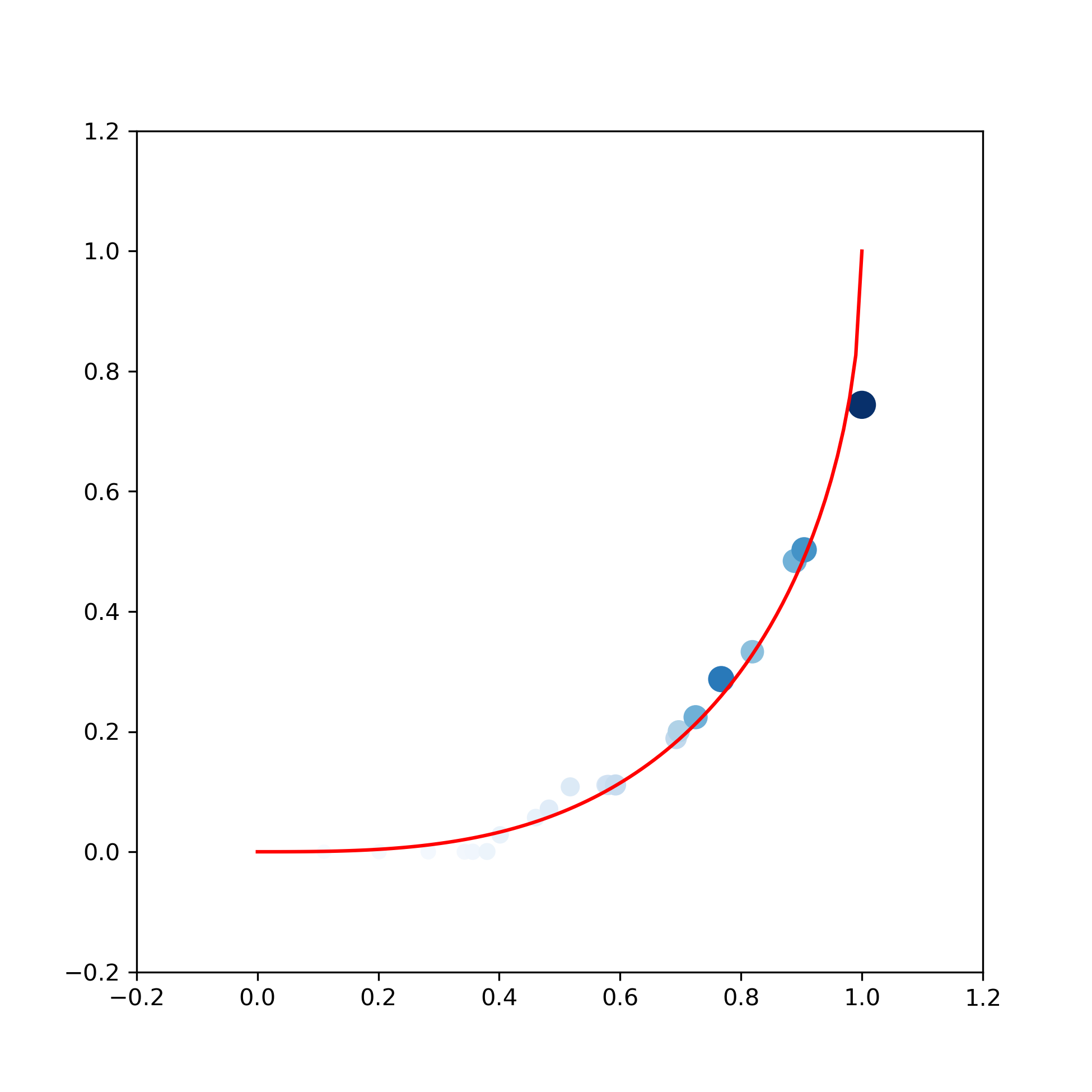}
      \caption{{iteration 301}}
    \end{subfigure}
    \caption{Convergence  with 10 test functions on each set for a quadratic cost \label{fig:configs1DN10}}
\end{figure}

\begin{figure}[htp]
  \centering
    \begin{subfigure}[h]{0.32\textwidth}
      \includegraphics[width=\textwidth]{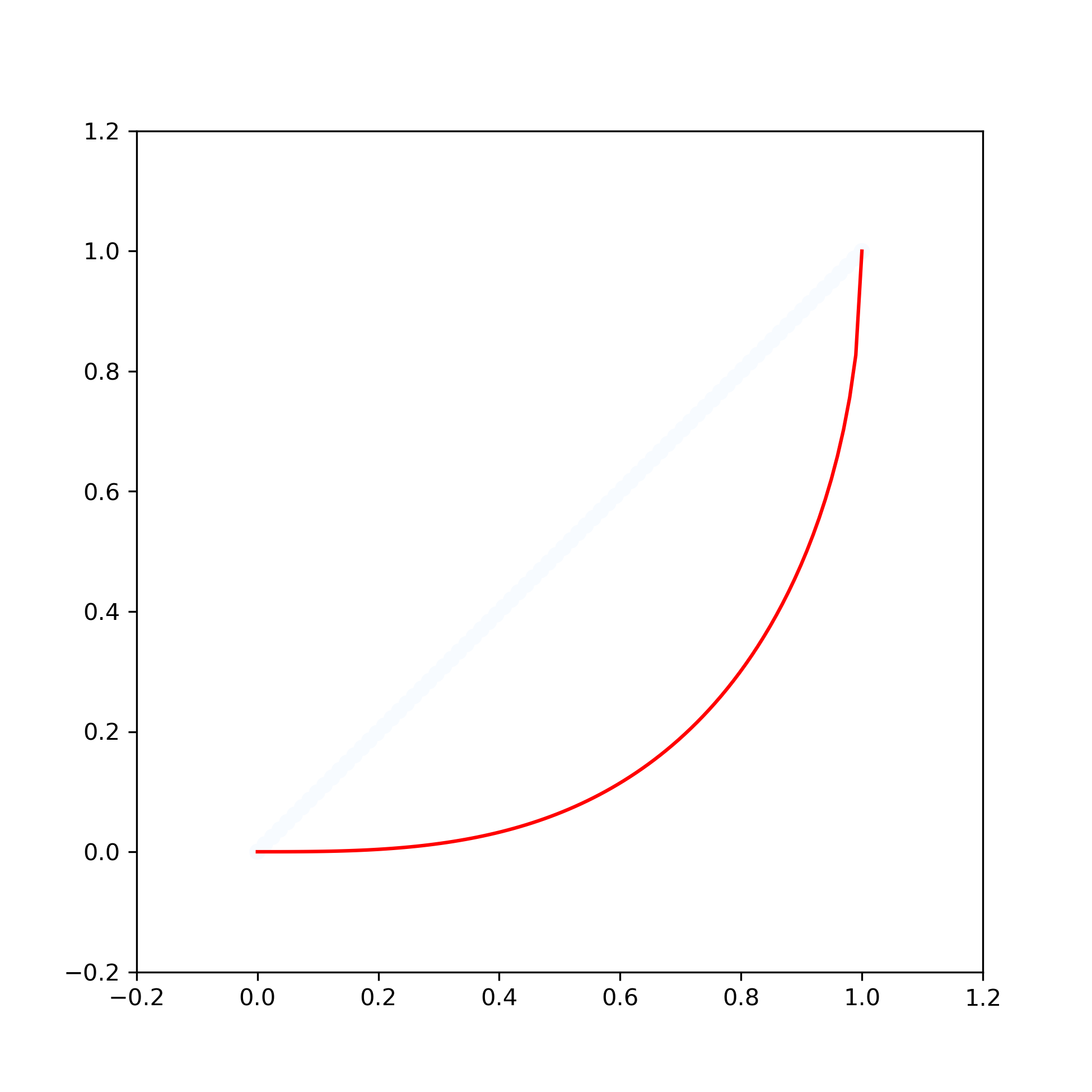}
      \caption{iteration 0}
    \end{subfigure}
    \begin{subfigure}[h]{0.32\textwidth}
      \includegraphics[width=\textwidth]{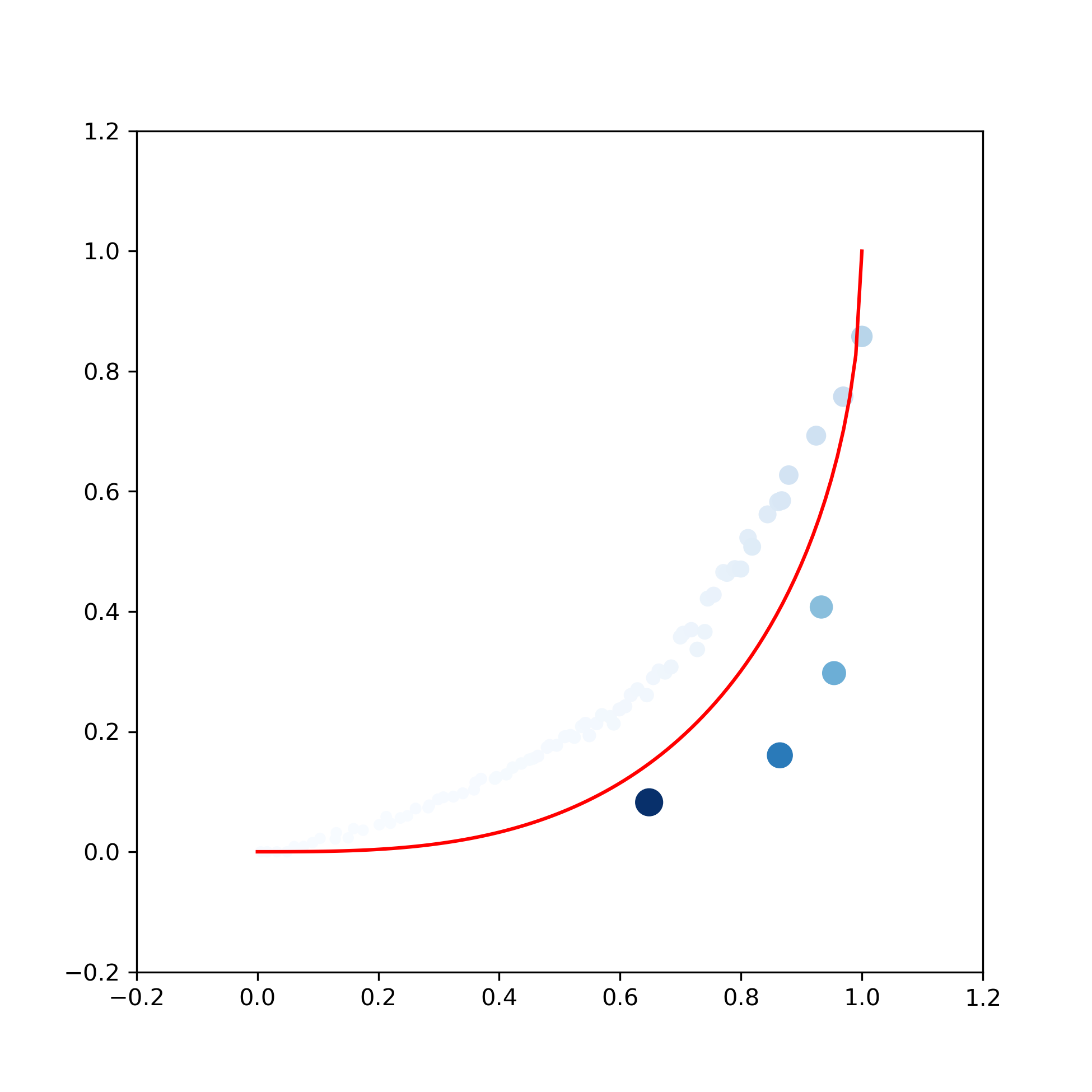}
      \caption{iteration 101}
    \end{subfigure}
    \begin{subfigure}[h]{0.32\textwidth}
      \includegraphics[width=\textwidth]{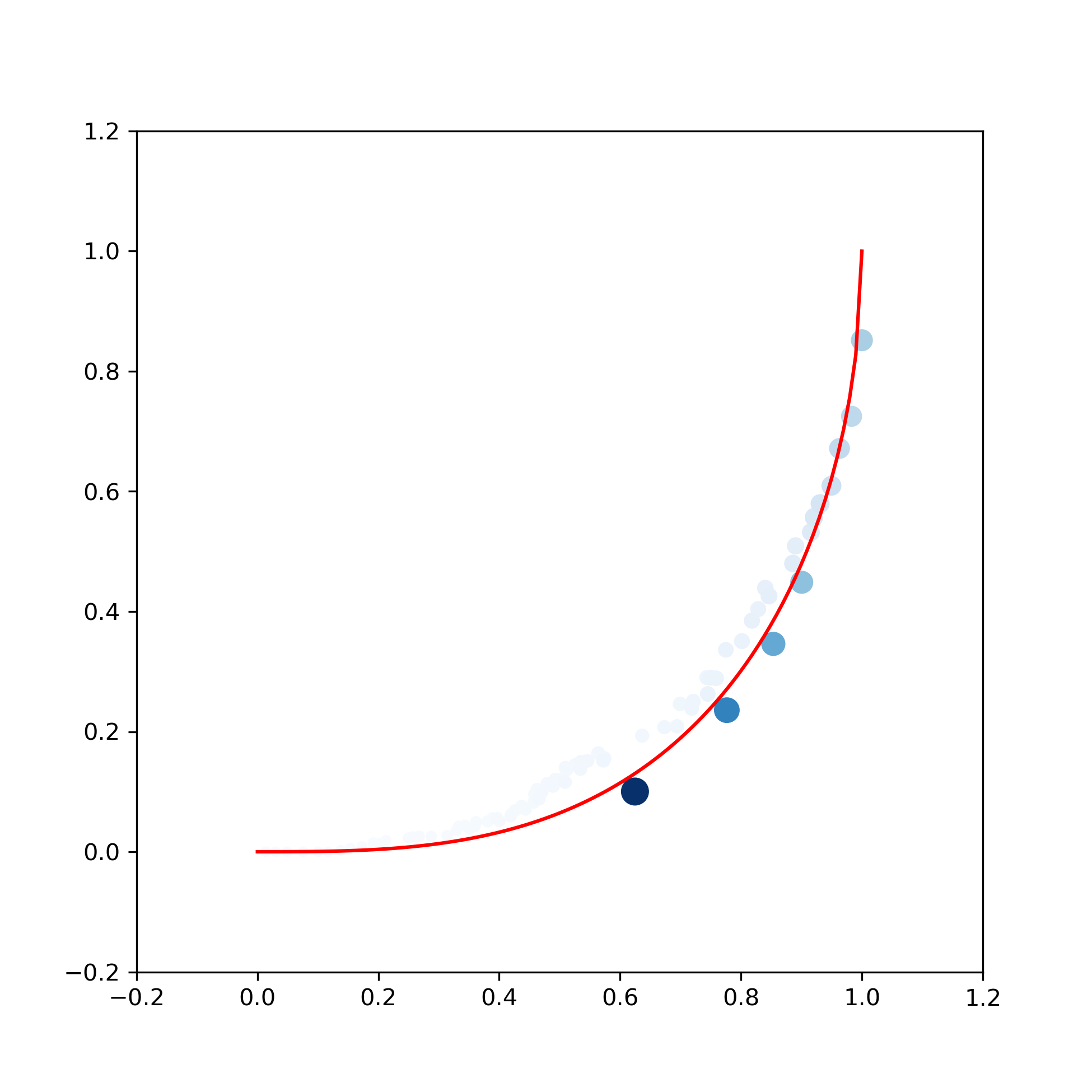}
      \caption{{iteration 501}}
    \end{subfigure}
    \caption{Convergence  with 40 test functions on each set for a quadratic cost \label{fig:configs1DN40}}
\end{figure}

We note on these examples that the particles $(x_k,y_k)$ tend to cluster in some places. This is due to the fact that the cost function is convex and that the test functions are (up to the regularization) locally linear.

\subsubsection{2D numerical example}

We consider two normal marginal laws in~$\R^2$: $\mu \sim \mathcal{N}_2(\mu,\Sigma_{\mu})$ and $\nu \sim \mathcal{N}_2(\nu,\Sigma_{\nu})$, with
\begin{equation}
  m_\mu= \begin{pmatrix}
  0 \\ 0
  \end{pmatrix}, \ \Sigma_\mu=\begin{pmatrix}
  1 & 0 \\ 0 & 1
  \end{pmatrix} , \quad
  m_\nu =\begin{pmatrix}
  1 \\ 1
  \end{pmatrix},\ \Sigma_\nu= \begin{pmatrix}
  1 & 0.7 \\ 0.7 & 1
  \end{pmatrix}, 
\end{equation}
and the quadratic cost function. In this case, it is known that the optimal cost is given by $|m_{\mu}-m_\nu|^2+\text{Tr}(\Sigma_\mu+\Sigma_\nu-2(\Sigma_\mu^{1/2} \Sigma_\nu \Sigma_\mu^{1/2})^{1/2})$ and the optimal transport map is given by $x\mapsto m_\nu+\Sigma_\mu^{-1/2} (\Sigma_\mu^{1/2} \Sigma_\nu \Sigma_\mu^{1/2})^{1/2} \Sigma_\nu^{-1/2}$, see e.g.~\cite{DoLa}. In Figures \ref{fig:configAtCV2D} and \ref{fig:configs2D},
the red density is $\mu$'s one and the blue one $\nu$'s.
We consider regularized piecewise linear test functions on $[-4,4]^2$ obtained by rescaling the functions~\eqref{phi_mn_1},~\eqref{phi_mn_2},~\eqref{phi_mn_3} and~\eqref{phi_mn_4} on $[0,1]^2$.

We represent several iterations of the optimization for $N = 36$ and $1/\eta = 2$ in Figure \ref{fig:configs2D},
where the green arrows represent the transport map computed by the algorithm
from $\mu$ (red) to $\nu$ (blue). The greener the arrow, the more weight it has.

We represent the configuration of particles at convergence on Figure \ref{fig:configAtCV2D}
where each particles consists in a red dot linked to a blue dot. The bigger, the more mass it transports.
The green dots represent the location where the red dot would have been transported if the particle
was on the transport plan. The convergence of the cost is represented in Figure \ref{fig:costCv2D}
where the pink line represents the cost of the Optimal Transport problem we approximate.
\begin{figure}[htp]
  \centering
  \begin{subfigure}[h]{0.48\textwidth}
    \includegraphics[width=\textwidth]{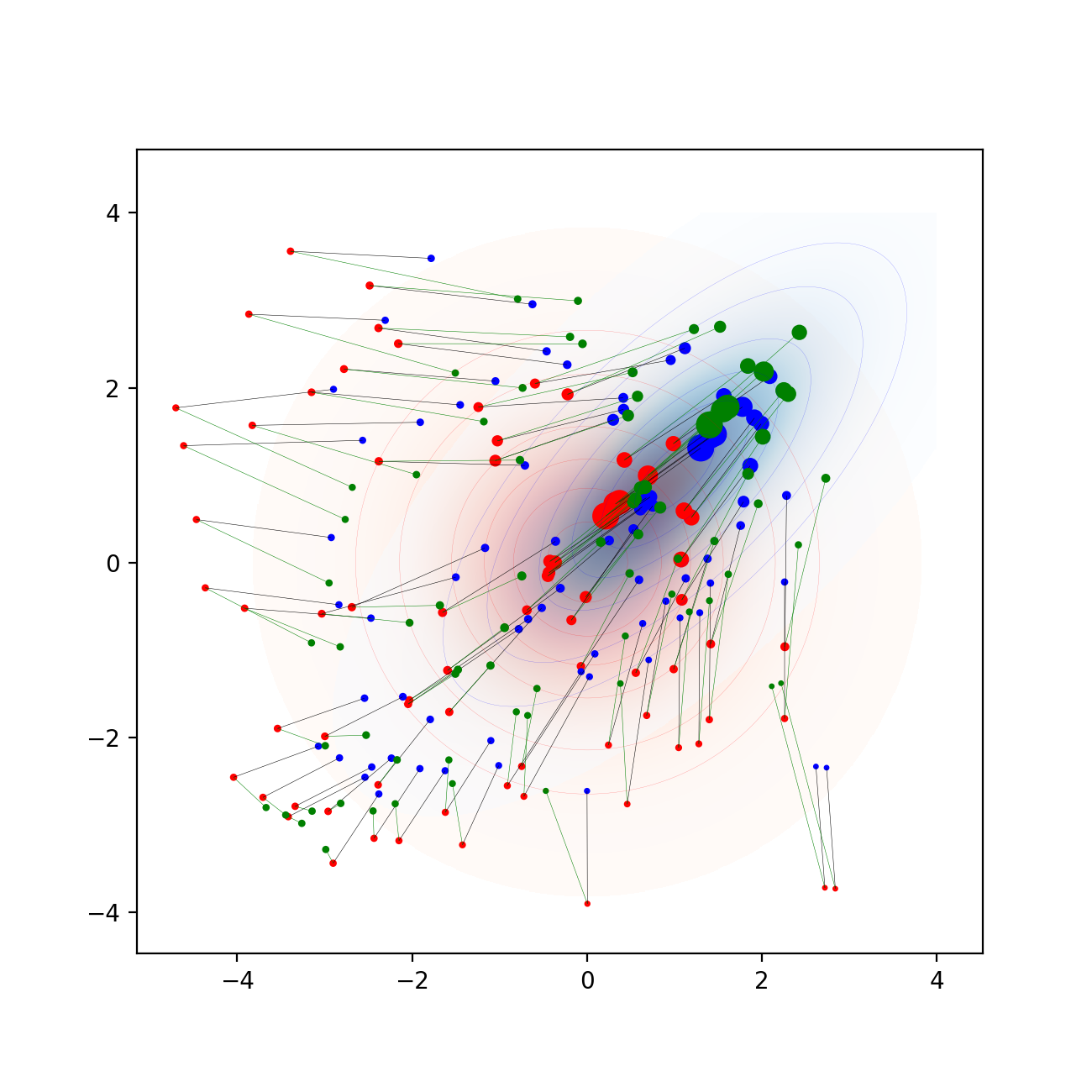}
    \caption{Transport map at convergence \label{fig:configAtCV2D}}
  \end{subfigure}
  \begin{subfigure}[h]{0.48\textwidth}
    \includegraphics[width=\textwidth]{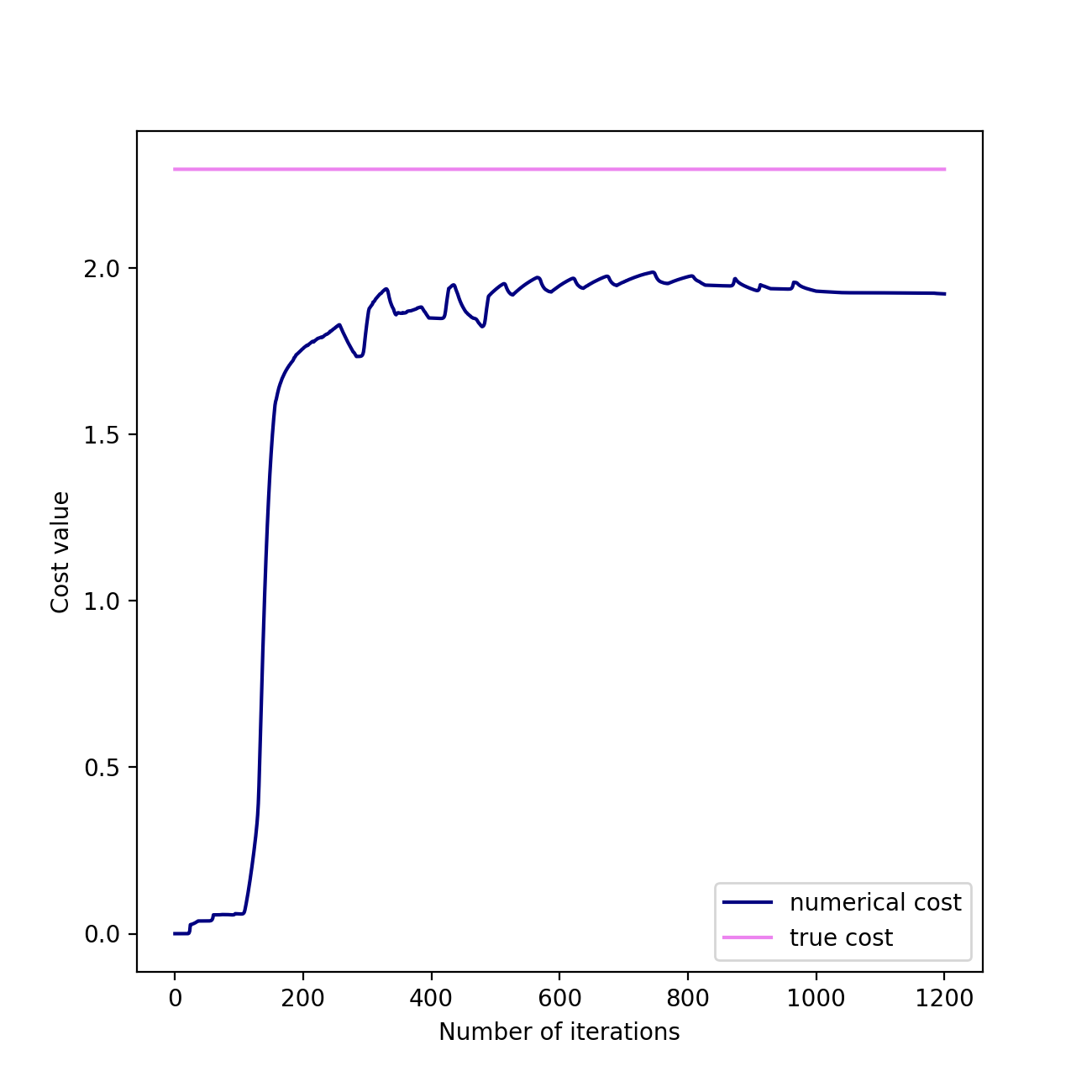}
    \caption{Cost convergence \label{fig:costCv2D}}
  \end{subfigure}
  \caption{Cost convergence and approximation of the transport plan at convergence \label{fig:TransportMapAnsCost2D}}
\end{figure}

\begin{figure}[htp]
  \centering
    \begin{subfigure}[h]{0.32\textwidth}
      \includegraphics[width=\textwidth]{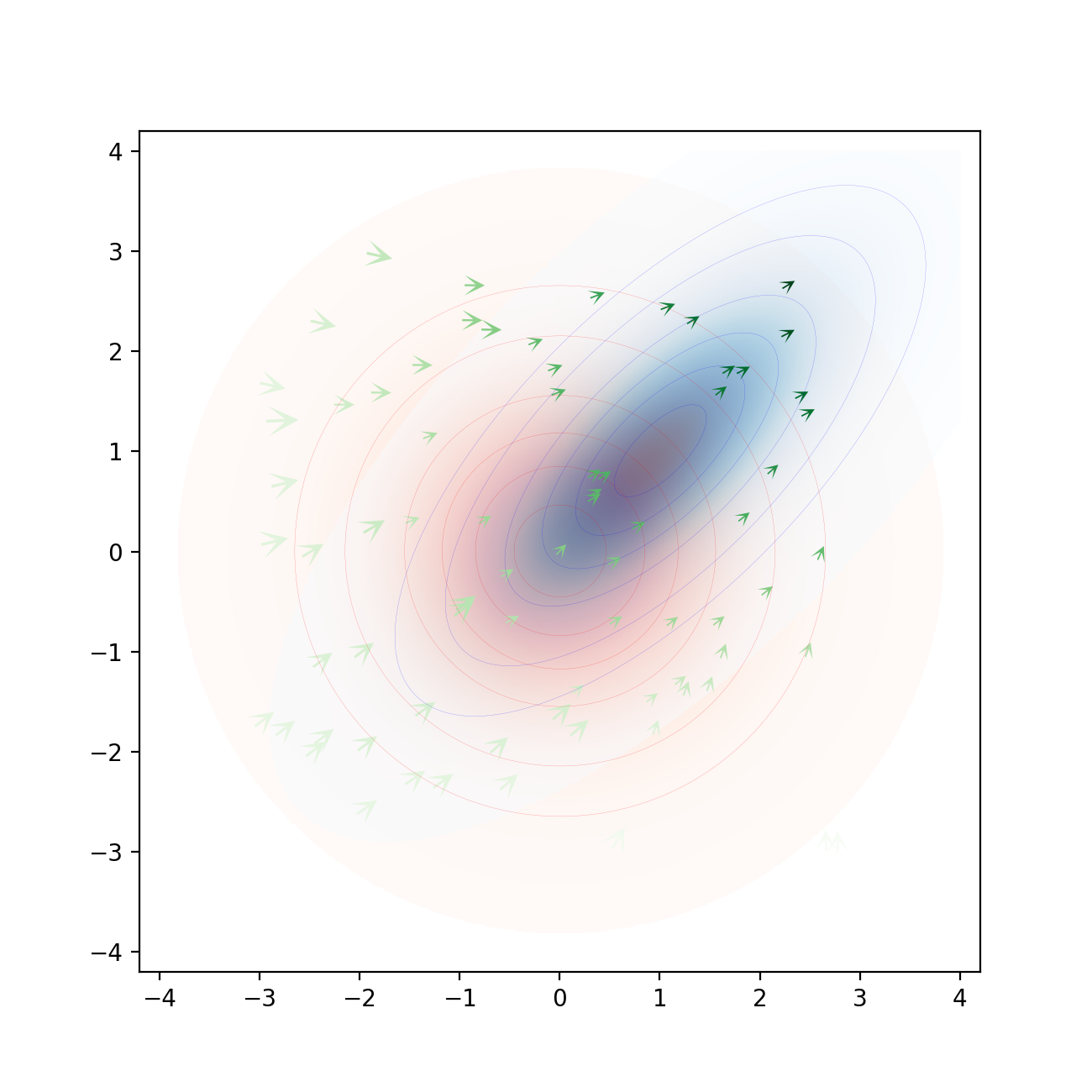}
      \caption{iteration 50}
    \end{subfigure}
    \begin{subfigure}[h]{0.32\textwidth}
      \includegraphics[width=\textwidth]{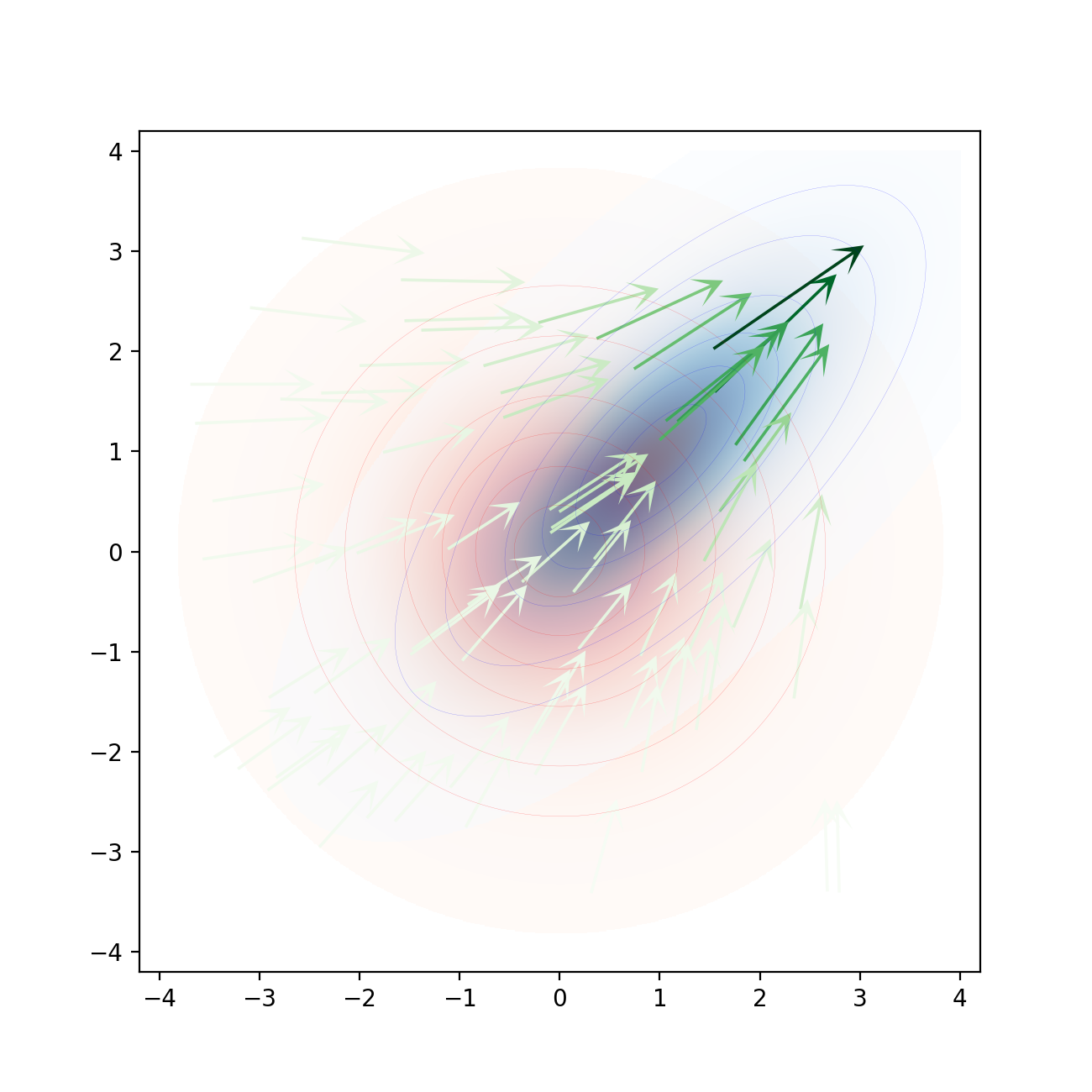}
      \caption{iteration 150}
    \end{subfigure}
    \begin{subfigure}[h]{0.32\textwidth}
      \includegraphics[width=\textwidth]{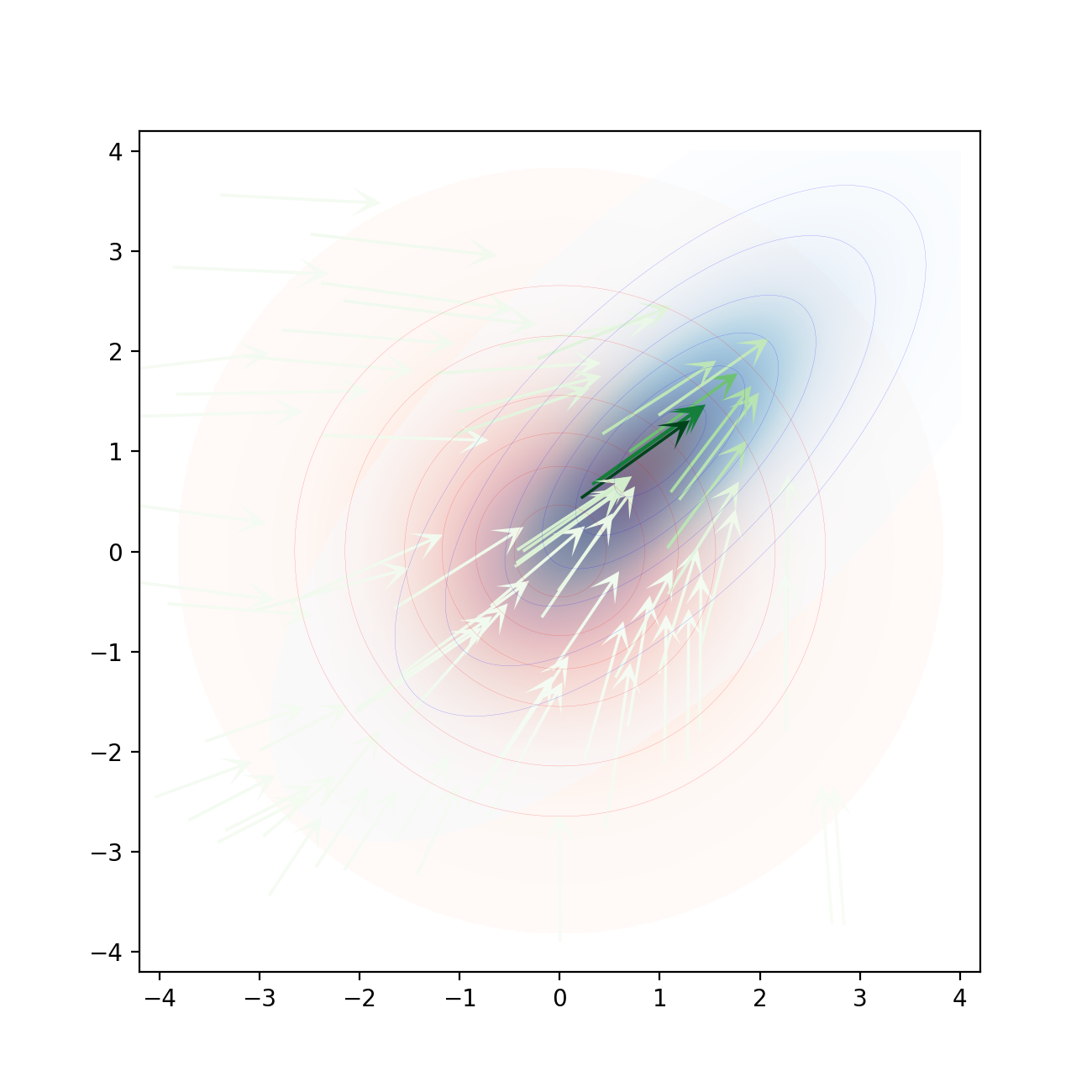}
      \caption{iteration 1200}
    \end{subfigure}
    \caption{Convergence for two 2D marginal laws with 36 test functions on each set for a quadratic cost \label{fig:configs2D}}
\end{figure}

\subsubsection{Martingale Optimal Transport numerical example}

We tested the algorithm for the marginal laws $\mu$ and $\nu$ being respectively the uniform random variables on $[\frac{1}{4}, \frac{3}{4}]$ and $[0,1]$,
with the cost $c(x,y)=|y-x|^3$. Note that $\int |y-x|^2\dd \pi(x,y)=\int |y|^2\dd\nu(y)-\int |x|^2\dd\mu(x)=1/16$ for any martingale coupling~$\pi$. By Jensen's inequality, we have
$ \int |y-x|^3 \dd \pi(x,y)\ge (1/16)^{3/2}=(1/4)^3$ and therefore $\dd \pi(x,y)= \dd \mu(x)(\frac 12 \dd \delta_{x+1/4}+\frac 12 \dd \delta_{x-1/4})$ is an optimal martingale coupling and the equality condition in Jensen's inequality shows that this is the unique optimal martingale coupling.

The two lines $y=x+1/4$ and $y=x-1/4$ characterizing the optimal martingale coupling are represented by the
red lines on Figure \ref{fig:configsMartN10}. We have made one minimization with $N=10$ and $1/\eta = 60$,
and $N'=10$ continuous piecewise affine moment constraints for the martingale constraint, see Problem~\eqref{MgINN'}.
The evolution of the configurations through the iterations are represented
in Figure \ref{fig:configsMartN10}.
The darker the particle $(x_k,y_k)$, the larger the value of its weight $p_k$.
The convergence of the numerical cost is illustrated in Figure \ref{fig:costCvMart}, where 
the pink line represents the cost of the exact martingale Optimal Transport problem we approximate.

\begin{figure}[htp]
  \centering
  \includegraphics[width=0.4\textwidth]{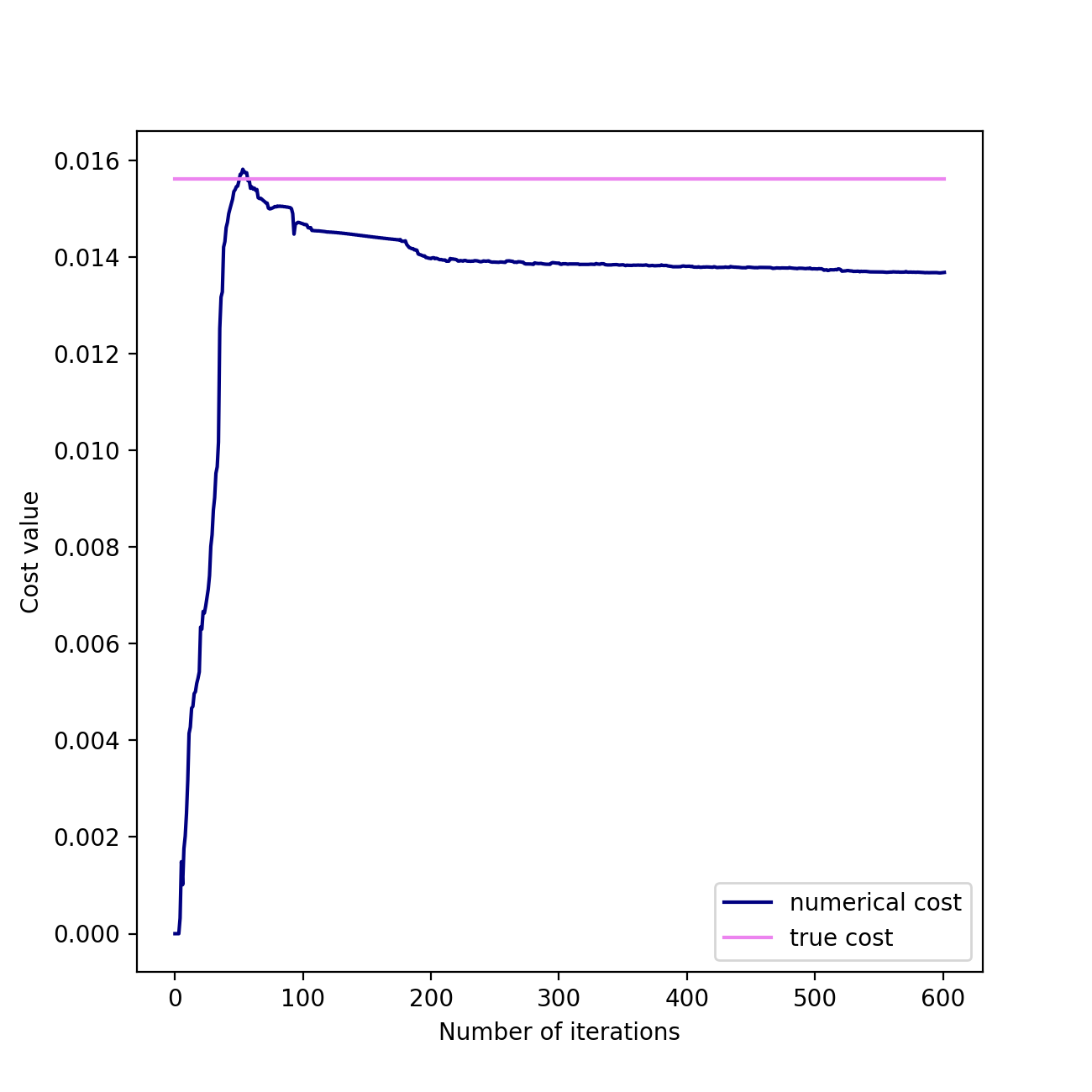}
  \caption{Cost  in function of the number of iterations in the gradient algorithm \label{fig:costCvMart}}
\end{figure}

\begin{figure}[htp]
  \centering
    \begin{subfigure}[h]{0.32\textwidth}
      \includegraphics[width=\textwidth]{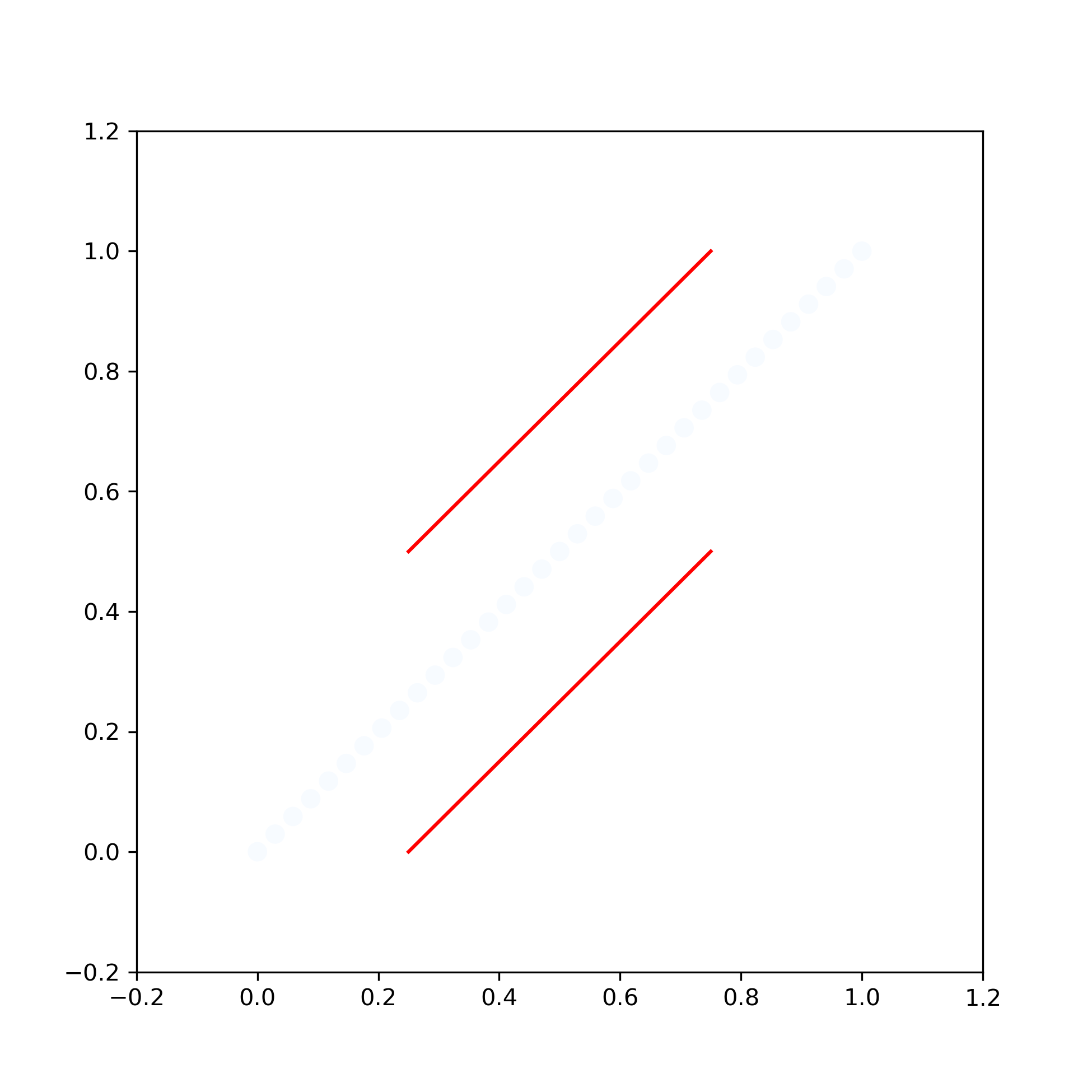}
      \caption{iteration 0}
    \end{subfigure}
    \begin{subfigure}[h]{0.32\textwidth}
      \includegraphics[width=\textwidth]{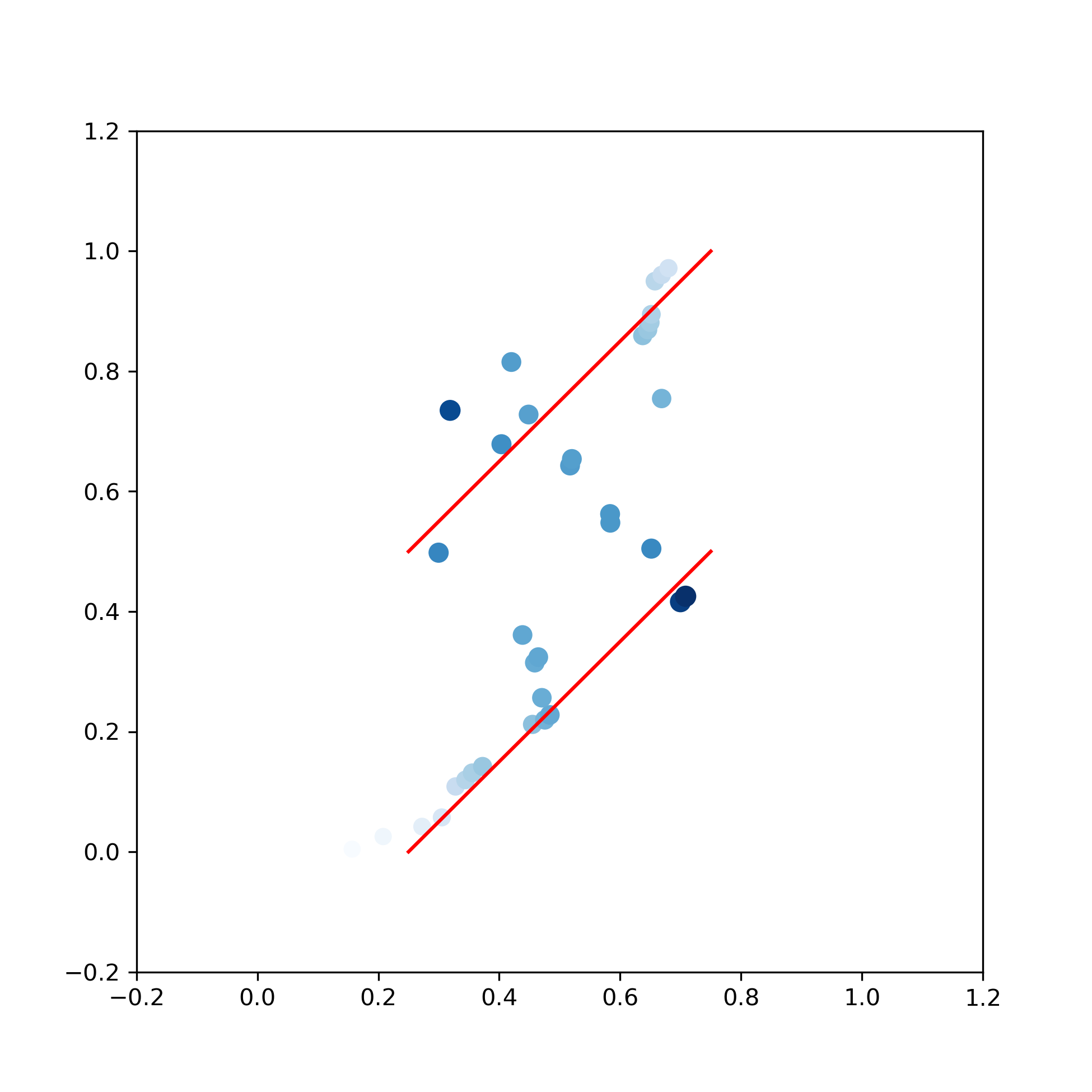}
      \caption{iteration 61}
    \end{subfigure}
    \begin{subfigure}[h]{0.32\textwidth}
      \includegraphics[width=\textwidth]{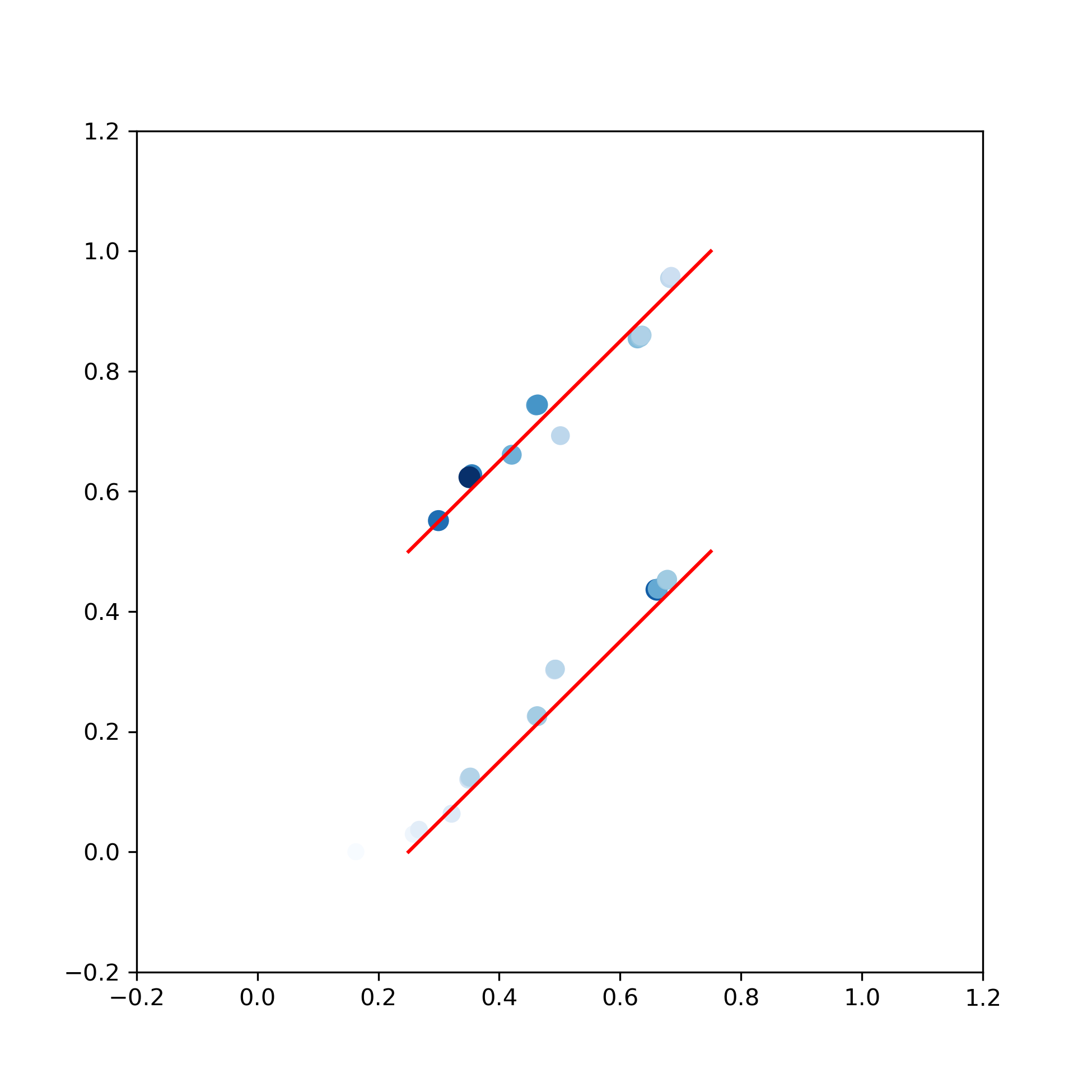}
      \caption{{iteration 601}}
    \end{subfigure}
    \caption{Convergence with 10 test functions on each set for a quadratic cost \label{fig:configsMartN10}}
\end{figure}

\appendix

\section{Technical proofs of Section~\ref{Sec:rateCV}}\label{sec_Appendix}

{\bf Proof of Proposition~\ref{prop:P1ControlRegularMarginal}.}
\begin{proof}
Let us first prove (\ref{eq:W1}). Lemma \ref{lem:P1MomentsEqConsqces} implies that
  \begin{equation}\label{eqn:P1MomentsEqInt}
\forall 1\leq m \leq N, \quad \int_{T_m^N} F_1(x) \dd x = \int_{T_m^N} F_2(x) \dd x
\end{equation}
and
\begin{equation}\label{eqn:P1MomentsEqVrtx}
\forall 1 \leq k \leq N, \quad F_1\left(\frac{k}{N}\right) = F_2\left(\frac{k}{N}\right).
\end{equation}
Then, using a Taylor expansion, as $F_1, F_2\in C^2([0,1])$, it holds that for all $1\leq m \leq N$, all $u\in T_m^N$, and all $l=1,2$,
\begin{equation}\label{eqn:P1MomentsEq1}
\left| F_l(u) - F_l\left(\frac{m}{N}\right) - F_l'\left(\frac{m}{N}\right)
				\left(
					u - \frac{m}{N}
				\right) \right|
				\leq \frac{\|F_l^{\prime \prime} \|_\infty }{2}
					\left(
						u - \frac{m}{N}
					\right)^2.
	\end{equation}
Integrating over $T_m^N$, one gets
	\begin{equation*}
		\left| \int_{T_m^N} F_l(u) \dd u
			- F_l\left(\frac{m}{N}\right) \frac{1}{N} + F_l'\left(\frac{m}{N}\right) \frac{1}{2N^2}
		\right|
		\leq \frac{\|F_l^{\prime \prime} \|_\infty }{6 N^3}.
	\end{equation*}
This implies, using \eqref{eqn:P1MomentsEqInt} and \eqref{eqn:P1MomentsEqVrtx}, that
	\begin{equation}\label{eqn:P1MomentsEq2}
		\left|F_1'\left(\frac{m}{N}\right) - F_2'\left(\frac{m}{N}\right) \right|
		\leq \frac{\|F_1^{\prime \prime} \|_\infty + \|F_2^{\prime \prime} \|_\infty }{3 N}.
	\end{equation}
	Thus, using \eqref{eqn:P1MomentsEq1}, for all $l=1,2$ and $u\in T_m^N$,
  \begin{equation*}
  F_l(u) = F_l\left(\frac{m}{N}\right) + \left( u - \frac{m}{N}\right) F_l'\left(\frac{m}{N}\right) + \varphi_l(u),
  \end{equation*}
  where $|\varphi_l(u)| \le \frac{\|F_l^{\prime \prime} \|_\infty }{2}
    \left(
      u - \frac{m}{N}
    \right)^2$.
 Then, using \eqref{eqn:P1MomentsEq2}, one gets that for all $u\in T_m^N$,
\begin{multline}\label{eqn:P1CDFMajExpansion}
		\left|
			F_1(u) - F_2(u)
		\right|
		\leq
		\frac{\|F_1^{\prime \prime} \|_\infty + \|F_2^{\prime \prime} \|_\infty}{3 N}
		\left(
		 \frac{m}{N}-u
		\right)
	 + \frac{\|F_1^{\prime \prime} \|_\infty + \|F_2^{\prime \prime} \|_\infty}{2}
			\left(
				u - \frac{m}{N}
			\right)^2.
	\end{multline}
Integrating over $T_m^N$ yields that
	\begin{equation}\label{eqn:P1W1ControlIntervalIneq}
		\int_{T_m^N} \left|
			F_1(u) - F_2(u)
		\right| \dd u
		\leq
			\frac{\|F_1^{\prime \prime} \|_\infty + \|F_2^{\prime \prime} \|_\infty}{3 N^3}.
	\end{equation}
Using the fact that
	\begin{equation*}
W_1(\mu_1, \mu_2)
    = \int_{[0,1]} \left| F_1(u) - F_2(u) \right| \dd u
		= \sum_{m = 1}^N \int_{T_m^N} \left| F_1(u) - F_2(u) \right| \dd u,
	\end{equation*}
we obtain that
\begin{equation*}
W_1(\mu_1, \mu_2)\leq \frac{\|F_1^{\prime \prime} \|_\infty + \|F_2^{\prime \prime} \|_\infty}{3 N^2}.
\end{equation*}

\medskip

Let us now prove (\ref{eq:Wp}).	The main result needed is the expression of the Wasserstein distance in
	term of the cumulative distribution functions (cdf) and not their inverse
	(see \cite{Jourdain2013} Lemma B.3), which holds for $p>1$,
	\begin{equation}\label{eqn:WPInTermOfCDF}
		W_p^p(F,G) =
			p (p-1) \int_{\Reel^2} \mathbf{1}_{\{ x < y \}}
			\left(
				\left[
					G(x) - F(y)
				\right]^+
				+ 
				\left[
					F(x) - G(y)
				\right]^+
			\right)
			(y - x)^{p-2}
			\dd x \dd y
	\end{equation}
	because the reasoning of the beginning of this proof introduced a control on
	the norm between the cdf of the marginal law and the cdf of a marginal law
	satisfying the same moments.

	Then, one can proceed with the following induction.
	Suppose that we know for $p \in \Natural^*$ that
	\begin{equation}\label{eqn:P1MomentsIndcHyp}
		W^p_p(\mu, \tilde{\mu})
		\leq \left( \frac{\|F^{\prime \prime} \|_\infty + \|\tilde{F}^{\prime \prime} \|_\infty}{3N^2}\right)^p
			\left(\frac{5}{2} \left(\frac{1}{m_\mu} + \frac{1}{m_{\tilde{\mu}}}\right) \right)^{p-1}
      p!,
	\end{equation}
which holds for $p=1$.	Then,
	\begin{equation*}
		\begin{split}
			&W_{p+1}^{p+1} (\mu, \tilde{\mu}) \\
			&= (p+1) p \int_{\Reel^2} \mathbf{1}_{\{ x < y \}}
			\left(
				\left[
					F(x) - \tilde{F}(y)
				\right]^+
				+
				\left[
					\tilde{F}(x) - F(y)
				\right]^+
			\right)
			(y - x)^{p-1}
			\dd x \dd y \\
			&= (p+1) p \int_0^1  \left( \int_x^1
				\left(
				\left[
					\tilde{F}(x) - F(y)
				\right]^+
				+
				\left[
					F(x) - \tilde{F}(y)
				\right]^+
				\right)
				(y-x)^{p-1}
				\dd y \right) \dd x \\
			&= (p+1) p \int_0^1
				\left(
					\int_x^1
					\left[
						F(x) - \tilde{F}(y)
					\right]^+
					(y-x)^{p-1}
					\dd y
					\right. \\
					&\quad + \left.
					\int_x^1
					\left[
						\tilde{F}(x) - F(y)
					\right]^+
					(y-x)^{p-1}
					\dd y
				\right) \dd x.
		\end{split}
	\end{equation*}
	Let us treat the first term of the sum, as the second one can be treated symmetrically.
	If $F(x) \geq \tilde{F}(x)$, we can define $y_x = \tilde{F}^{-1}\left(F(x)\right)$ and because $\tilde{F}:[0,1]\rightarrow [\tilde{F}(0),1]$ is continuous increasing, and we have
	\begin{equation*}
		\begin{split}
				\int_x^1
				\left[
					F(x) - \tilde{F}(y)
				\right]^+
				(y-x)^{p-1}
				\dd y
			&= \int_x^{y_x}
			\left(
				F(x) - \tilde{F}(y)
			\right)
			(y-x)^{p-1}
			\dd y \leq \frac{1}{p} \left| F(x) - \tilde{F}(x) \right| (y_x - x)^p.
		\end{split}
	\end{equation*}
Thus, by using~\eqref{eqn:P1MomentsEq1}, we get
	\begin{align*}		
&		\int_0^1
		\int_x^1 
		\left[
			F(x) - \tilde{F}(y)
		\right]^+
		(y-x)^{p-1}
		\dd y \dd x \\
		&\qquad \leq	\frac{1}{p} \int_0^1 \mathbf{1}_{\{F(x) \geq \tilde{F}(x)\}}
			\left| F(x) - \tilde{F}(x) \right| (y_x - x)^p \dd x \\
		&\qquad \leq \frac{1}{p}
			\sum_{m = 1}^N \int_{T^N_m} \mathbf{1}_{\{F(x) \geq \tilde{F}(x)\}}
				\left| F(x) - \tilde{F}(x) \right| (y_x - x)^p \dd x \\
		&\qquad \leq \frac{1}{p}
			\sum_{m = 1}^N \int_{T^N_m} \left(
				\frac{\|F^{\prime \prime} \|_\infty + \|\tilde{F}^{\prime \prime} \|_\infty}{3N}
				\left( \frac{m}{N}-x  \right)
				+ \frac{\|F^{\prime \prime} \|_\infty + \|\tilde{F}^{\prime \prime} \|_\infty}{2}
					\left( x - \frac{m}{N} \right)^2
				\right) \mathbf{1}_{\{F(x) \geq \tilde{F}(x)\}} (y_x - x)^p \dd x 
    \end{align*}
    \begin{align*}
    &\qquad \leq \frac{5}{6p} \frac{\|F^{\prime \prime} \|_\infty + \|\tilde{F}^{\prime \prime} \|_\infty}{N^2}
  	 \int_0^1 \mathbf{1}_{\{F(x) \geq \tilde{F}(x)\}} ( \tilde{F}^{-1}\left(F(x)\right) -  {F}^{-1}\left(F(x)\right) )^p \dd x \\
		&\qquad \leq \frac{5}{6p} \frac{\|F^{\prime \prime} \|_\infty + \|\tilde{F}^{\prime \prime} \|_\infty}{N^2}
			\int_{F(0)}^1 \mathbf{1}_{\{u \geq  \tilde{F}(F^{-1}(u))\}}
				\left( \tilde{F}^{-1}(u) - F^{-1}(u) \right)^p \frac{\dd u}{F'(F^{-1}(u))} \\
		&\qquad \leq \frac{5}{6p} \frac{\|F^{\prime \prime} \|_\infty + \|\tilde{F}^{\prime \prime} \|_\infty}{N^2}
			\frac{1}{\min_{u \in [0,1]}F'(F^{-1}(u))}
			\int_0^1 \left| \tilde{F}^{-1}(u) - F^{-1}(u) \right|^p \dd u,
	\end{align*}
	where we used the formula bounding the difference between the cdf \eqref{eqn:P1CDFMajExpansion}.

	Therefore,
	as $m_\mu > 0$ and $m_{\tilde{\mu}} > 0$,
	and using the symmetry of the the formula \eqref{eqn:WPInTermOfCDF}, one gets
	\begin{equation}
		W_{p+1}^{p+1}(\mu, \tilde{\mu})
		\leq \frac{5(p+1)}{2} \frac{\|F^{\prime \prime} \|_\infty + \|\tilde{F}^{\prime \prime} \|_\infty}{3N^2}
			\left(\frac{1}{m_\mu} + \frac{1}{m_{\tilde{\mu}}} \right) W_p^p(\mu, \tilde{\mu}).
	\end{equation}
	Hence, using the induction hypothesis \eqref{eqn:P1MomentsIndcHyp}, we obtain that \eqref{eqn:P1MomentsIndcHyp} holds for $p+1$, which gives the claim.
\end{proof}

\begin{lemma}\label{lem:cvgP1_W2}
  Let $\mu\in \cP([0,1])$ and $F_\mu$ its cumulative distribution function. Let $N\in \N^*$. Then, for any $1\le m \le N$, we define $x^N_m\in T^N_m$ by
\begin{align*}
  &x^N_m=\frac{m}{N} \text{ if } F_\mu\left( \frac{m}N \right)=F_\mu\left( \frac{m-1}N \right)\\
  &x^N_m=\frac{\int_{T^N_m}F_\mu +\frac{m-1}{N}F_\mu\left( \frac{m-1}N \right) -\frac{m}{N}F_\mu\left( \frac{m}N \right) }{F_\mu\left( \frac{m}N \right)-F_\mu\left( \frac{m-1}N \right)} \text{ if }F_\mu\left( \frac{m}N \right)>F_\mu\left( \frac{m-1}N \right),
\end{align*}
and $\hat{\mu}^N=F_\mu(0)\delta_0+\sum_{m=1}^N (F_\mu\left( \frac{m}N \right)-F_\mu\left( \frac{m-1}N \right))\delta_{x_m}$. Then, we have for all $1\le m \le N$,
$$ F_{\hat{\mu}^N}\left( \frac{m}N \right)=F_\mu\left( \frac{m}N \right), \ \int_{T^N_m}F_{\hat{\mu}^N}=\int_{T^N_m}F_\mu,\ \forall x \in T^N_m, \int_{\frac{m-1}N}^xF_\mu \ge  \int_{\frac{m-1}N}^xF_{\hat{\mu}^N}.$$
Besides, if $\mu(\dd x)=\rho_{\mu}(x)\dd x$ with a $\rho_{\mu}\in L^\infty([0,1],\dd x;\R_+)$ and $\tilde{\mu}\in\cP([0,1])$ is such that $F_{\tilde{\mu}}\left( \frac{m}N \right)=F_\mu\left( \frac{m}N \right)$ and $\int_{T^N_m}F_{\tilde{\mu}}=\int_{T^N_m}F_\mu$, we have
\begin{equation}
  \int_{T^N_m} \left(\int_{\frac{m-1}N}^xF_\mu\right) \dd x \le \int_{T^N_m} \left(\int_{\frac{m-1}N}^xF_{\tilde{\mu}}\right) \dd x + \frac{\|\rho_\mu\|_\infty}{6N^3}
\end{equation}

\end{lemma}
\begin{proof}
  If $F_\mu\left( \frac{m}N \right)>F_\mu\left( \frac{m-1}N \right)$, we have $$\frac 1N F_\mu\left( \frac{m-1}N \right) \le \int_{T^N_m}F_\mu  < \frac 1N F_\mu\left( \frac{m}N \right)$$ since $F_{\mu}$ is non-decreasing and right-continuous. Therefore, there is a unique $x^N_m  \in T^N_m$ such that
  $$\left(x^N_m-\frac{m-1}{N} \right)F_\mu\left( \frac{m-1}N \right)+\left(\frac{m}{N} - x^N_m\right) F_\mu\left( \frac{m}N \right)=\int_{T^N_m}F_\mu, $$
  which is precisely the definition of $x^N_m$. By construction, we have  $F_{\hat{\mu}^N}\left( \frac{m}N \right)=F_\mu\left( \frac{m}N \right)$ and the previous equation gives
  $$\int_{T^N_m}F_{\hat{\mu}^N}=\int_{\frac{m-1}{N} }^{x^N_m}F_\mu\left( \frac{m-1}N \right)dx + \int_{x^N_m}^{\frac{m}{N} }F_\mu\left( \frac{m}N \right)dx =\int_{T^N_m}F_\mu$$
  when $F_\mu\left( \frac{m}N \right)>F_\mu\left( \frac{m-1}N \right)$ (this identity is obvious if $F_\mu\left( \frac{m}N \right)=F_\mu\left( \frac{m-1}N \right)$).
  Last, since for $x \in T^N_m$, $F_\mu\left( \frac{m-1}N \right)\le F_\mu(x)\le F_\mu\left( \frac{m}N \right)$, we get that $x\mapsto \int_{\frac{m-1}N}^x(F_\mu-F_{\hat{\mu}^N})$ is non-decreasing on $[\frac{m-1}N,x^N_m]$, non-increasing on $[x^N_m,\frac{m}N]$ and vanishes for $x\in \{\frac{m-1}N,\frac{m}N \}$: it is therefore non-negative on $T^N_m$.

  Now, let us assume that $\mu$ has a bounded density probability function $\rho_\mu$. We have
  \begin{align*}
    x\in \left[\frac{m-1}N,x^N_m\right], \int_{\frac{m-1}N}^x(F_\mu-F_{\hat{\mu}^N})&=\int_{\frac{m-1}N}^x\int_{\frac{m-1}N}^z \rho_{\mu}(u)\dd u \dd z \le \frac{\|\rho_\mu\|_\infty}{2}(x-\frac{m-1}N)^2,\\
    x\in \left[x^N_m,\frac{m}N\right], \int_{\frac{m-1}N}^x(F_\mu-F_{\hat{\mu}^N})&=-\int_x^{\frac{m}N}(F_\mu-F_{\hat{\mu}^N}) \\
    &=\int_x^{\frac{m}N}\int_z^{\frac{m}N}\rho_{\mu}(u)\dd u \dd z \le \frac{\|\rho_\mu\|_\infty}{2}(\frac{m}N-x)^2,
  \end{align*}
  and therefore
  \begin{equation}\label{majo_diff_Fmu} \int_{T^N_m} \left(\int_{\frac{m-1}N}^x(F_\mu-F_{\hat{\mu}^N})\right) \dd x \le \frac{\|\rho_\mu\|_\infty}{6}\left[\left(x^N_m-\frac{m-1}{N} \right)^3+\left(\frac{m}{N}-x^N_m \right)^3\right]\le \frac{\|\rho_\mu\|_\infty}{6N^3}.
  \end{equation}
  Now, we observe that we either have $ \int_{T^N_m} \left(\int_{\frac{m-1}N}^xF_\mu\right) \dd x \le \int_{T^N_m} \left(\int_{\frac{m-1}N}^xF_{\tilde{\mu}}\right) \dd x$ or  $\int_{T^N_m} \left(\int_{\frac{m-1}N}^xF_\mu\right) \dd x \ge \int_{T^N_m} \left(\int_{\frac{m-1}N}^xF_{\tilde{\mu}}\right) \dd x \ge  \int_{T^N_m} \left(\int_{\frac{m-1}N}^xF_{\hat{\mu}^N}\right) \dd x $. In the first case, the claim is obvious. In the second one, we then have
  $$\int_{T^N_m} \left(\int_{\frac{m-1}N}^xF_\mu\right) \dd x \le \int_{T^N_m} \left(\int_{\frac{m-1}N}^xF_{\tilde{\mu}}\right)+\int_{T^N_m} \left(\int_{\frac{m-1}N}^x(F_\mu-F_{\hat{\mu}^N})\right),$$
and we get the result using~\eqref{majo_diff_Fmu}.
\end{proof}

\section*{Acknowledgements}

The Labex B\'ezout is acknowledged for funding the PhD thesis of Rafa\"el Coyaud. Aur\'elien Alfonsi benefited from the support of the ``Chaire Risques Financiers'', Fondation du Risque. We are very grateful to Luca Nenna for stimulating discussions.

\bibliographystyle{plain}
\bibliography{bibli}
\vspace{1cm}
{\footnotesize

\begin{tabular}{rl}
A. Alfonsi & \textsc{Universit\'e Paris-Est, CERMICS (ENPC), INRIA,} \\
 &  F-77455 Marne-la-Vall\'ee, France\\
 &  \textsl{E-mail address}:  {\texttt{aurelien.alfonsi@enpc.fr}} \\
 \\
R. Coyaud & \textsc{Universit\'e Paris-Est, CERMICS (ENPC), INRIA,} \\
 &  F-77455 Marne-la-Vall\'ee, France\\
 &  \textsl{E-mail address}:  {\texttt{rafael.coyaud@enpc.fr}} \\
 \\
V. Ehrlacher & \textsc{Universit\'e Paris-Est, CERMICS (ENPC), INRIA,} \\
 &  F-77455 Marne-la-Vall\'ee, France\\
 &  \textsl{E-mail address}:  {\texttt{virginie.ehrlacher@enpc.fr}} \\
  \\
D. Lombardi & \textsc{INRIA Paris and Sorbonne Universit\'es, UPMC Univ Paris 6, UMR 7598 LJLL} \\
 &  F-75589 Paris Cedex 12, France\\
 &  \textsl{E-mail address}:  {\texttt{damiano.lombardi@inria.fr}} \\
 \\
\end{tabular}

}
\end{document}